\documentclass[runningheads]{llncs}

\newcommand{\p}{\mathfrak{p}}
\newcommand{\m}{\mathfrak{m}}
\newcommand{\R}{\mathbb{R}}

\usepackage{latexsym, amssymb, amsfonts,amsmath}
\usepackage{graphicx}
\usepackage[all,knot,arc,import,poly]{xy}
\usepackage{hyperref, cleverref}

\begin{document}
\title{Von Neumann Regular $\mathcal{C}^{\infty}-$Rings and Applications}
\author{Jean Cerqueira Berni\inst{1} \and
Hugo Luiz Mariano\inst{2}}
\authorrunning{J. C. Berni and H. L. Mariano}
\institute{Institute of Mathematics and Statistics, University of S\~{a}o Paulo,
Rua do Mat\~{a}o, 1010, S\~{a}o Paulo - SP, Brazil. \email{jeancb@ime.usp.br} \and
Institute of Mathematics and Statistics, University of S\~{a}o Paulo,
Rua do Mat\~{a}o, 1010, S\~{a}o Paulo - SP, Brazil, \email{hugomar@ime.usp.br}}
\maketitle              % typeset the header of the contribution
%
%
%\begin{document}
%\maketitle
\begin{abstract}
	In this paper we present the notion of a von Neumann regular $\mathcal{C}^{\infty}-$ring, we prove some results about them  and we describe some of their properties. We prove, using two different methods, that the category of von Neumann regular $\mathcal{C}^{\infty}-$rings is a reflective  subcategory of $\mathcal{C}^{\infty}{\rm \bf Rng}$. We prove that every homomorphism between Boolean algebras can be represented by a $\mathcal{C}^{\infty}-$ring homomorphism between von Neumann regular $\mathcal{C}^{\infty}-$rings.
\end{abstract}

\section*{Introduction}

We assume that the reader is familiar with the notion of $\mathcal{C}^{\infty}-$rings, with its universal algebra (cf. \cite{BM1}) and with ``smooth commutative algebra" (cf. \cite{BM2}).\\

{\bf Overview of the Paper:}\\

In {\bf Section 1} we present a notion of von Neumann regular $\mathcal{C}^{\infty}-$ring, which is basically a $\mathcal{C}^{\infty}-$ring whose underlying commutative unital ring is a von Neumann regular ring and explore the main properties and characterizations of this concept. In {\bf Section 2}, we show, by two different methods, that every $\mathcal{C}^{\infty}-$ring has a ``closest'' von Neumann regular $\mathcal{C}^{\infty}-$ring, and define their von Neumann regular hull; moreover we describe some properties of the reflection functor as to the preservation of finite limits. In {\bf Section 3}, we present an application: we prove that every homomorphism between boolean algebras can be represented by a $\mathcal{C}^{\infty}-$ring homomorphism of von Neumann regular $\mathcal{C}^{\infty}-$rings.

\section{On von Neumann regular $\mathcal{C}^{\infty}-$rings}

We register a fact which is valid for any $\mathcal{C}^{\infty}-$ring as to idempotents and localizations:\\

\begin{lemma}\label{destino}Let $A$ be \underline{any} $\mathcal{C}^{\infty}-$ring and $e \in A$ an idempotent element. There are unique isomorphisms:
$$A\{ e^{-1}\} \cong \dfrac{A}{(1-e)} \cong A \cdot e := \{ a \cdot e | a \in A\}$$
\end{lemma}
\begin{proof}
Let
$$\begin{array}{cccc}
m_e: & A & \twoheadrightarrow & A\cdot e \\
     & a & \mapsto & a \cdot e
\end{array}$$
and
$$\begin{array}{cccc}
    q & A & \twoheadrightarrow & \dfrac{A}{(1-e)} \\
     & a & \mapsto & a + (1-e)
  \end{array}$$

Since $(1-e)\cdot e = e - e^2 = e - e = 0$, given any $a = x(1-e) \in (1-e)$, $m_e(a) = m_e(x(1-e)) = x(1-e)e = x\cdot 0 = 0$, so $(1-e) \subseteq \ker m_e$.\\

Noting that $(1-e) \subseteq \ker m_e$ and that $m_e$ is a surjective map, by the \textbf{Theorem of the Isomorphism} there is a unique isomorphism $\psi: \dfrac{A}{(1-e)} \to A \cdot e$ such that the following triangle commutes:

$$\xymatrix{
A \ar@{->>}[r]^{m_e} \ar[d]^{q} & A \cdot e \\
\dfrac{A}{(1-e)} \ar@{.>}[ur]^{\exists ! \psi}
}$$

so $\dfrac{A}{(1-e)} \cong A \cdot e$.\\

Finally we show that $A\{ e^{-1}\} \cong \dfrac{A}{(1-e)}$. First we note that since $\eta_e(e) \in (A\{ e^{-1}\})^{\times}$  and $(\eta_e(e))^2 = \eta_e(e^2) = \eta_e(e)$ (that is, $\eta_e(e)$ is an idempotent element of $A\{ e^{-1}\}$), it follows that $\eta_e(e) = 1$.\\

We have $\eta_e(1-e)=\eta_e(1) - \eta_e(e) = 1 - 1 = 0$, so $(1-e) \subseteq \ker \eta_a$, so applying the Theorem of the Homomorphism we get a unique $\mathcal{C}^{\infty}-$homomorphism $\varphi: \dfrac{A}{(1-e)} \to A\{ e^{-1}\}$ such that the following diagram commutes:
$$\xymatrix{
A \ar[r]^{\eta_e} \ar[d]^{q} & A\{ e^{-1}\}\\
\dfrac{A}{(1-e)} \ar@{.>}[ur]^{\exists ! \varphi}
}$$

We have also:
$$q(e)-q(1) = q(e-1) \in (e-1),$$
hence
$$q(e) = q(1) \,\, {\rm in}\,\, \dfrac{A}{(1-e)}.$$

Since $(q(e))^2 = q(e^2)=q(e)$, \textit{i.e.}, it is idempotent, it follows that $q(e) = 1 + (e-1) \in \left( \dfrac{A}{(1-e)}\right)^{\times}$. By the universal property of $\eta_e: A \to A\{ e^{-1}\}$, there is a unique $\mathcal{C}^{\infty}-$homomorphism $\psi: A\{ e^{-1}\} \to \dfrac{A}{(1-e)}$ such that the following diagram commutes:
$$\xymatrix{
A \ar[r]^{\eta_e} \ar[d]^{q} & A\{ e^{-1}\} \ar@{.>}[dl]^{\exists ! \psi}\\
\dfrac{A}{(1-e)}
}$$

By the uniqueness of the arrows $\varphi$ and $\psi$, we conclude that $\varphi \circ \psi = {\rm id}_{A\{ e^{-1}\}}$ and $\psi \circ \varphi = {\rm id}_{\frac{A}{(1-e)}}$, hence $A\{ e^{-1}\} \cong \dfrac{A}{(1-e)}$.
\end{proof}

We now give a definition of a von Neumann regular $\mathcal{C}^{\infty}-$ring: it is a $\mathcal{C}^{\infty}-$ring $(A,\Phi)$ such that $\widetilde{U}(A,\Phi)$ is a von Neumann regular commutative unital ring (here $\widetilde{U}: \mathcal{C}^{\infty}{\rm \bf Rng} \to {\rm \bf CRing}$ is the forgetful functor). More precisely:\\

\begin{definition}\label{arnaldo}Let $\mathfrak{A}=(A,\Phi)$ be a $\mathcal{C}^{\infty}-$ring. We say that $\mathfrak{A}$ is a \index{von Neumann regular $\mathcal{C}^{\infty}-$ring}\textbf{von Neumann regular $\mathcal{C}^{\infty}-$ring} if one (and thus all) of the following equivalent\footnote{See for instance \cite{Mariano} for a proof of these equivalences in the  setting of commutative rings.},  conditions is satisfied:
\begin{itemize}
  \item[(i)]{$(\forall a \in A)(\exists x \in A)(a = a^2x)$;}
  \item[(ii)]{Every principal ideal of $A$ is generated by an idempotent element, \textit{i.e.},
  $$(\forall a \in A)(\exists e \in A)(\exists y \in A)(\exists z \in A)((e^2=e)\& (ey=a) \& (az = e))$$}
  \item[(iii)]{$(\forall a \in A)(\exists ! b \in A)((a=a^2b)\& (b = b^2a))$}
\end{itemize}
\end{definition}

We give a proof of the equivalences above in \textbf{Proposition \ref{arnaldo}}.\\

From now on we must write, when there is no danger of confusion, $A$ instead of $\mathfrak{A}$.\\

An homomorphism of von Neumann $\mathcal{C}^{\infty}-$rings, $A$ and $B$ is simply a $\mathcal{C}^{\infty}-$homo\-morphism between these $\mathcal{C}^{\infty}-$rings. We have the following:\\

%Now we aim at proving that for any ideal $I \subseteq \mathcal{C}^{\infty}(\mathbb{R}^n)$ its $\mathcal{C}^{\infty}-$radical is also an ideal.\\

\begin{definition}We denote by $\mathcal{C}^{\infty}{\rm \bf vNRng}$ the category whose objects are von Neumann-regular $\mathcal{C}^{\infty}-$rings and whose morphisms are the $\mathcal{C}^{\infty}-$homomorphisms between them. Thus, $\mathcal{C}^{\infty}{\rm \bf vNRng}$ is a full subcategory of $\mathcal{C}^{\infty}{\rm \bf Rng}$.
\end{definition}

The following lemma tells us that, in $\mathcal{C}^{\infty}{\rm \bf vNRng}$, taking localizations and taking the ring of fractions with respect to a special element yield  the same object.

\begin{lemma}\label{zeus}If $A$ is a von Neumann regular $\mathcal{C}^{\infty}-$ring, then given any $a \in A$ there is some idempotent element $e \in A$ such that $A\{ a^{-1}\} \cong A\{ e^{-1}\} \cong \dfrac{A}{(1-e)}$.
\end{lemma}
\begin{proof} Let $a \in A$ be any element. Since $A$ is a von Neumann regular $\mathcal{C}^{\infty}-$ring, there exists some idempotent element $e \in A$ such that $a \in (e)$ and $e \in (a)$, so
$$(a) = (e).$$

We claim that:
$$\dfrac{A}{(1-e)} \cong A\{ a^{-1}\}.$$

First notice that the quotient map $q: A \to \dfrac{A}{(1-e)}$ inverts $a$, i.e., $q(a) \in \left( \dfrac{A}{(1-e)} \right)^{\times}$. Indeed, since $(a) = (e)$, there must exist some $y \in A$ such that $ay = e$, so $1 - ay + (1-e)$ is the multiplicative inverse of $q(a) = a + (1-e)$, hence $q(a) \in \left( \dfrac{A}{(1-e)}\right)^{\times}$.

Given $f: A \to B$ a $\mathcal{C}^{\infty}-$rings homomorphism such that $f(a) \in B^{\times}$, consider the following diagram:
$$\xymatrix{
A \ar[r]^{q} \ar[dr]^{f} & \dfrac{A}{(1-e)}\\
   & B
}$$

\textbf{Claim:} $1-e \in \ker f$.\\

Since $f(a) \in B^{\times}$, $f(a)= f(ae) = f(a)\cdot f(e) \in B^{\times}$, so there is some $b \in B$ such that $f(e) \cdot [f(a) \cdot b] = [f(e)f(a)] \cdot b = 1$. Hence $f(e)^{-1} = f(a)\cdot b$.\\

Note that $f(e) \in B^{\times}$ and $f(e) = f(e^2) = f(e)^2$, that is, $f(e)$ is an invertible idempotent of $A$, hence $f(e) = 1$ and $1 - e \in \ker f$.\\

By the \textbf{Theorem of the Homomorphism}, since $(1-e) \subseteq \ker f$ there is a unique $\mathcal{C}^{\infty}-$homomorphism $\widetilde{f}: \dfrac{A}{(1-e)} \to B$ such that
$$\xymatrix{
A \ar[r]^{q} \ar[dr]^{f} & \dfrac{A}{(1-e)} \ar[d]^{\widetilde{f}}\\
   & B
}$$
commutes. \\

Since $\dfrac{A}{(1-e)}$ has the universal property of the ring of fractions $A\{ a^{-1}\}$, it follows that:
$$\dfrac{A}{(1-e)} \cong A\{ a^{-1}\}$$
\end{proof}

\begin{lemma}Let $A$ be a von Neumann-regular $\mathcal{C}^{\infty}-$ring, $S \subseteq A$ and let $\widetilde{U}: \mathcal{C}^{\infty}{\rm \bf Rng} \rightarrow {\rm \bf CRing}$ be the forgetful functor. Then:
$$\widetilde{U}\left( A\{ S^{-1}\}\right) = \widetilde{U}(A)[S^{-1}]$$
\end{lemma}
\begin{proof}
We prove the result first in the case $S=\{ a\}$ for some  $a \in A$.\\

Since $A$ is a von Neumann-regular $\mathcal{C}^{\infty}-$ring, by \textbf{Lemma \ref{zeus}}, given  $a \in A$ there is some idempotent element such that:

$$A\{ a^{-1}\} \cong A\{ e^{-1}\} \cong \dfrac{A}{(1-e)}$$

Now, $\dfrac{A}{(1-e)} \cong A[e^{-1}]$, and $A[e^{-1}] \cong \widetilde{U}(A)[e^{-1}]$, and since $\widetilde{U}(A)[e^{-1}] \cong \widetilde{U}(A)[a^{-1}]$, the result follows.

Whenever $S$ is finite, we have $A\{ {S}^{-1}\} = A\{ a^{-1}\}$, for $a = \prod S$, and we can use the proof we have just made.\\

For a general $S \subseteq A$, we write:

$$S = \bigcup_{S' \subseteq_{\rm fin} S} S'$$

and use the fact that $\widetilde{U}: \mathcal{C}^{\infty}{\rm \bf Rng} \rightarrow {\rm \bf CRing}$ preserves directed colimits.
\end{proof}

As a corollary, we have the following:

\begin{proposition}\label{vNRingsClosedUnderLocalizations}
$\mathcal{C}^{\infty}{\rm \bf vNRng} \subseteq \mathcal{C}^{\infty}{\rm \bf Rng}$ is closed under localizations.
\end{proposition}

The following result is an adaptation of \textbf{Proposition 1} of \cite{Mariano} for the $\mathcal{C}^{\infty}-$case.\\

\begin{theorem}\label{backy}If $A$ is a von Neumann regular $\mathcal{C}^{\infty}-$ring then $A$ is a reduced $\mathcal{C}^{\infty}-$ring.
\end{theorem}
\begin{proof}By the \textbf{Lemma \ref{zeus}},
$$\sqrt[\infty]{(0)} = \biggl\{ a \in A | A\{ a^{-1}\} \cong \{ 0\}\biggr\} = \biggl\{ a \in A | \dfrac{A}{(1-e)} \cong \{ 0\}\biggr\}$$

Now $\dfrac{A}{(1-e)} \cong \{ 0\}$ yields $1 \in (1-e)$, so there must exist some $z \in A$ such that $1 = z\cdot (1-e)$, and $(1-e)$ is an invertible idempotent of $A$, so:
$$1 - e = 1$$
$$e = 0$$
Therefore $a = 0$, so $\sqrt[\infty]{(0)} \subseteq \{ 0\}$. The other inclusion is always true, so:
$$\sqrt[\infty]{(0)} \cong \{ 0\},$$
and $A$ is reduced.
\end{proof}

The following result shows us that whenever $A$ is a von Neumann regular $\mathcal{C}^{\infty}-$ring, the notions of $\mathcal{C}^{\infty}-$spectrum, Zariski spectrum, maximal spectrum and thus, the \index{structure sheaf}structure sheaf of its affine scheme coincide.\\

\begin{theorem}\label{bender}Let $A$ be a von Neumann regular $\mathcal{C}^{\infty}-$ring. Then:
\begin{itemize}
  \item[1)]{$\sqrt[\infty]{(0)} = \sqrt{(0)}=(0)$;}
  \item[2)]{${\rm Spec}^{\infty}\,(A) = {\rm Specm}\,(A) = {\rm Spec}\, (A)$, as topological spaces;}
  \item[3)]{The structure sheaf of $A$ in the category $\mathcal{C}^{\infty}{\rm \bf Rng}$ is equal to the structure sheaf of $U(A)$ in the category ${\rm \bf CRing}$.}
\end{itemize}
\end{theorem}
\begin{proof}
Ad 1): By \textbf{Theorem \ref{backy}}, since $A$ is a von Neumann regular $\mathcal{C}^{\infty}-$ring, $\sqrt[\infty]{(0)} \cong (0)$. But $\widetilde{U}(A)$ is also a von Neumann regular ring, so given any $a \in A$ there exists some idempotent element $e \in A$ such that $(a) = (e)$, i.e.,
$$a \in (e)$$
and
$$e \in (a)$$

The former condition implies the existence of some $x \in A$ such that $a = ex$ so
$$ae = e^2x = ex = a$$
and latter condition implies the existence of some $y \in A$ such that $e = ay$, so
$$ ay = e = e^2 = a^2y^2.$$

Now, given $a \in \sqrt{(0)}$, there is some $n \in \mathbb{N}$ such that $a^n = 0$, so
$$0 = a^n = (ae)^n = a^n e^n = a^ne.$$
Since $e^n = a^n y^n = 0$, then $e=0$, so $a = a \cdot 0 = 0$.

Hence, $\sqrt{(0)} = (0) = \sqrt[\infty]{(0)}$.

Ad 2):

\textbf{Claim:} In a von Neumann regular $\mathcal{C}^{\infty}-$ring every prime ideal is a maximal ideal.\\

 Let $\mathfrak{p}$ be a prime ideal in $A$. Given $a + \mathfrak{p} \neq \mathfrak{p}$ in $\dfrac{A}{\mathfrak{p}}$, then $a + \mathfrak{p} \in \left( \dfrac{A}{\mathfrak{p}}\right)^{\times}$.\\

Since $A$ is a von Neumann regular ring, there exists some $b \in A$ such that $aba = a$, so
$$a + \mathfrak{p} = aba + \mathfrak{p}$$
$$a + \mathfrak{p} = (ab + \mathfrak{p})\cdot (a + \mathfrak{p})$$
$$ab + \mathfrak{p} = 1 + \mathfrak{p}$$
so
$$ab = 1.$$

Hence, every non-zero element of $\dfrac{A}{\mathfrak{p}}$ is invertible, so $\dfrac{A}{\mathfrak{p}}$ is a field. Under those circumstances, it follows that $\mathfrak{p}$ is a maximal ideal, so ${\rm Spec}\, (A) = {\rm Specm}\,(A)$.\\

We always have ${\rm Specm}\, (A) \subseteq {\rm Spec}^{\infty}\, (A)$ and ${\rm Spec}^{\infty}\, (A) \subseteq {\rm Spec}\,(A)$, so:
$${\rm Spec}\,(A) \subseteq {\rm Specm}\, (A) \subseteq {\rm Spec}^{\infty}\,(A) \subseteq {\rm Spec}\,(A)$$
and
$${\rm Spec}^{\infty}\,(A) = {\rm Spec}\, (A).$$

We note, also, that both the topological spaces ${\rm Spec}\,(A)$ amd ${\rm Spec}^{\infty}\,(A)$ have the same basic open sets:
$$D^{\infty}(a) = \{ \mathfrak{p} \in {\rm Spec}^{\infty}\,(A) | a \notin \mathfrak{p} \} = \{ \mathfrak{p} \in {\rm Spec}\,(A) | a \notin \mathfrak{p}\} = D(a),$$
hence $${\rm Spec}^{\infty}\,(A) = {\rm Spec}\, (A)$$ as topological spaces.\\

Ad 3). Immediate.
\end{proof}

\begin{proposition}\label{proposition1}Let $A$ be a $\mathcal{C}^{\infty}-$ring. Then the following are equivalent:
\begin{itemize}
  \item[(i)]{$A$ is von Neumann-regular, \textit{i.e.},
  $$(\forall a \in A)(\exists x \in A)(a = a^2x)$$}
  \item[(ii)]{Every principal ideal of $A$ is generated by an idempotent element, \textit{i.e.},
  $$(\forall a \in A)(\exists e \in A)(\exists y \in A)(\exists z \in A)((e^2=e)\wedge(ey=a)\wedge(az=e))$$}
  \item[(iii)]{$(\forall a \in A)(\exists ! b \in A )((a=a^2b)\wedge(b=b^2a))$}
\end{itemize}
Moreover, when $A$ is von Neumann-regular, then $A$ is reduced (\textit{i.e.}, $\sqrt[\infty]{(0)}=(0)$) and for each $a \in A$ the idempotent element $e$ satisfying (ii) and the element $b$ satisfying (iii) are uniquely determined.
\end{proposition}
\begin{proof}
The implication (iii) $\to$ (i) is obvious, so we omit the proof.\\

Ad (i) $\to$ (ii): Let $I = (a)$ be a principal ideal of $A$. By (i), there is $x \in A$ such that $a = a^2x$, so we define $e:=ax$, which is idempotent since $e^2 = (ax)^2 = a^2x^2 = (a^2x)x = ax = e$. By definition, $e = ax \in (a) = I$, so $(e) \subseteq I$, and since  $a = a^2x = (ax)a = ea$ we also have $a \in (e)$, so $I = (a) \subseteq (e)$. Hence, $I = (e)$.\\

Ad (ii) $\to$ (i): Let $a \in A$ be any element. By (ii) there are $e \in A$, $y \in A$ and $z \in A$ such that $e^2 =e$, $a = ey$ and $e = az$. Define $x:= z^2y$, and we have $a^2x = a^2z^2y = e^2y = ey = a$.\\

Ad (i) $\to$ (iii): Let $a \in A$ be any element. By (i), there is some $x \in A$ such that $a= a^2x$. There can be many $x \in A$ satisfying this role, but there is a ``minimal'' one: the element $ax$ is idempotent and we can project any chosen $x$ down with this idempotent, obtaining $b:=ax^2$. Then $aba = aab^2a = (ax)(ax)a = axa = a$ and $bab = (ax^2)a(ax^2) = (ax)^3x = (ax)x = b$.\\

Now suppose that $A$ is a von Neumann-regular $\mathcal{C}^{\infty}-$ring, and let $a \in A$ be such that $a \in \sqrt[\infty]{(0)}$. Then let $e$ be an idempotent such that $ey=a$, $az=e$, for some $y,z \in A$. Then $a$ is such that $A\{ a^{-1}\} \cong \{0\}$, and by \textbf{Lemma \ref{zeus}} there is some idempotent $e \in A$ such that $A\{ a^{-1}\} \cong \dfrac{A}{(1-e)}$. Now, $A\{ a^{-1}\} \cong \{ 0\}$ occurs if and only if, $\dfrac{A}{(1-e)} \cong \{0\}$, \textit{i.e.}, if and only if, $(1-e) = A$. Since $(1-e) = A$, it follows that $1-e \in A^{\times}$, and since $e \cdot (1-e) = 0$, it follows by cancellation that $e = 0$, hence $a = ey = 0y=0$.\\

Let $e,e' \in A$ be idempotents in an arbitrary ring satisfying $(e) = (e')$. Select $r,r'\in A$ such that $er'=e'$ and $e'r = e$. Then $e' = er' = er'e = e'e = e're'=e'r=e$. Thus, if an ideal is generated by an idempotent element, this element is uniquely determined.\\

Finally, let $A$ be a von Neumann-regular $\mathcal{C}^{\infty}-$ring. Select a member $a \in A$ and consider $b,b'\in A$ such that $a^2b'=a=a^2b$, $b=b^2a$, $b'={b'}^2a$. Then $(b-b')a^2=(b-b')a^2=(b-b')(ba^2-b'a^2)=(b-b')(a-a)=(b-b')\cdot 0 = 0$ and $[(b-b')\cdot a]^2 \in (0)$. Since $A$ is $\mathcal{C}^{\infty}-$reduced, $[(b-b')\cdot a]^2 \in (0) = \sqrt[\infty]{(0)}$. By item (1) of \textbf{Theorem \ref{bender}}, $\sqrt[\infty]{(0)} = \sqrt{(0)}$, so $[(b-b')\cdot a]^2 \in \sqrt[\infty]{(0)} = \sqrt{(0)}$ and $(b-b')\cdot a = 0$. Therefore $b - b' = b^2a - {b'}^2a= (b^2-{b'}^2)a = (b+b')(b-b')a = (b+b') \cdot 0 = 0$.
\end{proof}

\begin{remark}\label{agp}Let $A$ be a von Neumann-regular $\mathcal{C}^{\infty}-$ring and $e \in A$ be any idempotent element. Then $A \cdot e$ is a von Neumann-regular $\mathcal{C}^{\infty}-$ring. Indeed, we have $A \cdot e \cong \dfrac{A}{(1-e)}$ and the latter is an homomorphic image of a von Neumann-regular $\mathcal{C}^{\infty}-$ring, namely $\dfrac{A}{(1-e)} = q[A]$. Since $\dfrac{A}{(1-e)}$ is the homomorphic image of a von Neumann-regular $\mathcal{C}^{\infty}-$ring, $\dfrac{A}{(1-e)}$ is a von Neumann-regular $\mathcal{C}^{\infty}-$ring. Since $A \cdot e \cong \dfrac{A}{(1-e)}$, it follows that $A \cdot e$ is a von Neumann-regular $\mathcal{C}^{\infty}-$ring.
\end{remark}

\begin{lemma}\label{moura}Let $A$ be a local $\mathcal{C}^{\infty}-$ring. The only idempotent elements of $A$ are $0$ and $1$.
\end{lemma}
\begin{proof}
Since $A$ is a local $\mathcal{C}^{\infty}-$ring, we have:
$$(\forall x \in A)(\forall y \in A)((x+y \in A^{\times}) \to (x \in A^{\times})\vee(y \in A^{\times})).$$

Let $e \in A$ be an idempotent. We have $(1 = e + (1-e) \in A^{\times})\to((e \in A^{\times})\vee(1-e \in A^{\times}))$. If $e \in A^{\times}$ then $e=1$ (since $1$ is the only invertible idempotent), and if $1-e \in A^{\times}$ then, by the same reasons, $1-e = 1$, hence $e=0$.
\end{proof}

\begin{proposition}\label{dan}Let $A$ be a von Neumann-regular $\mathcal{C}^{\infty}-$ring whose only idempotent elements are $0$ and $1$. Then the following assertions are equivalent:
\begin{itemize}
  \item[(i)]{$A$ is a $\mathcal{C}^{\infty}-$field;}
  \item[(ii)]{$A$ is a $\mathcal{C}^{\infty}-$domain;}
  \item[(iii)]{$A$ is a local $\mathcal{C}^{\infty}-$ring.}
\end{itemize}
\end{proposition}
\begin{proof}
The implications (i) $\to$ (ii), (i) $\to$ (iii) are immediate, so we omit their proofs. \\

Ad (iii) $\to$ (i): Suppose $A$ is a local $\mathcal{C}^{\infty}-$ring. Since $A$ is a von Neumann-regular $\mathcal{C}^{\infty}-$ring, given any $x \in A\setminus \{ 0\}$ there exists some idempotent element $e \in A$ such that $(x)=(e)$. However, the only idempotent elements of $A$ are, by \textbf{Lemma \ref{moura}}, $0$ and $1$. We claim that $(x)=(1)$, otherwise we would have $(x)=(0)$, so $x=0$.\\

Now, $(x) = (1)$ implies $1 \in (x)$, so there is some $y \in A$ such that $1 = x \cdot y = y \cdot x$, and $x$ is invertible. Thus $A$ is a $\mathcal{C}^{\infty}-$field.\\

Ad (ii) $\to$ (i):Suppose $A$ is a $\mathcal{C}^{\infty}-$domain. Given any $x \in A \setminus \{ 0\}$, we have:
$$(\forall y \in A \setminus \{ 0\})(x\cdot y \neq 0),$$
so $(x)\neq (0)$. Since $A$ is a von Neumann-regular $\mathcal{C}^{\infty}-$ring, $(x)$ is generated by some non-zero idempotent element, namely, $1$. Hence $(x) = (1)$ and $x \in A^{\times}$.
\end{proof}

\begin{proposition}\label{proposition4} The inclusion functor $\imath : \mathcal{C}^{\infty}{\rm \bf vNRng} \hookrightarrow \mathcal{C}^{\infty}{\rm \bf Rng}$ creates filtered colimits, \textit{i.e.}, $\mathcal{C}^{\infty}{\rm \bf vNRng} \subset \mathcal{C}^{\infty}{\rm \bf Rng}$
\end{proposition}
\begin{proof}A filtered colimit of von Neumann-regular $\mathcal{C}^{\infty}-$rings, taken in $\mathcal{C}^{\infty}{\rm \bf Rng}$ is a von Neumann-regular $\mathcal{C}^{\infty}-$ring again. Indeed, filtered colimits in $\mathcal{C}^{\infty}{\rm \bf Rng}$ are formed by taking the colimit of the underlying sets and defining the $n-$ary functional symbol $f^{(n)}$ of an $n-$tuple $(a_1, \cdots, a_n)$ into a common $\mathcal{C}^{\infty}-$ring occurring in the diagram and taking the element $f^{(n)}(a_1, \cdots, a_n)$ there. \\

Let $(I, \leq)$ be a filtered poset and $\{ A_i\}_{i \in I}$ be a filtered family of $\mathcal{C}^{\infty}-$ring. Given any element $\alpha \in \varinjlim A_i$, there is some $i \in I$ and $a_i \in A_i$ such that $\alpha = [(a_i,i)]$. Since $A_i$ is a von Neumann-regular $\mathcal{C}^{\infty}-$ring, there must exist some idempotent $e_i \in A_i$ such that $(a_i)=(e_i)$. It suffices to take $\eta = [(e_i,i)] \in \varprojlim A_i$, which is an idempotent element of $\varinjlim A_i$ such that $(\alpha) = ([(a_i, i)]) = ([(e_i,i)]) = (\eta)$. Hence $\varinjlim A_i$ is a von Neumann-regular $\mathcal{C}^{\infty}-$ring.
\end{proof}

We have the following important result, which relates von Neumann-regular $\mathcal{C}^{\infty}-$rings to the topology of its smooth Zariski spectrum:\\

\begin{theorem}\label{ota} Let $A$ be a $\mathcal{C}^{\infty}-$ring. The following assertions are equivalent:
\begin{itemize}
  \item[(i)]{$A$ is a von Neumann-regular $\mathcal{C}^{\infty}-$ring;}
  \item[(ii)]{$A$ is a $\mathcal{C}^{\infty}-$reduced $\mathcal{C}^{\infty}-$ring (i.e., $\sqrt[\infty]{(0)}=(0)$) and ${\rm Spec}^{\infty}\,(A)$ is a \index{Boolean space}Boolean space, \textit{i.e.},a compact, Hausdorff and totally disconnected space.}
\end{itemize}
\end{theorem}
\begin{proof}
Ad $(i) \to (ii)$:   Suppose $A$ is a von Neumann regular $\mathcal{C}^{\infty}-$ring. We are going to show that $\sqrt[\infty]{(0)} = (0)$.\\

Given $a \in \sqrt[\infty]{(0)}$, since $A$ is a von Neumann regular $\mathcal{C}^{\infty}-$ring, there are some $e \in A$ such that $e^2 =e$, some $x \in A$ and some $y \in A$ such that $a = ex$ and $e = ay$.\\

Since $a \in \sqrt[\infty]{(0)}$ and $\sqrt[\infty]{(0)}$ is an ideal, we have $e = ay \in \sqrt[\infty]{(0)}$. From \textbf{Theorem 18}, p. 78 of \cite{BM2}, we conclude that:
$$(\forall \p \in {\rm Spec}^{\infty}\,(A))(e \in \p).$$

Now, $e^2 = e$ implies that $(1-e)\cdot e = 0$, so either $e=0$ or $e=1$. The latter does not occur, since $\p$ is a proper prime ideal, hence $e=0$ and $a =0$. Hence $\sqrt[\infty]{(0)} \subseteq (0)$ and $\sqrt[\infty]{(0)}=(0)$.\\

It remains to show that ${\rm Spec}^{\infty}\,(A)$ is a Boolean space. Since ${\rm Spec}^{\infty}\,(A)$ is a spectral space, we only need to show that:
$$\mathcal{B} = \{ D^{\infty}(a) | a \in A \}$$
is a clopen basis for its topology.\\

Given any $a \in A$, since $A$ is a von Neumann regular $\mathcal{C}^{\infty}-$ring, there is some idempotent element $e \in A$ such that $(a)=(e)$, so $D^{\infty}(a) = D^{\infty}(e)$. \\

We claim that ${\rm Spec}^{\infty}\,(A) \setminus D^{\infty}(e) = D^{\infty}(1-e)$, hence $D^{\infty}(a) = D^{\infty}(e)$ is a clopen set.\\

From item (iii) of \textbf{Lemma 1.2} of \cite{rings2},
$$D^{\infty}(e) \cap D^{\infty}(1-e) = D^{\infty}(e\cdot(1-e)) = D^{\infty}(0) = \{ \p \in {\rm Spec}^{\infty}\,(A) | 0 \notin \p \} = \varnothing .$$
Also by item (iii) of \textbf{Lemma 1.2} of \cite{rings2},
$$D^{\infty}(e) \cup D^{\infty}(1-e) = D^{\infty}(e + (1-e)) = D^{\infty}(1) = {\rm Spec}^{\infty}\,(A)$$
so it follows that ${\rm Spec}^{\infty}\,(A)$ is a Boolean space and $A$ is a $\mathcal{C}^{\infty}-$reduced $\mathcal{C}^{\infty}-$ring.\\

Ad $(ii) \to (i)$. Suppose that $A$ is a $\mathcal{C}^{\infty}-$reduced $\mathcal{C}^{\infty}-$ring and that ${\rm Spec}^{\infty}\,(A)$ is a Boolean space.\\

Since ${\rm Spec}^{\infty}\,(A)$ is a Boolean space, it is a Hausdorff space and for every $a \in A$, $D^{\infty}(a)$ is compact, hence it is closed. From that we conclude that for every $a \in A$, $D^{\infty}(a)$ is a clopen set. \\

We conclude that ${\rm Spec}^{\infty}\,(A) \setminus D^{\infty}(a)$ is a clopen subset of ${\rm Spec}^{\infty}\,(A)$.\\

We claim that for every clopen $C$ in ${\rm Spec}^{\infty}\,(A)$ there is some $b \in A$ such that $C = D^{\infty}(b)$.\\

Since $C$ is clopen in ${\rm Spec}^{\infty}\,(A)$, it is in particular an open set, and since $\{ D^{\infty}(a) | a \in A\}$ is a basis for the topology of ${\rm Spec}^{\infty}\,(A)$,  there is a family $\{ b_i\}_{i \in I}$ of elements of $A$ such that:
$$C = \bigcup_{i \in I} D^{\infty}(b_i)$$

Since $C$ is compact, there is a finite subset $I' \subseteq I$ such that:
$$C = \bigcup_{i \in I'} D^{\infty}(b_i)$$

By item (iii) of \textbf{Lemma 1.4} of \cite{rings2},
$$\bigcup_{i \in I'} D^{\infty}(b_i) = D^{\infty}\left( \sum_{i \in I'} {b_i}^2 \right).$$

Hence, given a clopen set $C$, there is an element $b = \sum_{i \in I'} {b_i}^2$ such that $C = D^{\infty}(b)$.\\

Applying the above result to the clopen ${\rm Spec}^{\infty}\,(A) \setminus D^{\infty}(a)$, we obtain an element $d \in A$ such that:
$${\rm Spec}^{\infty}\,(A) \setminus D^{\infty}(a) = D^{\infty}(d).$$

We have:
$$\varnothing = D^{\infty}(a) \cap D^{\infty}(d) = D^{\infty}(a\cdot d) = \{ \p \in {\rm Spec}^{\infty}\,(A) | a \cdot d \notin \p \}$$
so
$$(\forall \p \in {\rm Spec}^{\infty}\,(A))(a \cdot d \in \p)$$
hence
$$a \cdot d \in \bigcap {\rm Spec}^{\infty}\,(A) = \sqrt[\infty]{(0)} = (0)$$
where the last equality is due to the fact that $A$ is supposed to be a $\mathcal{C}^{\infty}-$reduced $\mathcal{C}^{\infty}-$ring.\\

We have, then,
$$a \cdot d = 0.$$

We have, also:
$$D^{\infty}(a^2 + d^2) = D^{\infty}(a) \cup D^{\infty}(b) = {\rm Spec}^{\infty}\,(A) = D^{\infty}(1).$$

By item (i) of \textbf{Lemma 1.4} of \cite{rings2}, $D^{\infty}(1) \subseteq D^{\infty}(a^2+d^2)$ implies $1 \in \{ a^2 + d^2\}^{\infty-{\rm sat}}$, which happens if, and only if, $a^2 + d^2 \in \sqrt[\infty]{(1)}$.\\

Now,  $a^2 + d^2 \in \sqrt[\infty]{(1)}$ occurs if, and only if, $(\exists c \in \{ 1\}^{\infty-{\rm sat}})$ and $(\exists \lambda \in A)$ such that:
$$\lambda \cdot c \cdot (a^2 + d^2) \in \langle 1 \rangle = \{ 1\}$$
Let $y = \lambda \cdot c$, and we get:
$$y \cdot (a^2 + d^2) = 1$$
$$ya^2 + yd^2 = 1$$
Since $a \cdot d = 0$, we get:
$$a(a^2y)+a(b^2y)= a \cdot 1 = a$$
$$a^2 (a \cdot y) = a$$

and considering the forgetful functor:
$$\widetilde{U}: \mathcal{C}^{\infty}{\rm \bf Rng} \to \, {\rm \bf CRing}$$

Let $a \in \sqrt[\infty]{(0)}$, i.e., $a \in A$ is such that $\left( \dfrac{A}{(0)}\right)\{ (a + (0))^{-1} \} \cong A\{ a^{-1}\} \cong 0$. Since $A\{ a^{-1}\} \cong \dfrac{A\{ x\}}{(ax-1)}$, $A\{ a^{-1}\} \cong 0$ implies
$$1_{\frac{A\{ x\}}{(ax - 1)}} = 0_{\frac{A\{ x\}}{(ax - 1)}},$$
so
$$1_{A\{ x\}} \in (ax-1),$$
which occurs if and only if, there exists some $h \in A\{ x\} \cong A \otimes_{\infty} \mathcal{C}^{\infty}(\R)$ such that $1 = h(x) \cdot [\Psi(\jmath(a)) \cdot \Psi(\jmath'({\rm id}_{\R})) - 1] = h(x) \cdot (ax -1)$, where $\jmath : A \to A\{ x\}$ and $\jmath': \mathcal{C}^{\infty}(\R) \to A\{ x\}$ are the coproduct morphisms and $\Psi: A \otimes_{\infty} \mathcal{C}^{\infty}(\R) \to A\{ x\}$ is the isomorphism between them, as we see in the following diagram:
$$\xymatrix{
\mathcal{C}^{\infty}(\R) \ar[dr]^{\jmath'} &  & \\
       & A \otimes_{\infty}\mathcal{C}^{\infty}(\R) \ar[r]^{\Psi} & A\{ x\} \\
A \ar[ur]^{\jmath} & &
}$$

Under those circumstances we conclude that if $A$ is a reduced $\mathcal{C}^{\infty}-$ring, then there is $h \in A\{ x\}$ such that $1 = h(x) \cdot (ax-1)$, so $1 \in (ax - 1)$ and
$$1_{\dfrac{A\{ x\}}{(ax - 1)}} = 0_{\dfrac{A\{ x\}}{(ax - 1)}},$$
or
$$1_{A\{ a^{-1}\}} = 0_{A\{ a^{-1}\}}$$
so
$$ a \in \sqrt[\infty]{(0)}$$
\end{proof}

%%%%%%%%%%%%%%%%%%%%%%%%%%%%%%%%%%%%%%%%%%%%%%%%%%%%%%%%%%%%%
The following proposition will be useful to characterize the von Neumann-regular $\mathcal{C}^{\infty}-$rings by means of the ring of global sections of the structure sheaf of its affine scheme.\\

\begin{proposition}\label{Phil} If a $\mathcal{C}^{\infty}-$ring $A$ is a von-Neumann-regular $\mathcal{C}^{\infty}-$ring and $\p \in {\rm Spec}^{\infty}\,(A)$, then $A\{ {A \setminus \p}^{-1}\}$ is a $\mathcal{C}^{\infty}-$field.
\end{proposition}
\begin{proof}
We are going to show that the only maximal ideal of $A\{ {A \setminus \p}^{-1}\}$, $\m_{\p}$ is such that $\m_{\p} \cong \{ 0\}$.\\

Let $\eta_{\p}: A \to A\{ {A \setminus \p}^{-1}\}$ be the localization morphism of $A$ with respect to $A \setminus \p$. We have $\m_{\p} = \langle \eta_{\p}[A \setminus \p]\rangle = \left\{ \dfrac{\eta_{\p}(a)}{\eta_{\p}(b)} | (a \in \p)\wedge(b \in A \setminus \p)  \right\}$. We must show that for every $a \in \p$, $\eta_{\p}(a) = 0$, which is equivalent, by \textbf{Theorem 1.4} of \cite{rings2}, to assert that for every $a \in \p$ there is some $c \in (A \setminus \p)^{\infty-{\rm sat}} = A\setminus \p$  such that $c \cdot a = 0$ in $A$.\\

\textit{Ab absurdo}, suppose $\m_{\p} \neq \{ 0\}$, so there is $a \in \p$ such that $\eta_{\p}(a)\neq 0$, i.e., such that for every $c \in A \setminus \p$, $c \cdot a \neq 0$. Since $A$ is a von Neumann-regular $\mathcal{C}^{\infty}-$ring, for this $a$ there is some idempotent $e \in \p$ such that $(a) = (e)$.\\

Since $a \in (e)$, there is some $\lambda \in A$ such that $a = \lambda \cdot a$, hence:
$$0 \neq \eta_{\p}(a) = \eta_{\p}(\lambda \cdot e) = \eta_{\p}(\lambda) \cdot \eta_{\p}(e)$$
and $\eta_{\p}(e) \neq 0$.\\

Since $\eta_{\p}(e) \neq 0$,
\begin{equation}\label{eq1}
(\forall d \in A\setminus \p)(d \cdot e \neq 0).
\end{equation}

Since $e$ is an idempotent element, $1-e \notin \p$, for if $1-e \in \p$ then $e + (1 - e) = 1 \in \p$ and $\p$ would not be a proper prime ideal.\\

We have also:
\begin{equation}\label{eq2}
e \cdot (1 - e) = 0,
\end{equation}

The equation \eqref{eq2} contradicts \eqref{eq1}, so $\m_{\p} \cong \{ 0\}$ and $A\{{A \setminus \p}^{-1}\}$ is a $\mathcal{C}^{\infty}-$field.
\end{proof}

\begin{corollary}\label{pablo} Let $A$ be a von Neumann-regular $\mathcal{C}^{\infty}-$ring and $\p \in {\rm Spec}^{\infty}\,(A)$. We have:
$$\dfrac{A}{\p} \cong A\{ {A \setminus \p}^{-1}\}.$$
\end{corollary}
\begin{proof}Since $A$ is a von Neumann-regular $\mathcal{C}^{\infty}-$ring, ${\rm Spec}^{\infty}\,(A) = {\rm Specm}^{\infty}\,(A)$ and for every $\p \in {\rm Spec}^{\infty}\,(A) = {\rm Specm}^{\infty}\,(A)$, $\dfrac{A}{\p}$ is a $\mathcal{C}^{\infty}-$field, so
$$\left( \dfrac{A}{\p}\right)\left\{ \left( \dfrac{A}{\p} \setminus \{ 0 + \p\}\right)^{-1}\right\} \cong \dfrac{A}{\p}$$ Also, since $A\{ {A \setminus \p}^{-1}\}$ is a $\mathcal{C}^{\infty}-$field (by \textbf{Proposition \ref{Phil}}), the quotient map $q_{\m_{\p}}: A\{ {A \setminus \p}^{-1}\} \to \dfrac{A\{ {A \setminus \p}^{-1}\}}{\m_{\p}}$ is an isomorphism, so $A\{ {A \setminus \p}^{-1}\} \cong \dfrac{A\{ {A \setminus \p}^{-1}\}}{\m_{\p}}$.\\

By \textbf{Theorem 24} of p.96 of \cite{BM2}, $\left( \dfrac{A}{\p}\right)\left\{ \left( \dfrac{A}{\p} \setminus \{ 0 + \p\}\right)^{-1}\right\} \cong \dfrac{A\{ {A \setminus \p}^{-1} \}}{\m_{\p}}$, hence:
$$A\{ {A \setminus \p}^{-1}\} \cong \dfrac{A\{ {A \setminus \p}^{-1} \}}{\m_{\p}} \cong \left( \dfrac{A}{\p}\right)\left\{ \left( \dfrac{A}{\p} \setminus \{ 0 + \p\}\right)^{-1}\right\} \cong \dfrac{A}{\p}.$$
\end{proof}

As a consequence, we register another proof of $(iii)  \rightarrow (i)$ of Proposition \ref{dan}.

\begin{corollary}\label{fischer}Let $\mathfrak{A}=(A,\Phi)$ be a local von Neumann-regular $\mathcal{C}^{\infty}-$ring. Then $\mathfrak{A}$ is a $\mathcal{C}^{\infty}-$field.
\end{corollary}
\begin{proof} Since $A$ is a local $\mathcal{C}^{\infty}-$ring there is a unique maximal ideal, $\mathfrak{m} \subseteq A$. Since $A$ is a von Neumann-regular $\mathcal{C}^{\infty}-$ring, by \textbf{Corollary \ref{pablo}},
$$A_{\mathfrak{m}} \cong \dfrac{A}{\mathfrak{m}},$$
and the latter is a $\mathcal{C}^{\infty}-$field.\\

Also, $A_{\mathfrak{m}} = A\{ A \setminus \mathfrak{m}^{-1}\} = A\{ {A^{\times}}^{-1}\} \cong A$, and since $A_{\mathfrak{m}}$ is isomorphic to a $\mathcal{C}^{\infty}-$field, it follows that $A$ is a $\mathcal{C}^{\infty}-$field.
\end{proof}

Summarizing, we have the following result:

\begin{theorem}\label{lab}If $A$ is a von Neumann-regular $\mathcal{C}^{\infty}-$ring, then $({\rm Spec}^{\infty}\,(A), {\rm Zar}^{\infty})$ (which we denote simply by ${\rm Spec}^{\infty}(A)$) is a Boolean topological space, by  \textbf{Theorem \ref{ota}}. Moreover, by \textbf{Proposition \ref{Phil}}, for every $\mathfrak{p} \in {\rm Spec}^{\infty}\,(A)$,
$$A_{\mathfrak{p}} = \varinjlim_{a \notin \mathfrak{p}} A\{ a^{-1}\} \cong A\{ A\setminus \mathfrak{p}\}$$
is a $\mathcal{C}^{\infty}-$field.
\end{theorem}

The above theorem suggests us that von Neumann-regular $\mathcal{C}^{\infty}-$rings behave much like ordinary von Neumann-regular commutative unital rings. In the next chapter we are going to explore this result using sheaf theoretic machinery.

\begin{proposition}\label{proposition6}The limit in $\mathcal{C}^{\infty}{\rm \bf Rng}$ of a diagram of von Neumann-regular $\mathcal{C}^{\infty}-$rings is a von Neumann-regular $\mathcal{C}^{\infty}-$ring. In particular, $\mathcal{C}^{\infty}{\rm \bf vNRng}$ is a complete category and the inclusion functor $\imath: \mathcal{C}^{\infty}{\rm \bf vNRng} \hookrightarrow \mathcal{C}^{\infty}{\rm \bf Rng}$ preserves all limits.
\end{proposition}
\begin{proof}
It is clear from the definition that the class $\mathcal{C}^{\infty}{\rm \bf vNRng}$ of von Neumann-regular $\mathcal{C}^{\infty}-$rings is closed under arbitrary products in the class $\mathcal{C}^{\infty}{\rm \bf Rng}$ of all $\mathcal{C}^{\infty}-$rings. Thus it suffices to show that it is closed under equalizers. So let $A,B$ be  von Neumann-regular $\mathcal{C}^{\infty}-$rings and $f,g: A \to B$ $\mathcal{C}^{\infty}-$ring homomorphisms. Their equalizer in the category $\mathcal{C}^{\infty}{\rm \bf Rng}$ is given by the set $E:= \{ a \in A | f(a) = g(a)\}$, endowed with the restricted ring operations from $A$.\\

To see that $E$ is a von Neumann-regular $\mathcal{C}^{\infty}-$ring, we need to show that for $a \in E$, the unique element $b \in A$ satisfying $a^2b=a$ and $b^2a=b$ also belongs to $E$.\\

First we show that the idempotent element $ab$ belongs to $E$. Indeed, we have $f(ab)=f(a)f(b)=g(a)f(b)=g(a^2b)f(b) = g(a)g(ab)f(b) = f(a)g(ab)f(b)=f(ab)g(ab)$. Exchanging $f$ and $g$ in this chain of equations, we also get $g(ab)=f(ab)g(ab)$. Altogether we obtain $g(ab)=f(ab)$, and hence $ab \in E$.\\

Now we use this, as well as the fact that we also have the equation $b=ab^2$, and conclude $f(b)=f(ab^2)=f(b)f(ab)=f(b)g(ab)=f(b)g(a)g(b) = f(b)f(a)g(b)=f(ab)g(b) = g(ab)g(b)=g(ab^2)=g(b)$.
\end{proof}

\begin{proposition}\label{proposition8} The category $\mathcal{C}^{\infty}{\rm \bf vNRng}$ is the smallest subcategory of $\mathcal{C}^{\infty}{\rm \bf Rng}$ closed under limits containing all $\mathcal{C}^{\infty}-$fields.
\end{proposition}
\begin{proof}
Clearly all $\mathcal{C}^{\infty}-$fields are von Neumann-regular $\mathcal{C}^{\infty}-$rings, and by \textbf{Proposition \ref{proposition6}} so are limits of $\mathcal{C}^{\infty}-$fields. Thus $\mathcal{C}^{\infty}{\rm \bf vNRng}$ contains all limits of $\mathcal{C}^{\infty}-$fields. On the other hand the ring of global sections of a sheaf can be expressed as a limit of a diagram of products and ultraproducts of the stalks (by \textbf{Lemma 2.5} of \cite{Ken}). All these occurring (ultra)products are von Neumann-regular $\mathcal{C}^{\infty}-$rings as well and hence so is their limit, by \textbf{Proposition \ref{proposition6}}.
\end{proof}

\section{Sheaves and von Neumann Regular $\mathcal{C}^{\infty}-$Rings}

In this section we present applications of sheaf-theoretic notions to von Neumann Regular $\mathcal{C}^{\infty}-$Rings.

\subsection{The von Neumann Regular Hull of a $\mathcal{C}^{\infty}-$Ring}

In this subsection we follow \cite{Mariano}.\\

As defined, the affine scheme of a $\mathcal{C}^{\infty}-$ring $A$ is the $\mathcal{C}^{\infty}-$locally ringed space $({\rm Spec}^{\infty}\,(A), \Sigma_A)$, where given any basic open set $D^{\infty}\,(a)$, for some $a \in A$, one defines $\Sigma_A(D^{\infty}(a)):= A\{ a^{-1}\}$. The stalks of this sheaf are local $\mathcal{C}^{\infty}-$rings. In particular we have, for any $\mathcal{C}^{\infty}-$ring $A$:

$$\Gamma({\rm Spec}^{\infty}\,(A), \Sigma_A)\cong A,$$

as stated in \textbf{Theorem 1.7} of \cite{rings2}. This fact suggests us that if we get a sheaf of $\mathcal{C}^{\infty}-$rings over the booleanization of the spectral space ${\rm Spec}^{\infty}\,(A)$, that is, over ${\rm Spec}^{\infty}\,(A)$ together with the constructible topology, whose stalk at each prime $\mathcal{C}^{\infty}-$radical ideal $\mathfrak{p}$ is the residue field $k_{\mathfrak{p}}\,(A)$ of the local $\mathcal{C}^{\infty}-$ring $A_{\mathfrak{p}}$, then the ring of global sections of this sheaf should be the ``closest'' von Neumann regular $\mathcal{C}^{\infty}-$ring to $A$.\\

In the following, we are to see that \textit{if} a $\mathcal{C}^{\infty}-$ring $A$ has a von Neumann regular hull, that is, a pair $(\nu_A, {\rm vN}\,(A))$ where $\nu_A: A \rightarrow {\rm vN}\,(A)$ is such that for any von Neumann regular $\mathcal{C}^{\infty}-$ring $B$ and every $\mathcal{C}^{\infty}-$homomorphism $f: A \to B$ there is a unique $\mathcal{C}^{\infty}-$homomorphism $\widetilde{f}: {\rm vN}\,(A) \to B$ such that the following diagram commutes:

$$\xymatrixcolsep{3pc}\xymatrix{
A \ar[r]^{\nu_A} \ar[dr]_{f} & {\rm vN}\,(A) \ar@{-->}[d]^{\exists ! \widetilde{f}}\\
 & B
}$$

\textit{then} the spectral sheaf of ${\rm vN}\,(A)$ \textit{must} be such that the topological space ${\rm Spec}^{\infty}({\rm vN}\,(A))$ is homeomorphic to the booleanization of ${\rm Spec}^{\infty}\,(A)$ and the stalks ${\rm vN}\,(A)_{\mathfrak{p}}$, for $\mathfrak{p} \in {\rm Spec}^{\infty}\,({\rm vN}\,(A))$, are isomorphic to the residue fields of the local $\mathcal{C}^{\infty}-$ring $A_{\mathfrak{p}}$, for $\mathfrak{p} \in {\rm Spec}^{\infty}\,(A)$.\\

In fact, \underline{every} $\mathcal{C}^{\infty}-$ring has a von Neumann-regular hull.\\

Given any $\mathcal{C}^{\infty}-$ring $A$, there exists a von Neumann regular $\mathcal{C}^{\infty}-$ring and a $\mathcal{C}^{\infty}-$homomorphism $\nu: A \to V$ such that for every $\mathcal{C}^{\infty}-$ring $B$ and any $\mathcal{C}^{\infty}-$homomorphism $g: A \to B$  there is a unique $\mathcal{C}^{\infty}-$homomorphism $f: V \to B$ such that the following diagram commutes:

$$\xymatrixcolsep{3pc}\xymatrix{
A \ar[r]^{\nu} \ar[dr]_{g} & V \ar[d]^{\exists ! f}\\
 & B
}$$

It will follow that the inclusion functor $i: \mathcal{C}^{\infty}{\rm \bf vNRng} \hookrightarrow \mathcal{C}^{\infty}{\rm \bf Rng}$ has a left-adjoint $\nu$, that is, $\mathcal{C}^{\infty}{\rm \bf vNRng}$ is a reflexive subcategory of $\mathcal{C}^{\infty}{\rm \bf Rng}$ (cf. p. 91 of \cite{CWM}).\\

Given the $\mathcal{C}^{\infty}-$ring $(A, \Phi)$, we first take the coproduct of $A$ with the free $\mathcal{C}^{\infty}-$ring generated by $A$, $\eta_A: A \rightarrow (L(A),\Omega)$:

$$\xymatrixcolsep{3pc}\xymatrix{
 & A\{ x_a | a \in A\}& \\
A \ar[ur]^{\iota_A} & & L(A) \ar[ul]_{\iota_{L(A)}}
}$$

where:

$$\begin{array}{cccc}
    \iota_A: & A & \rightarrow & A\{ x_a | a \in A\}  \\
     & a & \mapsto & \iota_A(a)
  \end{array}$$

and:

$$\begin{array}{cccc}
    \iota_{L(A)}: & L(A) & \rightarrow & A\{ x_a | a \in A\}  \\
     & \widetilde{a} & \mapsto & \iota_{L(A)}(\widetilde{a})
  \end{array}$$

and then we divide it by the ideal generated by the set:

$$\{ (\iota_{L(A)}\circ \eta_A)(a)\cdot \iota_A(a) \cdot (\iota_{L(A)} \circ \eta_A)(a) - (\iota_{L(A)}\circ \eta_A)(a), \iota_A(a)\cdot (\iota_{L(A)}\circ \eta_A)(a) \cdot \iota_A(a) - \iota_A(a) | a \in A \}$$

so we have:

\begin{equation}\label{geleia}
\dfrac{A\{ x_a | a \in A\}}{\langle \{ (\iota_{L(A)}\circ \eta_A)(a)\cdot \iota_A(a) \cdot (\iota_{L(A)} \circ \eta_A)(a) - (\iota_{L(A)}\circ \eta_A)(a), \iota_A(a)\cdot (\iota_{L(A)}\circ \eta_A)(a) \cdot \iota_A(a) - \iota_A(a) | a \in A \}\rangle}\end{equation}

In order to simplify our exposition, within this context we are going to write $x_a$ to denote $(\iota_{L(A)}\circ \eta_A)(a) \in A\{ x_a | a \in A\}$ for any $a \in A$, and $a$ to denote $\iota_A(a) \in A\{ x_a | a \in A\}$. Thus we write:

$$\dfrac{A\{ x_a | a \in A\}}{\langle \{x_a \cdot a \cdot x_a - x_a, a \cdot x_a \cdot a - a | a \in A\}\rangle}$$

instead of \eqref{geleia}.\\

Given a von Neumann regular $\mathcal{C}^{\infty}-$ring $V$ and a $\mathcal{C}^{\infty}-$homomorphism $f: A \to V$, define the function:

$$\begin{array}{cccc}
    F: & A & \rightarrow & V \\
     & a & \mapsto & f(a)^{*}
  \end{array}$$

where $f(a)^{*}$ denotes the quasi-inverse of $f(a)$ in $V$.\\

From the universal property of the free $\mathcal{C}^{\infty}-$ring generated by $A$, $L(A)$, there is a unique $\mathcal{C}^{\infty}-$homomorphism $\widetilde{F}: L(A) \rightarrow V$ such that the following diagram commutes:

$$\xymatrixcolsep{3pc}\xymatrix{
U(A) \ar[r]^{\eta_A} \ar[dr]_{F} & L(A) \ar@{-->}[d]^{\exists ! \widetilde{F}}\\
 & V}$$

By the universal property of the $\mathcal{C}^{\infty}-$coproduct, given the $\mathcal{C}^{\infty}-$homomorphisms $f: A \to V$ and $\widetilde{F}: L(A) \to V$ (the latter is constructed from $f$), there is a unique $\mathcal{C}^{\infty}-$homomorphism $\widehat{f}: A\{ x_a | a \in A\} \to V$ such that:

$$\xymatrixcolsep{3pc}\xymatrix{
A \ar[dr]^{\iota_A}\ar@/^2pc/[rrd]^{f}& & \\
 & A\{x_a | a \in A \} \ar@{-->}[r]^{\exists ! \widehat{f}} & V \\
L(A) \ar[ur]_{\iota_{L(A)}} \ar@/_2pc/[rru]_{\widetilde{F}}& &
}$$

commutes.\\

Note that $\{ x_a \cdot a \cdot x_a - x_a, a \cdot x_a \cdot a - a | a \in A\} \subseteq \ker(\widehat{f})$, for given any $\iota_A(a) = a, (\iota_{L(A)}\circ \eta_A)(a) = x_a \in A\{ x_a | a \in A\}$:

$$\widehat{f}(x_a \cdot a \cdot x_a - x_a) = \widehat{f}(x_a) \cdot \widehat{f}(a) \cdot \widehat{f}(x_a) - \widehat{f}(x_a) = f(a)^{*}\cdot f(a) \cdot f(a)^{*} - f(a)^{*} = 0$$

and

$$\widehat{f}(a \cdot x_a \cdot a - a) = \widehat{f}(a) \cdot \widehat{f}(x_a) \cdot \widehat{f}(a) - \widehat{f}(a) = f(a)\cdot f(a)^{*}\cdot f(a) - f(a) = 0.$$

By the \textbf{Theorem of Homomorphism}, there is a unique $\mathcal{C}^{\infty}-$homomorphism:

$$\overline{f}: \dfrac{A\{x_a | a \in A \}}{\langle \{ x_a \cdot a \cdot x_a - x_a, a \cdot x_a \cdot a - a | a \in A \}\rangle} \rightarrow V$$

such that the following diagram commutes:

$$\xymatrixcolsep{3pc}\xymatrix{
A\{ x_a | a \in A\} \ar[r]^{\widehat{f}} \ar[d]_{q} & V \\
\dfrac{A\{x_a | a \in A \}}{\langle \{ x_a \cdot a \cdot x_a - x_a, a \cdot x_a \cdot a - a | a \in A \}\rangle} \ar@{-->}[ur]_{\exists ! \overline{f}}
},$$

where $q: A\{ x_a | a \in A\} \to \dfrac{A\{x_a | a \in A \}}{\langle \{ x_a \cdot a \cdot x_a - x_a, a \cdot x_a \cdot a - a | a \in A \}\rangle}$ is the canonical quotient map.\\

Thus we obtain a $\mathcal{C}^{\infty}-$ring $\dfrac{A\{x_a | a \in A \}}{\langle \{ x_a \cdot a \cdot x_a - x_a, a \cdot x_a \cdot a - a | a \in A \}\rangle}$ together with a $\mathcal{C}^{\infty}-$homomorphism
$$\nu_{0,1} = q \circ \iota_A: A \to \dfrac{A\{x_a | a \in A \}}{\langle \{ x_a \cdot a \cdot x_a - x_a, a \cdot x_a \cdot a - a | a \in A \}\rangle}$$
that have the following universal property: given any von Neumann regular $\mathcal{C}^{\infty}-$ring $V$ and any $\mathcal{C}^{\infty}-$homomorphism $f: A \to V$, there is a unique $\mathcal{C}^{\infty}-$homo\-morphism $\overline{f}: \dfrac{A\{x_a | a \in A \}}{\langle \{ x_a \cdot a \cdot x_a - x_a, a \cdot x_a \cdot a - a | a \in A \}\rangle} \to V$ such that the following diagram commutes:

$$\xymatrixcolsep{3pc}\xymatrix{
A \ar[r]^(.25){\nu_{0,1}} \ar[dr]_{f} & \dfrac{A\{x_a | a \in A \}}{\langle \{ x_a \cdot a \cdot x_a - x_a, a \cdot x_a \cdot a - a | a \in A \}\rangle} \ar@{-->}[d]^{\exists ! \overline{f}}\\
 & V
}$$

We have the following:\\

\begin{lemma}\label{decio}Given a $\mathcal{C}^{\infty}-$ring $B$, let:
$$B^{+}: = \dfrac{B\{ x_b | b \in B \}}{\langle \{x_b \cdot b \cdot x_b - x_b, b \cdot x_b \cdot b - b | b \in B\} \rangle}$$
and $\nu_{0,1}:B \to B^{+}$ be the following $\mathcal{C}^{\infty}-$homomorphism:

$$\begin{array}{cccc}
    \nu_{0,1}: & B & \rightarrow & B^{+} \\
     & b & \mapsto & q(\iota_B(b))
  \end{array}$$

where $\iota_B: B \to B\{ x_b | b \in B\}$ is the inclusion map in the coproduct of $B$ and $L(B)$, and $q: B\{ x_b | b \in B\} \to \dfrac{B\{ x_b | b \in B \}}{\langle \{x_b \cdot b \cdot x_b - x_b, b \cdot x_b \cdot b - b | b \in B\} \rangle} $ is the canonical quotient map.\\

Given any von Neumann regular $\mathcal{C}^{\infty}-$ring $V$ and any $\mathcal{C}^{\infty}-$homomorphism $f: B \to V$, there is a unique $\mathcal{C}^{\infty}-$homomorphism $\overline{f}: B^{+} \to V$ such that the following diagram commutes:

$$\xymatrixcolsep{3pc}\xymatrix{
B \ar[r]^{\nu_{0,1}} \ar[dr]_{f} & B^{+} \ar@{-->}[d]^{\exists ! \overline{f}}\\
& V
}$$
\end{lemma}
\begin{proof}
The existence of such a $\mathcal{C}^{\infty}-$homomorphism $\overline{f}$ has already been established in the construction, so we need only to prove its uniqueness.\\

Let $f': \dfrac{B\{ x_b | b \in B \}}{\langle \{x_b \cdot b \cdot x_b - x_b, b \cdot x_b \cdot b - b | b \in B\} \rangle} \to V$ be a $\mathcal{C}^{\infty}-$homomorphism such that $\overline{f'} \circ \nu_{0,1} = 0$, that is, such that for every $b \in A$, $\overline{f'} \circ (q \circ \iota_B)(b) = f(b)$. This implies $((\overline{f'}\circ q \circ \iota_B)(b))^{*} = f(b)^{*}$.\\

Also, since $\overline{f'}$ is a $\mathcal{C}^{\infty}-$homomorphism, it follows that:

$$\overline{f'}(q(\iota_B(b))\cdot q(\iota_{L(B)}(b)) \cdot q(\iota_B(b)) - q(\iota_B(b))) = \overline{f'}(0) = 0$$
so

$$\overline{f'}(q(\iota_B(b)))\cdot \overline{f'}(q(\iota_{L(B)}(b))) \cdot \overline{f'}(q(\iota_B(b))) - \overline{f'}(q(\iota_B(b))) = 0$$

and

$$f(b) \cdot [\overline{f'}(q(\iota_B(b)))] \cdot f(b) - f(b) = 0.$$

Analogously we prove that:

$$[\overline{f'}(q(\iota_B(b)))] \cdot f(b) \cdot [\overline{f'}(q(\iota_B(b)))] - [\overline{f'}(q(\iota_B(b)))] = 0,$$

so
$$(\forall b \in B)(\overline{f'}(q(\iota_B(b)))= f(b)^{*})$$
and the following diagram commutes:

$$\xymatrixcolsep{3pc}
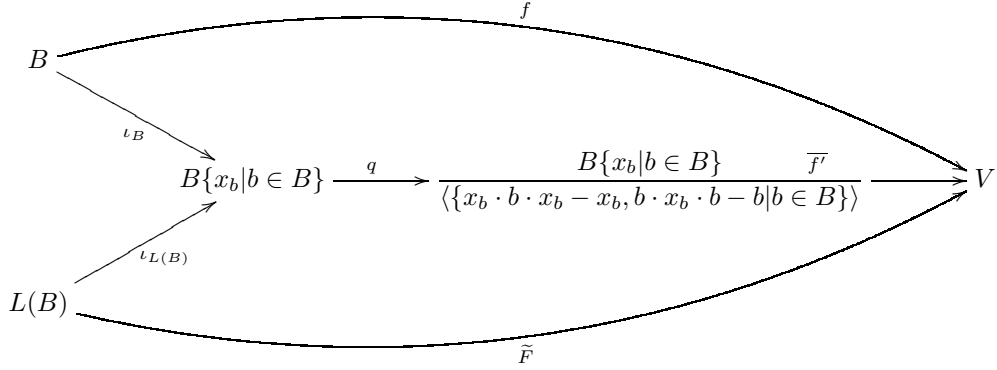
\xymatrix{
B \ar[dr]_{\iota_B} \ar@/^3pc/[rrrd]^{f} & & & \\
 & B\{ x_b | b \in B \} \ar[r]^(.3){q} & \dfrac{B\{ x_b | b \in B \}}{\langle \{ x_b \cdot b \cdot x_b - x_b, b \cdot x_b \cdot b - b | b \in B \} \rangle} \ar[r]^(0.5){\overline{f'}} & V \\
L(B) \ar[ur]_{\iota_{L(B)}} \ar@/_3pc/[rrru]_{\widetilde{F}} & & &
}$$

where for every $x_b \in L(B)$, $\widetilde{F}(x_b)=f(b)^*$.\\

Now, $\overline{f}$ is the unique $\mathcal{C}^{\infty}-$homomorphism such that $\overline{f}\circ q = \widehat{f}$, and $\widehat{f}$ is the unique $\mathcal{C}^{\infty}-$homomorphism such that both $\widehat{f}\circ \iota_B = f$ and $\widehat{f}\circ \iota_{L(B)} = \widetilde{F}$. Thus, since $\overline{f'}\circ q \circ \iota_B = f$ and $\overline{f'}\circ q \circ \iota_{L(B)} = \widetilde{F}$, it follows that $\overline{f} = \overline{f'}$.
\end{proof}

Let $A$ be any $\mathcal{C}^{\infty}-$ring, so we define:

$$\begin{cases}
    A_0:= A \\
    A_{n+1} := A_n^{+}
  \end{cases}$$

and for every $n \in \mathbb{N}$ we define $\nu_{n,n+1}: A_n \to A_n^{+}=A_{n+1}$ as:

$$\begin{array}{cccc}
    \nu_{n,n+1}: & A_n & \rightarrow & \dfrac{A_n\{ x_a | a \in A_n\}}{\langle \{a \cdot x_a \cdot a - a, x_a \cdot a \cdot x_a - x_a | a \in A_n\} \rangle} \\
     & a & \mapsto & \iota_{A_n}(a) + \langle \{a \cdot x_a \cdot a - a, x_a \cdot a \cdot x_a - x_a | a \in A_n\} \rangle
  \end{array}$$

For every $n,m \in \mathbb{N}$ such that $n<m$, we define:

$$\nu_{n,m} = \displaystyle\bigcirc_{i=n}^{m-1} \nu_{i,i+1}$$

$$\xymatrixcolsep{3pc}\xymatrix{
A_n \ar@/_/[drrrr]_{\nu_{n,m}}\ar[r]^{\nu_{n,n+1}} & A_{n+1} \ar[r]^{\nu_{n+1,n+2}} & A_{n+3} \ar[r]^{\nu_{n+3,n+4}} & \cdots \ar[r]^{\nu_{m-2,m-1}}& A_{m-1} \ar[d]^{\nu_{m-1,m}}\\
 & & & & A_m
}$$

so we have an inductive system, $(\{ A_n\}_{n \in \mathbb{N}}, \{ A_n \stackrel{\nu_{n,m}}{\rightarrow} A_m | n \leq m\}_{n,m \in \mathbb{N}})$.\\

Let $A_{\omega}:= \varinjlim_{n \in \mathbb{N}} A_n$ be the colimit of the above inductive system, so the following diagram commutes for every $n, m \in \mathbb{N}$ such that $n<m$:

$$\xymatrixcolsep{3pc}\xymatrix{
 & A_{\omega} & \\
A_n \ar[rr]^{\nu_{n,m}} \ar[ur]^{\nu_n} & & A_m \ar[ul]_{\nu_m}
}$$

We claim that $A_{\omega}$ is a von Neumann regular $\mathcal{C}^{\infty}-$ring.\\

Given $[(a,n)] \in A_{\omega}$, there is some $a \in A_n$ such that $\nu_n(a) = [(a,n)]$. By construction, $\nu_{n,n+1}(a) \in A_{n+1}$ has a unique quasi-inverse, $\nu_{n,n+1}(a)^* \in A_{n+1}$, that is:

$$\nu_{n,n+1}(a) \cdot \nu_{n,n+1}(a)^{*} \cdot \nu_{n,n+1}(a) = \nu_{n,n+1}(a)$$
and

$$\nu_{n,n+1}(a)^{*} \cdot \nu_{n,n+1}(a) \cdot \nu_{n,n+1}(a)^{*} = \nu_{n,n+1}(a)^{*}.$$

Take $\nu_{n+1}(\nu_{n,n+1}(a)^{*}) = [(\nu_{n,n+1}(a)^{*}, n+1)]$. We have $[(a,n)]^{*} = [(\nu_{n,n+1}(a)^{*}, n+1)]$.\\

In fact, $\nu_n(a) = [(a,n)] = [(\nu_{n,n+1}(a),n+1)] = \nu_n \circ \nu_{n,n+1}(a)$, so:

\begin{multline*}
[(\nu_{n,n+1}(a)^{*}, n+1)] \cdot [(a,n)] \cdot [(\nu_{n,n+1}(a)^{*}, n+1)] = \\
= [(\nu_{n,n+1}(a)^{*}, n+1)] \cdot [(\nu_{n,n+1}(a),n+1)] \cdot [(\nu_{n,n+1}(a)^{*}, n+1)] =\\
= [(\nu_{n, n+1}(a) \cdot \nu_{n,n+1}(a)^{*} \cdot \nu_{n,n+1}(a), n+1)] = [(\nu_{n,n+1}(a)^{*}, n+1)]
\end{multline*}
and

\begin{multline*}[(a,n)] \cdot [(\nu_{n,n+1}(a)^{*}, n+1)] \cdot [(a,n)] = \\
= [(\nu_{n,n+1}(a),n+1)] \cdot [(\nu_{n,n+1}(a)^{*}, n+1)] \cdot [(\nu_{n,n+1}(a),n+1)] =\\
= [(\nu_{n, n+1}(a) \cdot \nu_{n,n+1}(a)^* \cdot \nu_{n,n+1}(a), n+1)] = [(\nu_{n,n+1}(a), n+1)] = [(a,n)]
\end{multline*}

We register the following:\\

\begin{definition}Let $A$ be a $\mathcal{C}^{\infty}-$ring. The \textbf{\index{$\mathcal{C}^{\infty}-$von Neumann-regular hull}$\mathcal{C}^{\infty}-$von Neumann-regular hull of} $A$, is a von Neumann-regular $\mathcal{C}^{\infty}-$ring, denoted by ${\rm vN}\,(A)$, together with a $\mathcal{C}^{\infty}-$homomorphism $\nu_A : A \to {\rm vN}\,(A)$ with the following universal property: for every von Neumann-regular $\mathcal{C}^{\infty}-$ring, $B$ and any $\mathcal{C}^{\infty}-$homomorphism $\mu: A \to B$, there is a unique $\mathcal{C}^{\infty}-$homomorphism $\varphi: {\rm vN}\,(A) \to B$ such that the following triangle commutes:
$$\xymatrix{
A \ar[r]^{\nu_A} \ar[rd]^{\mu} & {\rm vN}\,(A) \ar@{.>}[d]^{\exists ! \varphi}\\
    & B}$$
\end{definition}

\begin{theorem}\label{bbis}Let $A$ be any $\mathcal{C}^{\infty}-$ring. The pair $(A_{\omega}, \nu_0: A \to A_{\omega})$ (where $\nu_0$ is the colimit arrow) is the $\mathcal{C}^{\infty}-$von Neumann regular hull of $A$, that is, for every von Neumann-regular $\mathcal{C}^{\infty}-$ring $V$ and for every $\mathcal{C}^{\infty}-$homomorphism $f: A \to V$ there is a unique $\mathcal{C}^{\infty}-$homomorphism $\widetilde{f}: A_{\omega} \rightarrow V$ such that the following diagram commutes:

$$\xymatrixcolsep{3pc}\xymatrix{
A \ar[r]^{\nu_0} \ar[dr]_{f} & A_{\omega} \ar@{-->}[d]^{\exists ! \widetilde{f}}\\
 & V
}$$
\end{theorem}
\begin{proof}
Given $f: A \to V$, by \textbf{Lemma \ref{decio}} there is a unique $f_1: A_1 \to V$ such that $f_1 \circ \nu_{0,1} = f$, that is, such that the diagram:

$$\xymatrixcolsep{3pc}\xymatrix{
A \ar[r]^{\nu_{0,1}} \ar[dr]_{f} & A_1 \ar@{-->}[d]^{\exists ! f_1}\\
  & V
}$$

commutes. By induction, one proves that for this fixed $f: A \to V$, for every $n \in \mathbb{N}$ there is a unique $f_n: A_n \to V$ such that

$$\xymatrixcolsep{3pc}\xymatrix{
A \ar[r]^{\nu_{0,n}} \ar[dr]_{f} & A_n \ar@{-->}[d]^{\exists ! f_n}\\
  & V
}$$

commutes, that is, such that  $f_n \circ \nu_{0,n} = f$, so the for every $n,m \in \mathbb{N}$ with $n<m$, the following diagram commutes:

$$\xymatrixcolsep{3pc}\xymatrix{
 & V & \\
A_n \ar[ur]^{f_n} \ar[rr]^{\nu_{n,m}} & & A_m \ar[ul]_{f_m}
}.$$

By the universal property of the colimit $A_{\omega}$, there is a unique $\mathcal{C}^{\infty}-$homomorphism:

$$\widehat{f}: A_{\omega} \rightarrow V$$

such that the following diagram commutes:

$$\xymatrixcolsep{3pc}\xymatrix{
 & V & \\
 & A_{\omega} \ar@{-->}[u]^{\exists ! \widetilde{f}}& \\
A_n \ar[ur]^{\nu_n} \ar@/^2pc/[ruu]^{f_n} \ar[rr]_{\nu_{n,m}}& & A_m \ar[ul]_{\nu_m}\ar@/_2pc/[uul]^{f_m}}$$

for every $n,m \in \mathbb{N}$ such that $n<m$. In particular, for $n=0$, we have the following commutative diagram:

$$\xymatrixcolsep{3pc}\xymatrix{
A \ar[r]^{\nu_{0}} \ar[dr]_{f} & A_{\omega} \ar[d]^{\widetilde{f}}\\
  & V
}$$

\end{proof}

\begin{remark}Whenever there is some $n \in \mathbb{N}$ such that $A_n$ is a von Neumann regular $\mathcal{C}^{\infty}-$ring, we have ${\rm vN}\,(A)=A_n$ and $\nu = \nu_{0,n}: A \to A_n$.
\end{remark}

\begin{theorem}\label{Tadios}The forgetful functor:
$$U: \mathcal{C}^{\infty}{\rm \bf vNRng} \rightarrow {\rm \bf Set}$$
has a left-adjoint, $L: {\rm \bf Set} \rightarrow \mathcal{C}^{\infty}{\rm \bf vNRng}$.
\end{theorem}
\begin{proof}
  By \textbf{Proposition \ref{proposition8}}, the inclusion functor:

  $$\iota: \mathcal{C}^{\infty}{\rm \bf vNRng} \rightarrow \mathcal{C}^{\infty}{\rm \bf Rng}$$

has a left adjoint, $\nu: \mathcal{C}^{\infty}{\rm \bf Rng} \rightarrow \mathcal{C}^{\infty}{\rm \bf vNRng}$. \\

By \textbf{Proposition 11}, p.47 of \cite{BM1}, the forgetful functor $U': \mathcal{C}^{\infty}{\rm \bf Rng} \rightarrow {\rm \bf Set}$ has a left adjoint, $L': {\rm \bf Set} \rightarrow \mathcal{C}^{\infty}{\rm \bf Rng}$. \\

Since $U$ is the composition of the forgetful functor $U': \mathcal{C}^{\infty}{\rm \bf vNRng} \rightarrow \mathcal{C}^{\infty}{\rm \bf Rng}$ with the inclusion $\iota: \mathcal{C}^{\infty}{\rm \bf Rng} \to {\rm \bf Set}$, that is, $U = U' \circ \iota$, it suffices to define $L:= \nu \circ L': {\rm \bf Set} \rightarrow \mathcal{C}^{\infty}{\rm \bf vNRng}$, so $L \dashv U$.

$$\xymatrixcolsep{3pc}\xymatrix{
\mathcal{C}^{\infty}{\rm \bf Rng} \ar@<1ex>[rr]^{\nu} \ar[dr]_{U'} & & \mathcal{C}^{\infty}{\rm \bf vNRng} \ar[ll]^{\iota} \ar[dl]^{U}\\
 & {\rm \bf Set} \ar@<-1ex>[ul]_{L'} \ar@<1ex>[ur]^{L}&
}$$
\end{proof}

\begin{remark}\label{description}Let $A$ and $B$ be two $\mathcal{C}^{\infty}-$rings, $f: A \to B$ a $\mathcal{C}^{\infty}-$homomorphism, $\p \in {\rm Spec}^{\infty}\,(A)$ and $\mathfrak{q} \in {\rm Spec}^{\infty}\,(B)$. \\

We write $\pi^{A}_{\p}: A \twoheadrightarrow \dfrac{A}{\p}$ for the quotient homomorphism and $\alpha^{A}_{\p}: A \to k_{\p}(A)$ for the composition $A \twoheadrightarrow \dfrac{A}{\p} \rightarrowtail k_{\p}(A)$, where $k_{\p}(A)$ is the residue field of the local ring $A\{ {A \setminus \p}^{-1}\}$.\\

Consider the quotient maps $\pi_{f^{*}(\mathfrak{q})}: A \twoheadrightarrow \frac{A}{f^{*}(\mathfrak{q})}$ and $\pi_{\mathfrak{q}}: B \twoheadrightarrow \frac{B}{\mathfrak{q}}$. Since $\ker (\pi_{\mathfrak{q}} \circ f) = f^{\dashv}[\pi_{\mathfrak{q}}^{\dashv}[\{ 0 + \mathfrak{q} \}]] = f^{\dashv}[\mathfrak{q}] = f^{*}(\mathfrak{q})$, by the \textbf{Theorem of the Isomorphism} there exists a unique $\overline{f_{\mathfrak{q}}}: \frac{A}{f^{*}(\mathfrak{q})} \to \frac{B}{\mathfrak{q}}$ such that the following diagram commutes:
$$\xymatrixcolsep{3pc}\xymatrix{
A \ar@{->>}[d]^{\pi_{f^{*}(\mathfrak{q})}} \ar@{->>}[r]^{\pi_{\mathfrak{q}}\circ f} & \dfrac{B}{\mathfrak{q}}\\
\dfrac{A}{f^{*}(\mathfrak{q})} \ar@{>->}[ur]^{\overline{f_{\mathfrak{q}}}} &
}$$

Moreover, $\overline{f_{\mathfrak{q}}}$ is a monomorphism, since $\ker \overline{f_{\mathfrak{q}}} = \{ 0 + f^{*}(\mathfrak{q})\}$. We call $\overline{f_{\mathfrak{q}}}$ ``the quotient monomorphism between the associated $\mathcal{C}^{\infty}-$domains''. \\

Consider the localization maps $\eta_{f^{*}(\mathfrak{q})}: A \to A\{ {A \setminus f^{*}(\mathfrak{q})}^{-1}\}$ and $\eta_{\mathfrak{q}}: B \to B\{ {B \setminus \mathfrak{q}}^{-1}\}$. Since $\eta_{\mathfrak{q}}[f[A\setminus f^{\star}(\mathfrak{q})]] \subseteq \eta_{\mathfrak{q}}[B \setminus \mathfrak{q}] \subseteq (B\{ B \setminus \mathfrak{q}\}^{-1})^{\times}$, by the universal property of the map $\eta_{f^{*}(\mathfrak{q})}: A \to A\{ {f^{*}(\mathfrak{q})}^{-1}\}$ there is a unique $\check{f_{\mathfrak{q}}}: A\{ {A \setminus f^{\star}(\mathfrak{q})}^{-1}\} \to B\{ {B \setminus \mathfrak{q}}^{-1}\}$ such that the following square commutes:
$$\xymatrixcolsep{3pc}\xymatrix{
A \ar[r]^{\eta_{f^{*}(\mathfrak{q})}} \ar[d]_{f} & A\{ {A \setminus f^{*}(\mathfrak{q})}^{-1} \} \ar@{.>}[d]^{\check{f_{\mathfrak{q}}}}\\
B \ar[r]^{\eta_{\mathfrak{q}}} & B\{ {B \setminus \mathfrak{q}}^{-1} \}
}$$

We call $\check{f_{\mathfrak{q}}}$ ``the canonical local $\mathcal{C}^{\infty}-$homomorphism between the associated local $\mathcal{C}^{\infty}-$rings''.\\

Let $k_{\mathfrak{q}}(B) = {\rm Res}\,\left( B\{ {B \setminus \mathfrak{q}}^{-1} \}\right)$ and $k_{f^{*}(\mathfrak{q})}(A) = {\rm Res}\, \left( A \{ {A \setminus f^{*}(\mathfrak{q})}^{-1}\}\right)$ and consider the canonical $\mathcal{C}^{\infty}-$homomorphisms $B\{ {B \setminus \mathfrak{q}}^{-1}\} \twoheadrightarrow k_{\mathfrak{q}}(B)$ and $A\{ {A \setminus f^{*}(\mathfrak{q})}^{-1}\} \twoheadrightarrow k_{f^{*}(\mathfrak{q})}(A)$. We have already seen that both $A \{ {A \setminus f^{*}(\mathfrak{q})}^{-1}\}$ and $B\{ {B \setminus \mathfrak{q}}^{-1} \}$ are local $\mathcal{C}^{\infty}-$rings, so let $\m_{f^{*}(\mathfrak{q})}$ and $\m_{\mathfrak{q}}$ be the (unique) maximal ideals of $A \{ {A \setminus f^{*}(\mathfrak{q})}^{-1}\}$ and $B\{ {B \setminus \mathfrak{q}}^{-1} \}$, respectively. We have:
$$k_{\mathfrak{q}}(B) = {\rm Res}\,\left( B\{ {B \setminus \mathfrak{q}}^{-1} \}\right) = \dfrac{B\{ {B \setminus \mathfrak{q}}^{-1} \}}{\m_{\mathfrak{q}}}$$
and
$$k_{f^{*}(\mathfrak{q})}(A) = {\rm Res}\, \left( A \{ {A \setminus f^{*}(\mathfrak{q})}^{-1}\}\right) = \dfrac{A \{ {A \setminus f^{*}(\mathfrak{q})}^{-1}\}}{\m_{f^{*}(\mathfrak{q})}}$$

By the \textbf{Theorem of the Homomorphism}, since $\check{f_{\mathfrak{q}}}^{\dashv}[\m_{\mathfrak{q}}] = \m_{f^{*}(\mathfrak{q})}$ (for $\m_{f^{*}(\mathfrak{q})}$ is a local $\mathcal{C}^{\infty}-$rings homomorphism), there is a unique $\mathcal{C}^{\infty}-$homomorphism $\widehat{f_{\mathfrak{q}}}: k_{f^{*}(\mathfrak{q})}(A) \to k_{\mathfrak{q}}(B)$ such that the following diagram commutes:
$$\xymatrixcolsep{3pc}\xymatrix{
A\{ {A \setminus f^{*}(\mathfrak{q})}^{-1}\} \ar[d]^{q_{f^{*}(\mathfrak{q})}} \ar[r]^{q_{\mathfrak{m}_{\mathfrak{q}}} \circ \check{f_{\mathfrak{q}}}} & \dfrac{B\{ {B \setminus \mathfrak{q}}^{-1}\}}{\mathfrak{m}_{\mathfrak{q}}}\\
\dfrac{A\{ {A \setminus f^{*}(\mathfrak{q})}^{-1}\}}{\mathfrak{m}_{f^{*}(\mathfrak{q})}} \ar@{.>}[ur]^{ \exists ! \widehat{f_{\mathfrak{q}}}}
}$$

%%%%%%%%%%%%%%%%%%%%%%%%%%%%%%%%%%%%%%%%%%%%%%%%%%%%%%%%

 Thus, given a $\mathcal{C}^{\infty}-$homomorphism $f: A \to B$ we have a ``canonical monomorphism'' between the associated fields, that will be denoted by:
$$\widehat{f_{\mathfrak{q}}}: k_{f^{*}(\mathfrak{q})}(A) \rightarrow k_{\mathfrak{q}}(B).$$

 We will write:

 $$\widehat{f}: \prod_{\p \in {\rm Spec}^{\infty}\,(A)} k_{\p}(A) \to \prod_{\mathfrak{q} \in {\rm Spec}^{\infty}\,(B)} k_{\mathfrak{q}}(B)$$

for the unique $\mathcal{C}^{\infty}-$homomorphism such that for every $\mathfrak{q} \in {\rm Spec}^{\infty}\,(B)$ the following diagram commutes:

$$\xymatrixcolsep{3pc}\xymatrix{
\prod_{\p \in {\rm Spec}^{\infty}\,(A)} k_{\p}(A) \ar[r]^{\widehat{f}} \ar[d]_{{\rm proj}^{A}_{f^{*}(\mathfrak{q})}} & \prod_{\mathfrak{q} \in {\rm Spec}^{\infty}\,(B)} k_{\mathfrak{q}}(B) \ar[d]^{{\rm proj}^{B}_{\mathfrak{q}}}\\
k_{f^{*}(\mathfrak{q})}(A) \ar[r]^{\widehat{f_{\mathfrak{q}}}} & k_{\mathfrak{q}}(B)
}$$

\end{remark}

%%%%%%%%%%%%%%%%%%%%%%%%%%%

%%%%%%%%%%%%%%%%%%%%%%%%%%%%%%%%%%%%%%%%%%%%%%%%%%%

\begin{theorem}\label{11i}Let $A$ be a $\mathcal{C}^{\infty}-$ring and $(\nu_A, {\rm vN}\,(A))$ be its von Neumann-regular hull. Then:
\begin{itemize}
  \item[(i)]{$(\nu_A)^{*}: {\rm Spec}^{\infty}\,({\rm vN}\,(A)) \to {\rm Spec}^{\infty}\,(A)$ is a surjective spectral function;}
  \item[(ii)]{For any $\mathfrak{q} \in {\rm Spec}^{\infty}\,({\rm vN}\,(A))$, $\widehat{\nu_A}: k_{(\nu_A)^{*}(\mathfrak{q})}(A) \to k_{\mathfrak{q}}({\rm vN}\,(A))$ is a $C^\infty$-ring isomorphims.}
  \item[(iii)]{$(\nu_A)^{*}: {\rm Spec}^{\infty}\,({\rm vN}\,(A)) \to {\rm Spec}^{\infty-{\rm const}}\,(A)$ is a homeomorphism, where ${\rm Spec}^{\infty-{\rm const}}\,(A)$ is the boolean topological space on the set of prime $\mathcal{C}^{\infty}-$radical ideals of $A$ with the constructible topology.}
  \item[(iv)]{$\ker (\nu_A) = \sqrt[\infty]{(0)} = \bigcap {\rm Spec}^{\infty}\,(A)$. Thus $\nu_A$ is injective if and only if, $A$ is $\mathcal{C}^{\infty}-$reduced.}
\end{itemize}
\end{theorem}
\begin{proof}
Ad (i): Take $\mathfrak{p} \in {\rm Spec}^{\infty}\,(A)$ and consider the canonical $\mathcal{C}^{\infty}-$homomorphism $\alpha_{\mathfrak{p}}: A \to k_{\mathfrak{p}}(A)$, where $k_{\mathfrak{p}}(A)$ is a von Neumann-regular $\mathcal{C}^{\infty}-$ring, for it is a $\mathcal{C}^{\infty}-$field. By the universal property of $\nu_A : A \to {\rm vN}\,(A)$ we have a unique ``extension'' $\widetilde{\alpha}^{A}_{\mathfrak{p}} : {\rm vN}\,(A) \to k_{\mathfrak{p}}(A)$ such that:
$$\xymatrix{
A \ar[r]^{\nu_A} \ar[rd]_{\alpha^{A}_{\mathfrak{p}}} &  {\rm vN}\,(A) \ar[d]^{\widetilde{\alpha}^{A}_{\mathfrak{p}}}\\
    & k_{\mathfrak{p}}(A)}$$
commutes, \textit{i.e.}, $\widetilde{\alpha^{A}_{\mathfrak{p}}} \circ \nu_A = \alpha^{A}_{\mathfrak{p}}$. Now take $\gamma_A(\mathfrak{p}) := (\widetilde{\alpha^{A}_{\mathfrak{p}}}^{*})[\{ 0\}] \in {\rm Spec}^{\infty}\,({\rm vN}\,(A))$. In this way we have $(\nu_A)^{*}(\gamma_A(\mathfrak{p})) = (\nu_A)^{*}((\widetilde{\alpha_{\mathfrak{p}}})^{*})(\{ 0\}) = (\widetilde{\alpha_{\mathfrak{p}}}\circ \nu_A)^{*}(\{ 0\}) = \mathfrak{p}$, thus $(\nu_A)^{*} \circ \gamma_A = {\rm id}_{{\rm Spec}^{\infty}\,(A)}$ is surjective.\\

Ad (ii): Let $\mathfrak{q} \in {\rm Spec}^{\infty}\,({\rm vN}\,(A))$, by definition $\widehat{\nu_A}$ we have:
$$\alpha^{{\rm vN}\,(A)}_{\mathfrak{q}} \circ \nu_A = (\widehat{\nu_A})_{\mathfrak{q}} \circ \alpha^{A}_{\nu^{*}(\mathfrak{q})}$$
so we have the $\mathcal{C}^{\infty}-$field monomorphism:
$$\widehat{\eta_A}: k_{\nu^{*}_{A}(\mathfrak{q})}(A) \rightarrow k_{\mathfrak{q}}({\rm vN}\,(A)).$$

Let us prove that it is surjective: consider the extension of $\alpha^{A}_{\nu^{*}_{A}(\mathfrak{q})}$ to ${\rm vN}\,(A)$: $\widetilde{\alpha^{{\rm vN}\,(A)}_{\nu^{*}_{A}(\mathfrak{q})}} : {\rm vN}\,(A) \to k_{\nu^{*}_{A}(\mathfrak{q})}(A)$, then $\widehat{\nu_{A_{\mathfrak{q}}}} \circ \widetilde{\alpha}^{A}_{\nu^{*}_{A}(\mathfrak{q})} \circ \nu_A = \widehat{\nu_{A_{\mathfrak{q}}}} \circ \alpha^{A}_{\nu^{*}_{A}(\mathfrak{q})} = \alpha^{{\rm vN}\,(A)}_{\mathfrak{q}} \circ \nu_A$, thus $\widehat{A_{\mathfrak{q}}} \circ \widetilde{\alpha}^{A}_{\nu^{*}_{A}(\mathfrak{q})} = \alpha^{{\rm vN}\,(A)}_{\mathfrak{q}}$, by the universal property of $\nu_A$. But $\mathfrak{q} \in {\rm Spec}^{\infty}\,({\rm vN}\,(A))$ is a \textbf{maximal} ideal, so $\frac{{\rm vN}\,(A)}{\mathfrak{q}} \cong k_{\mathfrak{q}}({\rm vN}\,(A))$ and $\alpha^{{\rm vN}\,(A)}_{\mathfrak{q}}: {\rm vN}\,(A) \twoheadrightarrow \dfrac{{\rm vN}\,(A)}{\mathfrak{q}} \to k_{\mathfrak{q}}({\rm vN}\,(A))$ is \textbf{surjective}, therefore $\widehat{\nu_{A_{\mathfrak{q}}}}$ is surjective too, since $\widehat{\nu_{A_{\mathfrak{q}}}} \circ \widetilde{\alpha}^{A}_{\nu_A^{\star}(\mathfrak{q})} = \alpha^{{\rm vN}\,(A)}_{\mathfrak{q}}$.\\

Ad (iii): Since $(\nu_A)^{\star}$ is a map between Boolean spaces and it is a spectral map, in order to prove that $(\nu_A)^{\star}: {\rm Spec}^{\infty}\, ({\rm vN}\,(A)) \to {\rm Spec}^{\infty-{\rm const}}\,(A)$ is a homeomorphism, it is necessary and sufficient to prove that the spectral map $(\nu_A)^{\star}: {\rm Spec}^{\infty}\,({\rm vN}\,(A)) \to {\rm Spec}^{\infty}\,(A)$ is a bijection from the Boolean space ${\rm Spec}^{\infty}\,({\rm vN}\,(A))$ to the spectral space ${\rm Spec}^{\infty}\,(A)$. Keeping the notation in the proof of item (i), we will show that $\gamma_A$ is the inverse map of $\nu^{\star}_{A}$. By the proof of (i), it is enough to prove that $\gamma \circ (\nu_A)^{\star} = {\rm id}_{{\rm Spec}^{\infty}\,({\rm vN}\,(A))}$. Let $\mathfrak{q} \in {\rm Spec}^{\infty}\,({\rm vN}\,(A))$, then $\gamma_A(\nu^{\star}_{A}(\mathfrak{q})) = \ker (\widetilde{\alpha}_{\nu^{\star}_{A}(\mathfrak{q})}) = \ker (\alpha^{{\rm vN}\,(A)}_{\mathfrak{q}}) = \mathfrak{q}$, since $\widetilde{\nu_{A_{\mathfrak{q}}}}\circ \widetilde{\alpha}^{A}_{\nu^{\star}_{A}(\mathfrak{q})} = \alpha^{{\rm vN}\,(A)}_{\mathfrak{q}}$ and $\widehat{\nu_{A_{\mathfrak{q}}}}$ is injective.\\

Ad (iv): We will see that the result follows from the fact that ${\rm Spec}^{\infty}\,({\rm vN}\,(A))$ is homeomorphic to the booleanization of ${\rm Spec}^{\infty}\,(A)$. Take any $\mathcal{C}^{\infty}-$ring $B$ and consider the ``diagonal'' $\mathcal{C}^{\infty}-$homomorphism:
$$\delta_B := (\alpha^{B}_{\mathfrak{p}})_{\mathfrak{p} \in {\rm Spec}^{\infty}\,(B)}: B \to \prod_{\mathfrak{p} \in {\rm Spec}^{\infty}\,(B)} k_{\mathfrak{p}}(B).$$

From the equality of morphisms
\begin{multline*}(\alpha^{B}_{\mathfrak{p}})_{\mathfrak{p} \in {\rm Spec}^{\infty}\,(B)}: B \to \prod_{\mathfrak{p} \in {\rm Spec}^{\infty}\,(B)} k_{\mathfrak{p}}(B) = \\
= B \stackrel{(\pi^{B}_{\mathfrak{p}})_{\mathfrak{p} \in {\rm Spec}^{\infty}\,(B)}}{\longrightarrow} \prod_{\mathfrak{p} \in {\rm Spec}^{\infty}\,(B)} \dfrac{B}{\mathfrak{p}} \rightarrowtail \prod_{\mathfrak{p} \in {\rm Spec}^{\infty}\,(B)} k_{\mathfrak{p}}(B)
\end{multline*}

we have $\ker (\delta_B) = \bigcap {\rm Spec}^{\infty}\,(B) = \sqrt[\infty]{(0)}$.\\

In particular, for $B = {\rm vN}\,(A)$, we have $\ker (\delta_{{\rm vN}\,(A)}) = \sqrt[\infty]{(0)}=(0)$ and therefore obtain that $\delta_{{\rm vN}\,(A)}: {\rm vN}\,(A) \to \prod_{\mathfrak{q} \in {\rm Spec}^{\infty}\,({\rm vN}\,(A))}k_{\mathfrak{q}}({\rm vN}\,(A))$ is \textbf{injective}.\\

As $\nu^{\star}_{A}: {\rm Spec}^{\infty}\,({\rm vN}\,(A)) \to {\rm Spec}^{\infty}\,(A)$ is bijective we get that the arrow $\widehat{\nu_A}: \prod_{\mathfrak{p} \in {\rm Spec}^{\infty}\,(A)} k_{\mathfrak{p}}(A) \to \prod_{\mathfrak{q} \in {\rm Spec}^{\infty}\,({\rm vN}\,(A))} k_{\mathfrak{q}}({\rm vN}\,(A))$ is \textbf{isomorphic to} the $\mathcal{C}^{\infty}-$homomorphism
$$(\widehat{\nu_{A_{\mathfrak{q}}}})_{\mathfrak{q} \in {\rm Spec}^{\infty}\,({\rm vN}\,(A))}: \prod_{\mathfrak{q} \in {\rm Spec}^{\infty}\,({\rm vN}\,(A))} k_{\nu^{\star}_{A}(\mathfrak{q})}(A) \to \prod_{\mathfrak{q} \in {\rm Spec}^{\infty}\,({\rm vN}\,(A))} k_{\mathfrak{q}}({\rm vN}\,(A))$$. By the previous items, $(\widehat{\nu_{A_{\mathfrak{q}}}})_{\mathfrak{q} \in {\rm Spec}^{\infty}\,({\rm vN}\,(A))}$ is an isomorphism, thus $\widehat{\nu_A}$ is an isomorphism too. As $\widehat{\nu_A}\circ \delta_A = \delta_{{\rm vN}\,(A)} \circ \nu_A$ and $\widehat{\nu_A}, \delta_{{\rm vN}\,(A)}$ are injective, we have $\ker (\nu_A) = \ker (\nu_A \circ \delta_{{\rm vN}\,(A)}) = \ker (\widehat{\nu_A}\circ \delta_A) = \ker(\delta_A) = \sqrt[\infty]{(0)}$.

$$\xymatrix{
A \ar[r]^{\nu_A} \ar[rd]_{\alpha^{A}_{\mathfrak{p}}} & {\rm vN}\,(A) \ar[dr]^{\alpha^{{\rm vN}\,(A)}_{\mathfrak{q}}} \ar[d]_{\widetilde{\alpha}^{A}_{\mathfrak{p}}} &   \\
  & k_{\mathfrak{p}}(A) \ar[r]_{\widehat{\nu_{A_{\mathfrak{q}}}}} & k_{\mathfrak{q}}({\rm vN}\,(A))
}$$

where $\mathfrak{p}:= \nu^{\star}_{A}(\mathfrak{q})$.
\end{proof}

%%%%%%%%%%%%%%%%%%%%%%%%%%%

\begin{theorem}\label{proposition12i}Suppose that we have a functor $R: \mathcal{C}^{\infty}{\rm \bf Rng} \to \mathcal{C}^{\infty}{\rm \bf vNRng}$ and a natural transformation $(\eta_A)_{A \in {\rm Obj}\,(\mathcal{C}^{\infty}{\rm \bf Rng})} = (\eta_A : A \to R(A))_{A \in {\rm Obj}\,(\mathcal{C}^{\infty}{\rm \bf Rng})}$.\\
\begin{itemize}
  \item[(i)]{Suppose that the following condition is satisfied:\\

`` \textbf{(E)} For each von Neumann-regular $\mathcal{C}^{\infty}-$ring $V$ the arrow $\eta_V : V \to R(V)$ is a section (i.e., it has a left inverse).''\\

Then every $\mathcal{C}^{\infty}-$homomorphism $f: A \to V$ to a von Neumann-regular $\mathcal{C}^{\infty}-$ring factors through $\eta_A$.}
  \item[(ii)]{Suppose the following conditions are satisfied:\\

  ``\textbf{(U)} ${\eta_A}^{*}: {\rm Spec}^{\infty}\,(R(A)) \to {\rm Spec}^{\infty}\,(A)$ is a bijection, i.e., ${\rm Spec}^{\infty}\,(R(A))$ is homeomorphic to the booleanization of ${\rm Spec}^{\infty}\,(A)$.''\\

  ``\textbf{(U')} The stalk of the spectral sheaf of $R(A)$ at a prime ideal $\mathfrak{p}$ ``in $A$'' is isomorphic to $k_{\mathfrak{p}}(A)$, more precisely $\widehat{\eta_{A_{\mathfrak{q}}}}: k_{{\eta_A}^{*}(\mathfrak{q})}(A) \to k_{\mathfrak{q}}(R(A))$ is an isomorphism, $\mathfrak{q} \in {\rm Spec}^{\infty}\,(R(A))$.''\\

Then a $\mathcal{C}^{\infty}-$homomorphism $f:A \to V$ to a von Neumann-regular $\mathcal{C}^{\infty}-$ring $V$ admits at most one factorization through $\eta_A$.  }
 \end{itemize}

Thus, if all of \textbf{(E)}, \textbf{(U)} and \textbf{(U')} are satisfied, then the map $\eta_A$ is initial among maps to von Neumann-regular $\mathcal{C}^{\infty}-$rings, \textit{i.e.}, a $\mathcal{C}^{\infty}-$homomorphism $f:A \to V$ to a von Neumann-regular $\mathcal{C}^{\infty}-$ring factors uniquely through $\eta_A$.
\end{theorem}
\begin{proof}
Ad (i): Let $V$ be a von Neumann-regular $\mathcal{C}^{\infty}-$ring and $f: A \to V$ a $\mathcal{C}^{\infty}-$homomorphism. As $(\eta_A)_{A \in {\rm Obj}\,(\mathcal{C}^{\infty}{\rm Rng})}$ is a natural tranformation, we have $R(f)\circ \eta_A = \eta_V \circ f$. By hypothesis \textbf{(E)} there is a $\mathcal{C}^{\infty}-$homomorphism $r_V: R(V) \to V$ such that $r_V \circ \eta_V = {\rm id}_V$. Now \textit{define} $F: R(A) \to V$ as the composition $F:= r_V \circ R(f)$, which makes the following diagram commute:
$$\xymatrix{
A \ar[d]^{\eta_A} \ar[r]^{f}  & V\\
R(A) \ar[ur]^{F} \ar[r]_{R(f)} & R(V) \ar[u]^{r_V}
}$$

Indeed,
$$F \circ \eta_A := (r_V \circ R(f)) \circ \eta_A = r_V \circ (R(f) \circ \eta_A) = r_V \circ (\eta_V \circ f) = (r_V \circ \eta_V) \circ f = {\rm id}_{V} \circ f = f,$$

as we claimed.\\

Ad (ii): Let $V$ be a von Neumann-regular $\mathcal{C}^{\infty}-$ring and let $F_0,F_1: R(A) \to V$ be two factorizations of $f$ through $\eta_A$, \textit{i.e.}, $\mathcal{C}^{\infty}-$homomorphisms such that $F_0 \circ \eta_A = f = F_1 \circ \eta_A$, or such that the following diagram commutes:
$$\xymatrix{  R(A) \ar@<2pt> [r]^{F_0}
\ar@<-2pt> [r]_{F_1} & V \\
A \ar[u]^{\eta_A} \ar[ur]_{f} &   }$$

Then we get $\widehat{F_0} \circ \widehat{\eta_A} = \widehat{f} = \widehat{F_1} \circ \widehat{\eta_A}$ (to see that just compose these $\mathcal{C}^{\infty}-$homo\-morphisms with ${\rm proj}^{V}_{\mathfrak{s}}, \mathfrak{s} \in {\rm Spec}^{\infty}\,(V)$ and note that both $\widehat{F_0}\circ \widehat{\eta_A}$ and $\widehat{F_1}\circ \widehat{\eta_A}$ have the property which characterizes $\widehat{f}$).\\

$$\xymatrixcolsep{3pc}\xymatrix{
\prod_{\mathfrak{p} \in {\rm Spec}^{\infty}\,(A)} k_{\mathfrak{p}}(A) \ar@/^2pc/[rr]^{\widehat{f}} \ar[d]^{{\rm proj}^{A}_{\mathfrak{p}}} \ar[r]^{\widehat{\eta_A}} & \prod_{\mathfrak{q} \in {\rm Spec}^{\infty}\,(R(A))} k_{\mathfrak{q}}(R(A)) \ar[d]^{{\rm proj}^{R(A)}_{\mathfrak{q}}} \ar@<2pt> [r]^{\widehat{F_0}}
\ar@<-2pt> [r]_{\widehat{F_1}} & \prod_{\mathfrak{s} \in {\rm Spec}^{\infty}\,(V)}k_{\mathfrak{s}}(V) \ar[d]^{{\rm proj}^{V}_{\mathfrak{s}}}\\
k_{f^{*}(\mathfrak{q})}(A) \ar@/_2pc/[rr]_{f} \ar[r]^{\eta_A} & k_{\mathfrak{q}}(R(A)) \ar@<2pt> [r]^{F_0}
\ar@<-2pt> [r]_{F_1} & k_{\mathfrak{s}}(V)
}$$

From the hypotheses \textbf{(U)} and \textbf{(U')} we obtain that the arrow:
$$\widehat{\eta_A} : \prod_{\mathfrak{p} \in {\rm Spec}^{\infty}\,(A)}k_{\mathfrak{p}}(A) \to \prod_{\mathfrak{q}\in {\rm Spec}^{\infty}\,(R(A))} k_{\mathfrak{q}}(R(A))$$
is an \textit{isomorphism} and therefore $\widehat{F_0} = \widehat{F_1}$. It follows that $(\alpha^{V}_{\mathfrak{p}})_{\mathfrak{p} \in {\rm Spec}^{\infty}\,(V)} \circ F_0 = \widehat{F_0} \circ (\alpha^{R(A)}_{\mathfrak{q}})_{\mathfrak{q} \in {\rm Spec}^{\infty}\,(R(A))} = \widehat{F_1} \circ (\alpha^{R(A)}_{\mathfrak{q}})_{\mathfrak{q} \in {\rm Spec}^{\infty}\,(R(A))} = (\alpha^{V}_{\mathfrak{p}})_{\mathfrak{p} \in {\rm Spec}^{\infty}\,(V)} \circ F_1.$\\

As $V$ is $\infty$-reduced, we get $\ker (\alpha^{V}_{\mathfrak{p}})_{\mathfrak{p} \in {\rm Spec}^{\infty}\,(V)} = \sqrt[\infty]{(0)} = (0)$, so $(\alpha^{V}_{\mathfrak{p}})_{\mathfrak{p}\in {\rm Spec}^{\infty}\,(V)}$ is injective and, by the commutativity of the following diagram, we can conclude that $F_0 = F_1$, proving the uniqueness of the extensions.\\

$$\xymatrixcolsep{8pc}\xymatrix{
R(A) \ar@<2pt>[d]^{F_0} \ar@<-2pt>[d]_{F_1} \ar[r]^{(\alpha^{R(A)}_{\mathfrak{q}})_{\mathfrak{q} \in {\rm Spec}^{\infty}\,(R(A))}} & \displaystyle\prod_{\mathfrak{q} \in {\rm Spec}^{\infty}\,(R(A))} k_{\mathfrak{q}}(R(A)) \ar@<2pt>[d]^{\widehat{F_0}} \ar@<-2pt>[d]_{\widehat{F_1}}\\
V \ar[r]^{(\alpha^{V}_{\mathfrak{p}})_{\mathfrak{p} \in {\rm Spec}^{\infty}\,(V)}} & \displaystyle\prod_{\mathfrak{s} \in {\rm Spec}^{\infty}\,(V)}k_{\mathfrak{s}}(V)} $$
\end{proof}

\subsection{Sheaves for the von Neumann regular hull of $\mathcal{C}^{\infty}-$ring}

In this subsection, we present another construction of the vN-hull of a $C^\infty$-ring by a sheaf-theoretic method. In fact, this alternative construction is useful to obtain a preservation result of the vN-hull functor (\textbf{Proposition \ref{prp29MA}}) and should be useful to proceed with other calculations.\\

The major part of the work here (that we will do in details) is: for each  $\mathcal{C}^{\infty}-$ring, $A$, we are going to build a presheaf $P_A$ on the basis of the topology of $({\rm Spec}^{\infty-{\rm const}}\,(A),{\rm Zar}^{\infty})$ such that  its stalks at each $\mathfrak{p} \in {\rm Spec}^{\infty}\,(A)$ are the $\mathcal{C}^{\infty}-$fields $k_{\mathfrak{p}}(A)$. \\

Consider ${\rm Spec}^{\infty}\,(A)$ together with the constructible topology, that we shall denote by ${\rm Spec}^{\infty-{\rm const}}\,(A)$. We are going to construct a sheaf over this space, starting by the basic open sets of the constructible topology.\\

%We begin by presenting some preliminary results.\\

%\begin{theorem}\label{11iii} Let $A$ be a $\mathcal{C}^{\infty}-$ring and $(\nu_A, {\rm vN}\,(A))$ be its von Neumann-regular hull.
%$$(\nu_A)^{\star}: {\rm Spec}^{\infty}\,({\rm vN}\,(A)) \to {\rm Spec}^{\infty - {\rm const}}\,(A)$$
%is an homeomorphism where ${\rm Spec}^{\infty - {\rm const}}\,(A)$ is the boolean topological space on the set of prime $\mathcal{C}^{\infty}-$radical ideals of $A$ with the constructible topology.
%\end{theorem}
%\begin{proof}In order to show that $(\nu_A)^{\star}: {\rm Spec}^{\infty}\,({\rm vN}\,(A)) \to {\rm Spec}^{\infty - {\rm const}}\,(A)$ is a homeomorphism, it is necessary and sufficient to prove that the spectral map $(\nu_A)^{\star}: {\rm Spec}^{\infty}\,({\rm vN}\,(A)) \to {\rm Spec}^{\infty}\,(A)$ is a bijection from the boolean space ${\rm Spec}^{\infty}\,({\rm vN}\,(A))$ to the spectral space ${\rm Spec}^{\infty}\,(A)$.\\

%Keeping the notation in the proof of \textbf{Theorem \ref{pot}}, it is enough to prove that $\gamma_A \circ (\nu_A)^{\star} = {\rm id}_{{\rm Spec}^{\infty}\,({\rm vN}\,(A))}$. Let $\mathfrak{q} \in {\rm Spec}^{\infty}\,({\rm vN}\,(A))$, then $\gamma_A((\nu_A)^{\star}(\mathfrak{q})) = \ker (\widehat{{\rm Frac}^{A}_{\mathfrak{q}}}) = \ker ({\alpha}^{{\rm vN}\,(A)}_{\mathfrak{q}})_{\mathfrak{q} \in {\rm Spec}^{\infty}\,({\rm vN}\,(A))} = \mathfrak{q},$ since $\widehat{\nu_{A_{\mathfrak{q}}}} \circ \widetilde{{\rm Frac}^{A}_{\mathfrak{p}}} = {\rm Frac^{{\rm vN}\,(A)}_{\mathfrak{q}}}$ is injective.
%\end{proof}

Given any basic open set $V$ of ${\rm Spec}^{\infty-{\rm const}}\,(A)$, there are $a,b \in A$ such that $V = D^{\infty}(a)\cap Z^{\infty}(b)$. Given $\mathfrak{p} \in V = D^{\infty}(a)\cap Z^{\infty}(b)$, $b \in \mathfrak{p}$ and $a \notin \mathfrak{p}$, so
$$\sqrt[\infty]{(b)} = \bigcap \{ \mathfrak{p}' \in {\rm Spec}^{\infty}\,(A) | b \in \mathfrak{p}' \} \subseteq \mathfrak{p} = \sqrt[\infty]{\mathfrak{p}}.$$

Let $q_{b,\mathfrak{p}}: \dfrac{A}{\sqrt[\infty]{(b)}} \twoheadrightarrow \dfrac{A}{\mathfrak{p}}$ be the only $\mathcal{C}^{\infty}-$homomorphism such that:
$$\xymatrixcolsep{3pc}\xymatrix{
  & A \ar[dl]_{q_{\sqrt[\infty]{(b)}}} \ar[rd]^{q_{\mathfrak{p}}} &   \\
\dfrac{A}{\sqrt[\infty]{(b)}} \ar[rr]^{q_{b,\mathfrak{p}}} & & \dfrac{A}{\mathfrak{p}}
}$$

commutes. Consider the following diagram:

$$\xymatrixcolsep{3pc}\xymatrix{
  & \left(\frac{A}{\sqrt[\infty]{(b)}}\right) \ar[r]^{\eta_{a+\sqrt[\infty]{(b)}}} \ar[dd]^{q_{b,\mathfrak{p}}} & \left(\frac{A}{\sqrt[\infty]{(b)}}\right)\left\{{a + \sqrt[\infty]{(b)}}^{-1} \right\}\\
A \ar[ur]^{q_{\sqrt[\infty]{(b)}}} \ar[dr]_{q_{\mathfrak{p}}} & & \\
  & \left( \frac{A}{\mathfrak{p}}\right) \ar[r]^{\eta_{a+\mathfrak{p}}} & \left(\frac{A}{\mathfrak{p}}\right)\left\{{a+\mathfrak{p}}^{-1} \right\}
}$$

Note that $(\eta_{a+\mathfrak{p}} \circ q_{b,\mathfrak{p}})(a + \sqrt[\infty]{(b)}) = \eta_{a+\mathfrak{p}}(q_{b,\mathfrak{p}}(a+\sqrt[\infty]{(b)})) = \eta_{a+\mathfrak{p}}(a+\mathfrak{p}) \in \left( \dfrac{A}{\mathfrak{p}}\{ {a+\mathfrak{p}}^{-1}\}\right)^{\times}$, and due to the universal property of $\eta_{a+\sqrt[\infty]{(b)}}: \left(\frac{A}{\sqrt[\infty]{(b)}}\right) \to \left(\frac{A}{\sqrt[\infty]{(b)}}\right)\left\{ {a+\sqrt[\infty]{(b)}}^{-1}\right\}$ there is a unique $\mathcal{C}^{\infty}-$homomorphism $$\eta_{b,\mathfrak{p}}: \left(\frac{A}{\sqrt[\infty]{(b)}}\right)\left\{ {a + \sqrt[\infty]{(b)}}^{-1}\right\} \to \left(\frac{A}{\mathfrak{p}}\right)\left\{ {a+\mathfrak{p}}^{-1}\right\}$$ such that the following diagram commutes:
$$\xymatrixcolsep{3pc}\xymatrix{
\left(\frac{A}{\sqrt[\infty]{(b)}}\right) \ar[dr]_{\eta_{a+\mathfrak{p}}\circ q_{b,\mathfrak{p}}} \ar[r]^{\eta_{a+\sqrt[\infty]{(b)}}} & \left(\frac{A}{\sqrt[\infty]{(b)}}\right)\left\{{a+\sqrt[\infty]{(b)}}^{-1} \right\} \ar[d]^{\eta_{b,\mathfrak{p}}}\\
  & \left(\frac{A}{\mathfrak{p}}\right)\left\{ {a+\mathfrak{p}}^{-1}\right\}
}$$

 We now have the following commutative diagram:

$$\xymatrixcolsep{3pc}\xymatrix{
  & \left(\frac{A}{\sqrt[\infty]{(b)}}\right) \ar[r]^{\eta_{a+\sqrt[\infty]{(b)}}} \ar@{->>}[dd]^{q_{b,\mathfrak{p}}} & \left(\frac{A}{\sqrt[\infty]{(b)}}\right)\left\{{a + \sqrt[\infty]{(b)}}^{-1} \right\} \ar@{->>}[dd]^{\eta_{b,\mathfrak{p}}}  \\
A \ar[ur]^{q_{\sqrt[\infty]{(b)}}} \ar[dr]_{q_{\mathfrak{p}}} & & \\
  & \left(\frac{A}{\mathfrak{p}}\right) \ar[r]^{\eta_{a+\mathfrak{p}}} & \left(\frac{A}{\mathfrak{p}}\right)\left\{{a+\mathfrak{p}}^{-1} \right\} }$$

where $\eta_{b,\mathfrak{p}}$ is surjective due to the following:\\

Since $b \in \mathfrak{p}$, $\sqrt[\infty]{(b)} \subseteq \mathfrak{p}$ so $q_{b,\mathfrak{p}}: \left(\frac{A}{\sqrt[\infty]{(b)}}\right) \to \left(\frac{A}{\mathfrak{p}}\right)$ is surjective and $\langle \eta_a[\sqrt[\infty]{(b)}]\rangle \subseteq \langle \eta_a[\mathfrak{p}]\rangle$. Also, by \textbf{Corollary 4}, p.62 of \cite{BM2}, rings of fractions commute with quotients, so
$$\begin{cases}
    \left(\frac{A}{\sqrt[\infty]{(b)}}\right)\{ {a + \sqrt[\infty]{(b)}}^{-1} \} \stackrel{\mu_{a,b}}{\cong} \left(\frac{A\{ a^{-1}\}}{\langle \eta_a[\sqrt[\infty]{(b)}] \rangle} \right)\\
    \left(\frac{A}{\mathfrak{p}}\right)\left\{ {a+\mathfrak{p}}^{-1}\right\} \stackrel{\mu_{a,\mathfrak{p}}}{\cong} \left(\frac{A\{ a^{-1}\}}{\eta_a[\mathfrak{p}]}\right)
  \end{cases}$$

and we have the following commutative diagram:

$$\xymatrixcolsep{3pc}\xymatrix{
\left(\frac{A}{\sqrt[\infty]{(b)}}\right)\left\{{a+\sqrt[\infty]{(b)}}^{-1} \right\} \ar[dd]^{\eta_{b,\mathfrak{p}}} \ar[r]^{\mu_{a,b}}_{\cong} & \dfrac{A\{ a^{-1}\}}{\langle \eta_a[\sqrt[\infty]{(b)}]\rangle} \ar@{->>}[dd]^{\mathfrak{p}i_{b,\mathfrak{p}}} &   \\
  &  &  A\{ a^{-1}\} \ar[ul]_{q_{\eta_a[\sqrt[\infty]{(b)}]}} \ar[dl]^{q_{\eta_a[\mathfrak{p}]}} \\
\left(\frac{A}{\mathfrak{p}}\right)\left\{ {a+\mathfrak{p}}^{-1} \right\} \ar[r]^{\mu_{a,\mathfrak{p}}}_{\cong} & \dfrac{A\{ a^{-1}\}}{\langle \eta_a[\mathfrak{p}]\rangle}&  }$$
where $\mu_{a,b}$ and $\mu_{a,\mathfrak{p}}$ are the isomorphisms described in \textbf{Corollary 4}, p.62 of \cite{BM2}. We have:
$$\eta_{b,\mathfrak{p}} = (\mu_{a,\mathfrak{p}})^{-1} \circ \mathfrak{p}i_{b,\mathfrak{p}}\circ \mu_{a,b},$$
so $\eta_{b,\mathfrak{p}}$ is a composition of surjective maps, hence it is a surjective map.\\

Considering the map:
$$\eta_{\frac{A}{\mathfrak{p}}\setminus \{ 0 + \mathfrak{p}\}}: \dfrac{A}{\mathfrak{p}}\{ {a+\mathfrak{p}}^{-1}\} \rightarrow \dfrac{A}{\mathfrak{p}}\lbrace{ {\dfrac{A}{\mathfrak{p}}\setminus \{ 0 +\mathfrak{p}\}}^{-1}\rbrace} = k_{\mathfrak{p}}(A)$$

we define $f^{a,b}_{\mathfrak{p}} := \eta_{\frac{A}{\mathfrak{p}}\setminus \{ 0 + \mathfrak{p}\}} \circ \eta_{b,\mathfrak{p}}: \dfrac{A}{\sqrt[\infty]{(b)}}\{ {a + \sqrt[\infty]{(b)}}^{-1}\} \to k_{\mathfrak{p}}(A)$, \textit{i.e.}, such that the following diagram commutes:
$$\xymatrixcolsep{3pc}\xymatrix{
\dfrac{A}{\sqrt[\infty]{(b)}}\{ {a + \sqrt[\infty]{(b)}}^{-1} \} \ar[d]^{\eta_{b,\mathfrak{p}}} \ar[rd]^{f^{a,b}_{\mathfrak{p}}}\\
\dfrac{A}{\mathfrak{p}}\{ {a+\mathfrak{p}}^{-1}\} \ar[r]^{\eta_{\frac{A}{\mathfrak{p}}\setminus \{ 0 + \mathfrak{p}\}}}& k_{\mathfrak{p}}(A)
}$$

Let $t_{a,b} := \eta_{a+\sqrt[\infty]{(b)}} \circ q_{\sqrt[\infty]{(b)}}: A \to \dfrac{A}{\sqrt[\infty]{(b)}}\{ a + \sqrt[\infty]{(b)}\}$ and define:
$$t_{\mathfrak{p}}:= \eta_{\frac{A}{\mathfrak{p}}\setminus \{ 0 + \mathfrak{p}\}} \circ \eta^{a+\mathfrak{p}}_{\frac{A}{\mathfrak{p}}} \circ q_{\mathfrak{p}}: A \to k_{\mathfrak{p}}(A),$$

$$\xymatrixcolsep{3pc}\xymatrix{
\dfrac{A}{\mathfrak{p}} \ar[r]^{\eta_{a+\mathfrak{p}}} & \dfrac{A}{\mathfrak{p}}\{ {a+\mathfrak{p}}^{-1}\} \ar[d]^{\eta_{\frac{A}{\mathfrak{p}} \setminus \{ 0 + \mathfrak{p}\}}}\\
A \ar@{->>}[u]_{q_{\mathfrak{p}}} \ar[r]^{t_{\mathfrak{p}}} & k_{\mathfrak{p}}(A)
}$$

where $\eta_{a+\mathfrak{p}}: \dfrac{A}{\mathfrak{p}} \to \dfrac{A}{\mathfrak{p}}\{ {a+\mathfrak{p}}^{-1}\}$ is the ring of fractions of $\dfrac{A}{\mathfrak{p}}$ with respect to $a + \mathfrak{p}$.\\

We claim that $f^{a,b}_{\mathfrak{p}}: \left(\frac{A}{\sqrt[\infty]{(b)}}\right) \to \left(\frac{A}{\mathfrak{p}}\right)\lbrace \frac{A}{\mathfrak{p}} \setminus \left\{ 0 + \mathfrak{p}\right\}^{-1}\rbrace$ is the unique $\mathcal{C}^{\infty}-$homo\-morphism such that the triangle:

$$\xymatrix{
A\ar[r]^{t_{a,b}} \ar[dr]^{t_{\mathfrak{p}}} & \frac{A}{\sqrt[\infty]{(b)}}\{ {a + \sqrt[\infty]{(b)}}^{-1}\} \ar[d]^{f^{a,b}_{\mathfrak{p}}}\\
   & \left(\frac{A}{\mathfrak{p}}\right)\lbrace \frac{A}{\mathfrak{p}} \setminus \{ 0 + \mathfrak{p}\}\rbrace
}$$
commutes.\\

Let $g: \frac{A}{\sqrt[\infty]{(b)}} \to \left(\frac{A}{\mathfrak{p}}\right)\left\{ \frac{A}{\mathfrak{p}} \setminus \{ 0 + \mathfrak{p}\}^{-1}\right\}$ be a $\mathcal{C}^{\infty}-$homomorphism  such that the following triangle commutes:

$$\xymatrix{
A \ar[r]^{t_{a,b}} \ar[dr]^{t_{\mathfrak{p}}} & \left(\frac{A}{\sqrt[\infty]{(b)}}\right)\{ {a + \sqrt[\infty]{(b)}}^{-1}\} \ar[d]^{g}\\
   & \dfrac{A}{\mathfrak{p}}\lbrace \dfrac{A}{\mathfrak{p}} \setminus \{ 0 + \mathfrak{p}\}\rbrace
}$$

that is, such that $g \circ t_{a,b} = t_{\mathfrak{p}}$. We have:
$$g \circ \eta_{\frac{A}{\sqrt[\infty]{(b)}}} \circ q_{b} = t_{\mathfrak{p}} = \eta_{\frac{A}{\mathfrak{p}}\setminus \{ 0 + \mathfrak{p}\}} \circ q^{b}_{\mathfrak{p}} \circ q_b.$$

Since $q_b$ is an epimorphism, it follows that:
$$g \circ \eta_{\frac{A}{\sqrt[\infty]{(b)}}} = \eta_{\frac{A}{\mathfrak{p}} \setminus \{ 0 + \mathfrak{p}\}} \circ q_{b,\mathfrak{p}}$$
and
$$\eta_{\frac{A}{\mathfrak{p}} \setminus \{ 0 + \mathfrak{p}\}} \circ q_{b,\mathfrak{p}} = \eta_{a+\mathfrak{p}, \frac{A}{\mathfrak{p}}\setminus \{ 0 + \mathfrak{p} \}} \circ \eta_{b,\mathfrak{p}} \circ \eta_{\frac{A}{\sqrt[\infty]{(b)}}},$$

where:
$$\eta_{a+\mathfrak{p}, \frac{A}{\mathfrak{p}}\setminus \{ 0 + \mathfrak{p} \}} : \dfrac{A}{\mathfrak{p}}\{{a+\mathfrak{p}}^{-1} \} \rightarrow k_{\mathfrak{p}}(A.)$$

By the universal property of $\eta_{\frac{A}{\sqrt[\infty]{(b)}}}$, it follows that $g = \eta_{a+\mathfrak{p}, \frac{A}{\mathfrak{p}}\setminus \{ 0 + \mathfrak{p} \}}$, hence $g = f^{a,b}_{\mathfrak{p}}$ and $f^{a,b}_{\mathfrak{p}}$ is unique.\\

We have now the following commutative diagram:

$$\xymatrixcolsep{3pc}\xymatrix{
  & \left(\frac{A}{\sqrt[\infty]{(b)}}\right) \ar[r]^{\eta_{a+\sqrt[\infty]{(b)}}} \ar@{->>}[dd]^{q_{b,\mathfrak{p}}} & \left(\frac{A}{\sqrt[\infty]{(b)}}\right)\left\{{a + \sqrt[\infty]{(b)}}^{-1} \right\} \ar@{->>}[dd]^{\eta_{b,\mathfrak{p}}} \ar[ddr]^{f^{a,b}_{\mathfrak{p}}} & \\
A \ar[ur]^{q_{\sqrt[\infty]{(b)}}} \ar[dr]_{q_{\mathfrak{p}}} & & &\\
  & \left(\frac{A}{\mathfrak{p}}\right) \ar[r]^{\eta_{a+\mathfrak{p}}} & \left(\frac{A}{\mathfrak{p}}\right)\left\{{a+\mathfrak{p}}^{-1} \right\} \ar[r]^{\eta_{\frac{A}{\mathfrak{p}}\setminus\{ 0 + \mathfrak{p}\}}} & \mathop{\rm colim}_{a \notin \mathfrak{p}} \left(\frac{A}{\mathfrak{p}}\right)\left\{ {a+\mathfrak{p}}^{-1}\right\}
}$$

\begin{remark}\label{Encarnacao1}
Since $\sqrt[\infty]{(b)} \subseteq \mathfrak{p}$, it follows that $q_{b,\mathfrak{p}}$ is surjective and since $\bigcup {\rm im}\, \jmath_a = \mathop{\rm colim}_{a \notin \mathfrak{p}} \frac{A}{\mathfrak{p}}\{ {a+\mathfrak{p}}^{-1}\}$ (because it is a colimit and the right triangle of the above diagram commutes), it will follow that $\bigcup {\rm im}\, f^{a,b}_{\mathfrak{p}} = \mathop{\rm colim}_{a \notin \mathfrak{p}} \frac{A}{\mathfrak{p}}\{ {a+\mathfrak{p}}^{-1}\}$. Hence

$$\left\{f^{a,b}_{\mathfrak{p}}: \left( \frac{A}{\sqrt[\infty]{(b)}}\right)\{ {a+\sqrt[\infty]{(b)}}^{-1}\} \rightarrow \mathop{\rm colim}_{a \notin \mathfrak{p}} \frac{A}{\mathfrak{p}}\{ {a+\mathfrak{p}}^{-1}\} | \mathfrak{p} \in D^{\infty}(a)\cap Z^{\infty}(b) \right\}$$
 is collectively surjective;
\end{remark}

%%%%%%%%%%%%%%%%%%%%%%%%%%%%%%%%%%%%%%%%%%%%%%%%%%%%%%%%

Let $a,a',b,b' \in A$ be such that $D^{\infty}(a)\cap Z^{\infty}(b) \subseteq D^{\infty}(a')\cap Z^{\infty}(b')$. Then
 \begin{equation}\label{I}D^{\infty}(a)\cap Z^{\infty}(b) \subseteq Z^{\infty}(b')\end{equation}
  and
 \begin{equation}\label{II}D^{\infty}(a) \cap Z^{\infty}(b) \subseteq D^{\infty}(a') \end{equation}

We claim that \eqref{I} implies that $\sqrt[\infty]{(b')} \subseteq \ker (t_{a,b})$.\\

It suffices to show that:

$$\bigcap D^{\infty}(a)\cap Z^{\infty}(b) \subseteq \ker (t_{a,b}),$$
for
$$ \sqrt[\infty]{(b')} = \bigcap Z^{\infty}(b') \subseteq \bigcap D^{\infty}(a)\cap Z^{\infty}(b).$$

We have:
$$\bigcap D^{\infty}(a)\cap Z^{\infty}(b) = \{ x \in A | (\forall \mathfrak{p} \in {\rm Spec}^{\infty}\,(A))(a \notin \mathfrak{p}) \wedge (b \in \mathfrak{p}) \to (x \in \mathfrak{p})\}$$

$$t_{a,b}(x) = \eta^{a+\sqrt[\infty]{(b)}}_{\frac{A}{\sqrt[\infty]{(b)}}}(q_{b}(x)) = 0 \iff (\exists \theta \in \{ a+\sqrt[\infty]{(b)}\}^{\infty-{\rm sat}})(\theta \cdot q_b(x) = 0 \,\, {\rm in}\,\, \dfrac{A}{\sqrt[\infty]{(b)}})$$

Claim: $x \in \bigcap D^{\infty}(a)\cap Z^{\infty}(b) \rightarrow a \cdot x \in \sqrt[\infty]{(b)}$.\\

\textit{Ab absurdo}, suppose $b \in \mathfrak{p}$ and $a \cdot x \notin \sqrt[\infty]{(b)}$. Since $\sqrt[\infty]{(b)} = \bigcap Z^{\infty}(b)$, there must exist some $\mathfrak{p} \in Z^{\infty}(b)$ such that $a \cdot x \in \mathfrak{p}$. Since $\mathfrak{p}$ is a prime ideal, it follows that $a \notin \mathfrak{p}$ and $x \notin \mathfrak{p}$.\\

Thus we have a contradiction between $(x \in \bigcap D^{\infty}(a)\cap Z^{\infty}(b))$ and $(\exists \mathfrak{p} \in {\rm Spec}^{\infty}\,(A))((a \notin \mathfrak{p})\wedge (b \in \mathfrak{p}) \wedge (x \notin \mathfrak{p}))$.\\

Now,
$$t_{a,b}(x) = \eta_{\frac{A}{\sqrt[\infty]{(b)}}}(q_b(x)) \cdot \underbrace{\eta_{\frac{A}{\sqrt[\infty]{(b)}}}(q_{b}(a))}_{\in \left( \dfrac{A}{\sqrt[\infty]{(b)}}\{ {a + \sqrt[\infty]{(b)}}^{-1}\}\right)^{\times}} = \eta^{a+\sqrt[\infty]{(b)}}_{\frac{A}{\sqrt[\infty]{(b)}}}(q_b(a \cdot x)) = \eta_{\frac{A}{\sqrt[\infty]{(b)}}}(0)=0$$
so $x \in \bigcap D^{\infty}(a)\cap Z^{\infty}(b)$ implies $x \in \ker (t_{a,b})$, \textit{i.e.},
$$\bigcap D^{\infty}(a)\cap Z^{\infty}(b) \subseteq \ker (t_{a,b}),$$
hence
$$\sqrt[\infty]{(b')} \subseteq \bigcap D^{\infty}(a)\cap Z^{\infty}(b) \subseteq \ker (t_{a,b})$$
$$\sqrt[\infty]{(b')} \subseteq \ker (t_{a,b}).$$

By the \textbf{Theorem of the $\mathcal{C}^{\infty}-$Homomorphism}, there is a unique $\mathcal{C}^{\infty}-$homomorphism   $r^{b'}_{a,b}: \dfrac{A}{\sqrt[\infty]{(b')}} \to \dfrac{A}{\sqrt[\infty]{(b)}}\{ {a + \sqrt[\infty]{(b)}}^{-1}\}$ such that $r^{b'}_{a,b} \circ q_{\sqrt[\infty]{(b')}} = t_{a,b}$, that is, such that the following diagram commutes:

$$\xymatrixcolsep{3pc}\xymatrix{
A \ar[d]_{q_{b'}} \ar[r]^{t_{a,b}} & \dfrac{A}{\sqrt[\infty]{(b)}}\{ {a+\sqrt[\infty]{(b)}}^{-1} \}\\
\dfrac{A}{\sqrt[\infty]{(b')}} \ar[ur]_{r^{b'}_{a,b}}
}$$

Now we claim that there is a unique $\mathcal{C}^{\infty}-$homomorphism $t^{a',b'}_{a,b}: \dfrac{A}{\sqrt[\infty]{(b')}}\{ {a' + \sqrt[\infty]{(b')}}^{-1}\} \to \dfrac{A}{\sqrt[\infty]{(b)}}\{ {a + \sqrt[\infty]{(b)}}^{-1}\}$ such that the following diagram commutes:

$$\xymatrixcolsep{3pc}\xymatrix{
  & \dfrac{A}{\sqrt[\infty]{(b)}} \ar[r]^{\eta_{a+\sqrt[\infty]{(b)}}} & \dfrac{A}{\sqrt[\infty]{(b)}}\{ {a+\sqrt[\infty]{(b)}}^{-1}\}\\
A \ar[ur]^{q_b} \ar[dr]_{q_{b'}} \ar@/_5pc/[drr]_{t_{a',b'}} \ar@/^5pc/[urr]^{t_{a,b}} & & \\
  & \dfrac{A}{\sqrt[\infty]{(b')}} \ar[r]_{\eta_{a'+\sqrt[\infty]{(b')}}} \ar[uur]^{r^{b'}_{a,b}} & \dfrac{A}{\sqrt[\infty]{(b')}}\{ {a'+\sqrt[\infty]{(b')}}\} \ar[uu]^{ \exists ! t^{a',b'}_{a,b}}   }$$

We claim that \eqref{II} implies that $r^{b'}_{a,b}(a'+\sqrt[\infty]{(b')}) \in \left( \dfrac{A}{\sqrt[\infty]{(b)}} \{ {a + \sqrt[\infty]{(b)}}\}\right)^{\times}$. This is equivalent to assert that $t_{a,b}(a') \in \left( \dfrac{A}{\sqrt[\infty]{(b)}} \{ {a + \sqrt[\infty]{(b)}}\}\right)^{\times}$ or that $q_b(a') \in \{ a + \sqrt[\infty]{(b)}\}^{\infty-{\rm sat}}$.\\

\textit{Ab absurdo}, suppose that $q_b(a') \notin \{ a + \sqrt[\infty]{(b)}\}^{\infty-{\rm sat}} \subseteq \dfrac{A}{\sqrt[\infty]{(b)}}$. By \textbf{Theorem 21} of page 89 of \cite{BM2}, there is some $\mathfrak{P} \in {\rm Spec}^{\infty}\,\left(\dfrac{A}{\sqrt[\infty]{(b)}}\right)$ ($\sqrt[\infty]{(b)} \subseteq \mathfrak{P}$) such that $q_{b}(a') \in \mathfrak{P}$ and $a+\sqrt[\infty]{(b)} \notin \mathfrak{P}$. Since ${\rm Spec}^{\infty}\, \left( \dfrac{A}{\sqrt[\infty]{(b)}}\right) \cong {\rm Spec}^{\infty}_{\sqrt[\infty]{(b)}}\,(A) =\{ \mathfrak{s} \in {\rm Spec}^{\infty}\,(A) | \sqrt[\infty]{(b)} \subseteq \mathfrak{s}\}$, there is a unique $\mathfrak{p} \in {\rm Spec}^{\infty}_{\sqrt[\infty]{(b)}}\,(A)$ such that $\mathfrak{P} = \mathfrak{p} + \sqrt[\infty]{(b)}$ and $\mathfrak{p} \in Z^{\infty}(b)$ (because $b \in \sqrt[\infty]{(b)} \subseteq \mathfrak{P} = \mathfrak{p} + \sqrt[\infty]{(b)}$).\\

Since $q_b(a') \in \mathfrak{P} = \mathfrak{p} + \sqrt[\infty]{(b)}$ and $a+\sqrt[\infty]{(b)} \notin \mathfrak{P} = \mathfrak{p} + \sqrt[\infty]{(b)}$, it follows that $a' \in \mathfrak{p}$ ($\mathfrak{p} \in Z^{\infty}(a')$) and $a \notin \mathfrak{p}$ ($\mathfrak{p} \in D^{\infty}(a)$). Since $b \in \mathfrak{p}$ ($\mathfrak{p} \in Z^{\infty}(b)$), it follows that:
$$\mathfrak{p} \in D^{\infty}(a)\cap Z^{\infty}(b) \cap Z^{\infty}(a'),$$
hence
$$ D^{\infty}(a)\cap Z^{\infty}(b) \cap Z^{\infty}(a') \neq \varnothing$$
$$ D^{\infty}(a)\cap Z^{\infty}(b) \nsubseteq D^{\infty}(a'),$$
which contradicts \eqref{II}.\\

Hence:
$$q_b(a') \in \{ a + \sqrt[\infty]{(b)}\}^{\infty-{\rm sat}}$$

and by the universal property of $\eta_{a'+\sqrt[\infty]{(b')}}: \dfrac{A}{\sqrt[\infty]{(b')}} \rightarrow \dfrac{A}{\sqrt[\infty]{(b')}} \{ {a'+\sqrt[\infty]{(b')}}^{-1}\} $, there exists a unique $\mathcal{C}^{\infty}-$homomorphism $t^{a',b'}_{a,b} :  \dfrac{A}{\sqrt[\infty]{(b')}} \{ {a'+\sqrt[\infty]{(b')}}^{-1}\} \rightarrow \dfrac{A}{\sqrt[\infty]{(b)}} \{ {a+\sqrt[\infty]{(b)}}^{-1}\}$ such that the following triangle commutes:
$$\xymatrixcolsep{3pc}\xymatrix{
  & \dfrac{A}{\sqrt[\infty]{(b)}} \{ {a+\sqrt[\infty]{(b)}}^{-1}\}\\
\dfrac{A}{\sqrt[\infty]{(b')}} \ar[ur]^{r^{b'}_{a,b}} \ar[r]^{\eta_{a'+\sqrt[\infty]{(b')}}} & \dfrac{A}{\sqrt[\infty]{(b')}} \{ {a'+\sqrt[\infty]{(b')}}^{-1}\} \ar[u]^{t^{a',b'}_{a,b}}}$$

Now we claim that the following diagram is a colimit:

$$\xymatrixcolsep{3pc}\xymatrix{
  & \left( \frac{A}{\mathfrak{p}}\right)\left\{ \frac{A}{\mathfrak{p}} \setminus \{ 0+\mathfrak{p}\}^{-1}\right\} &   \\
\left( \frac{A}{\sqrt[\infty]{(b')}}\right)\{ {a'+\sqrt[\infty]{(b')}}^{-1} \} \ar@/^2pc/[ur]^{f^{a,b}_{\mathfrak{p}}} \ar[rr]^{t^{a',b'}_{a,b}} & & \left( \frac{A}{\sqrt[\infty]{(b)}}\right)\{ {a + \sqrt[\infty]{(b)}}^{-1}\} \ar@/_2pc/[ul]_{f^{a,b}_{\mathfrak{p}}}
}$$

In order to show that, three conditions must hold:

\begin{itemize}
  \item[1)]{For every $a,b,a',b'$ such that $\mathfrak{p} \in D^{\infty}(a)\cap Z^{\infty}(b)$ and $\mathfrak{p} \in D^{\infty}(a')\cap Z^{\infty}(b')$ the above diagram commutes;}
  \item[2)]{The family $\{ f^{a,b}_{\mathfrak{p}}: \left( \frac{A}{\sqrt[\infty]{(b)}}\right)\{ {a +\sqrt[\infty]{(b)}}^{-1}\} \rightarrow \left( \frac{A}{\mathfrak{p}}\right)\left\{ {\frac{A}{\mathfrak{p}} \setminus \{ 0+\mathfrak{p}\}}^{-1}\right\} \}$ is collectively surjective;}
  \item[3)]{Given $a_1,b_1,a_2,b_2 \in A$ such that $D^{\infty}(a_2)\cap Z^{\infty}(b_2) \subseteq D^{\infty}(a_1) \cap Z^{\infty}(b_1)$, $\theta_1 \in \left( \frac{A}{\sqrt[\infty]{(b_1)}}\right)\{ {a_1 + \sqrt[\infty]{(b_1)}}^{-1}\}$ and  $\theta_2 \in \left( \frac{A}{\sqrt[\infty]{(b_2)}}\right)\{ {a_2 + \sqrt[\infty]{(b_2)}}^{-1}\}$ such that $t^{a_1,b_1}_{\mathfrak{p}}(\theta_1) = t^{a_2,b_2}_{\mathfrak{p}}(\theta_2)$, there are $a,b \in A$ such that $\mathfrak{p} \in D^{\infty}(a)\cap Z^{\infty}(b) \subseteq D^{\infty}(a_2)\cap Z^{\infty}(b_2)$ and $t^{a_1,b_1}_{a,b}(\theta_1) = t^{a_2,b_2}_{a,b}(\theta_2)$, where $t^{a_1,b_1}_{a,b}$ and $t^{a_2,b_2}_{a,b}$ are given in the following commutative diagram:
      $$\xymatrixcolsep{3pc}
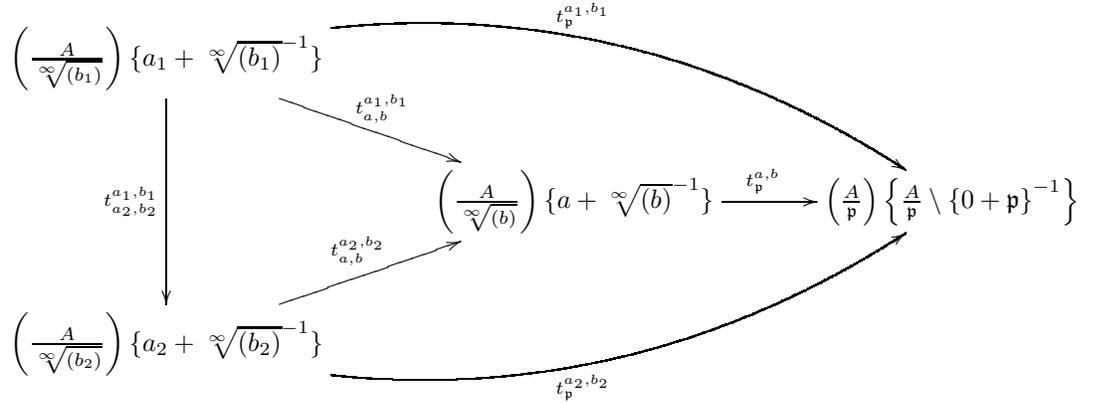
\xymatrix{
      \left( \frac{A}{\sqrt[\infty]{(b_1)}}\right)\{ {a_1 + \sqrt[\infty]{(b_1)}}^{-1}\} \ar[dd]_{t^{a_1,b_1}_{a_2,b_2}} \ar[rd]^{t^{a_1,b_1}_{a,b}} \ar@/^3pc/[rrd]^{t^{a_1,b_1}_{\mathfrak{p}}} &    &     \\
          & \left( \frac{A}{\sqrt[\infty]{(b)}}\right)\{ {a + \sqrt[\infty]{(b)}}^{-1}\} \ar[r]^{t^{a,b}_{\mathfrak{p}}} & \left( \frac{A}{\mathfrak{p}}\right)\left\{ {\frac{A}{\mathfrak{p}}\setminus \{ 0 + \mathfrak{p}\}}^{-1}\right\}\\
      \left( \frac{A}{\sqrt[\infty]{(b_2)}}\right)\{ {a_2 + \sqrt[\infty]{(b_2)}}^{-1}\} \ar[ur]^{t^{a_2,b_2}_{a,b}} \ar@/_3pc/[rru]_{t^{a_2,b_2}_{\mathfrak{p}}} &   &
      }$$

         }
\end{itemize}

By \textbf{Remark \ref{Encarnacao1}}, the condition 2) is already satisfied.\\

Note that for each basic open subset $U$ of ${\rm Spec}^{\infty-{\rm const}}\,(A)$, $F_A(U)$ has the structure of an $A-$algebra. The universal property of
$$F_A(U) = \varprojlim_{\substack{U = D^{\infty}(a)\cap Z^{\infty}(b) \\ (a,b) \in A \times A}} \dfrac{A}{\sqrt[\infty]{(b)}}\{ {a+\sqrt[\infty]{(b)}}^{-1} \}$$
induces a unique $\mathcal{C}^{\infty}-$homo\-morphism $u: A \to F_A(U)$ such that the following diagram commutes:
$$
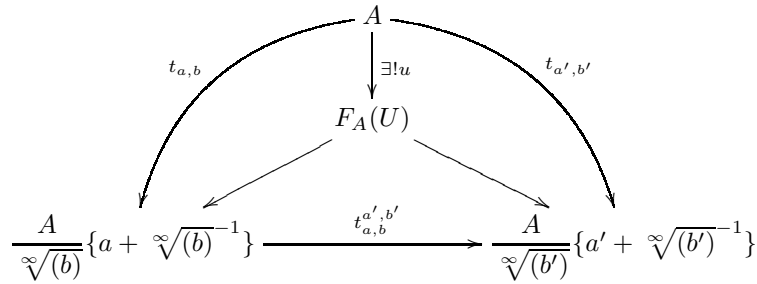
\xymatrix{
   & A \ar@/_2pc/[ldd]_{t_{a,b}} \ar@/^2pc/[rdd]^{t_{a',b'}} \ar[d]^{\exists ! u} &   \\
   & F_A(U) \ar[ld] \ar[rd] &   \\
\dfrac{A}{\sqrt[\infty]{(b)}}\{ {a + \sqrt[\infty]{(b)}}^{-1}\} \ar[rr]^{t^{a',b'}_{a,b}} & & \dfrac{A}{\sqrt[\infty]{(b')}}\{ {a'+\sqrt[\infty]{(b')}}^{-1}\}
}$$

for every $a,a',b,b' \in A$ such that $D^{\infty}(a)\cap Z^{\infty}(b) = U = D^{\infty}(a')\cap Z^{\infty}(b')$.\\

Now let $U$ and $V$ be any two basic open sets of ${\rm Spec}^{\infty-{\rm const}}\,(A)$ such that $U \subseteq V$, and let $\iota^{V}_{U}: U \hookrightarrow V$ be the inclusion map. Since $U$ and $V$ are basic open subsets of the constructible topology, there are $a,b,c,d \in A$ such that $U = D^{\infty}(a)\cap Z^{\infty}(b)$ and $V = D^{\infty}(c)\cap Z^{\infty}(d)$. Let $a',b',c',d' \in A$ be such that $U =D^{\infty}(a)\cap Z^{\infty}(b) = D^{\infty}(a')\cap Z^{\infty}(b')$ and $V = D^{\infty}(c)\cap Z^{\infty}(d) = D^{\infty}(c')\cap Z^{\infty}(d')$, so:

$$t^{a,b}_{a',b'}: \dfrac{A}{\sqrt[\infty]{(b)}}\{ {a+\sqrt[\infty]{(b)}}^{-1} \} \rightarrow \dfrac{A}{\sqrt[\infty]{(b')}}\{ {a' + \sqrt[\infty]{(b')}}^{-1}\}$$

and
$$t^{c,d}_{c',d'}: \dfrac{A}{\sqrt[\infty]{(d)}}\{ {c + \sqrt[\infty]{(d)}}^{-1}\} \rightarrow \dfrac{A}{\sqrt[\infty]{(d')}}\{ {c' + \sqrt[\infty]{(d')}}^{-1}\}$$

are $\mathcal{C}^{\infty}-$isomorphisms. Consider the following commutative diagram:\\

$$\xymatrixcolsep{3pc}\xymatrix{
 &  & F_A(U) & \\
 & \dfrac{A}{\sqrt[\infty]{(b)}}\{ {a+\sqrt[\infty]{(b)}}^{-1} \} \ar[rr]^{t^{a,b}_{a',b'}}_{\cong} \ar[ur]^{\jmath_{a,b}}& & \dfrac{A}{\sqrt[\infty]{(b')}}\{ {a' + \sqrt[\infty]{(b')}}^{-1}\} \ar[ul]^{\jmath_{a',b'}}\\
A \ar[ur]^{t_{a,b}} \ar[dr]_{t_{c,d}} & & \\
 &\dfrac{A}{\sqrt[\infty]{(d)}}\{ {c + \sqrt[\infty]{(d)}}\} \ar[rr]^{t^{c,d}_{c',d'}}_{\cong} \ar[uu]^{f^{c,d}_{a,b}}& & \dfrac{A}{\sqrt[\infty]{(d')}}\{ {c' + \sqrt[\infty]{(d')}}^{-1}\} \ar[uu]^{f^{c',d'}_{a',b'}} \ar[uull]^{f^{c',d'}_{a,b}}},$$

 where $f^{c,d}_{a,b}$ is the unique $\mathcal{C}^{\infty}-$homomorphism such that $f^{c,d}_{a,b} \circ t_{c,d} = t_{a,b}$, $f^{c',d'}_{a',b'}$ is the unique $\mathcal{C}^{\infty}-$homomorphism such that $f^{c',d'}_{a',b'} \circ t_{c',d'} = t_{a',b'}$ and $f^{c',d'}_{a,b}$ is the unique $\mathcal{C}^{\infty}-$homomorphism such that $f^{c',d'}_{a,b} \circ t^{c,d}_{c',d'} = f^{c,d}_{a,b}$

We have, then, the following cone:

$$\xymatrixcolsep{3pc}\xymatrix{
  & F_A(U) &   \\
\dfrac{A}{\sqrt[\infty]{(d)}}\{ {c + \sqrt[\infty]{(d)}}^{-1}\} \ar@/^1pc/[ur]^{\jmath_{a,b} \circ f^{c,d}_{a,b}} \ar[rr]^{t^{c,d}_{c',d'}} & & \dfrac{A}{\sqrt[\infty]{(d')}}\{ {c' + \sqrt[\infty]{(d')}}^{-1}\} \ar@/_1pc/[ul]_{\jmath_{a,b} \circ f^{c',d'}_{a,b}}}$$

that induces, by the universal property of $F_A(V)$ as a colimit, a \textit{unique} $\mathcal{C}^{\infty}-$homomorphism $f^{U}_{V}: F_A(V) \to F_A(U)$ such that for every $a,b,c,d,c',d' \in A$ the following diagram commutes:

$$\xymatrixcolsep{3pc}\xymatrix{
  & F_A(V) &   \\
  & F_A(U) \ar@{.>}[u]^{f^{U}_{V}}&   \\
\dfrac{A}{\sqrt[\infty]{(d)}}\{ {c + \sqrt[\infty]{(d)}}^{-1}\} \ar@/^2pc/[uur]^{\jmath_{c,d} \circ f^{c,d}_{a,b}} \ar[ur]^{k_{c,d}} \ar[rr]^{t^{c,d}_{c',d'}}& & \dfrac{A}{\sqrt[\infty]{(d')}}\{ {c' + \sqrt[\infty]{(d')}}^{-1}\} \ar[ul]_{k_{c',d'}} \ar@/_2pc/[uul]_{\jmath_{a,b} \circ t^{c',d'}_{a,b}}}$$

where $k_{c,d}$ and $k_{c',d'}$ are the canonical colimit $\mathcal{C}^{\infty}-$homo\-morphisms. \\

Let $f^{U}_{c,d}:= \jmath^{U}_{a,b} \circ f^{c,d}_{a,b}$, $f^{U}_{c,d}:= \jmath^{U}_{a',b'} \circ f^{c',d'}_{a',b'}$, $f^{V}_{e,f}:= \jmath^{V}_{c,d} \circ \jmath^{e,f}_{c,d}$ and $f^{V}_{e',f'}:= \jmath^{V}_{c',d'} \circ \jmath^{e',f'}_{c',d'}$.

$$\xymatrixcolsep{3pc}\xymatrix{
 & \dfrac{A}{\sqrt[\infty]{(b)}}\{ {a + \sqrt[\infty]{(b)}}^{-1}\} \ar[r]^{\jmath^{U}_{a,b}} & F_A(U) & \dfrac{A}{\sqrt[\infty]{(b')}}\{ {a' + \sqrt[\infty]{(b')}}^{-1}\} \ar[l]_{\jmath^{U}_{a',b'}}\\
A \ar[ur]^{t_{a,b}} \ar[r]^{t_{c,d}} \ar[rd]^{t_{e,f}} & \dfrac{A}{\sqrt[\infty]{(d)}}\{ {c + \sqrt[\infty]{(d)}}^{-1} \} \ar[ur]^{f^{U}_{c,d}} \ar[u]^{f^{c,d}_{a,b}} \ar[r]^{\jmath^{V}_{c,d}} & F_A(V) \ar[u]^{f^{U}_{V}} & \dfrac{A}{\sqrt[\infty]{(d')}}\{ {c' + \sqrt[\infty]{(d')}}^{-1}\} \ar[u]^{f^{c',d'}_{a',b'}} \ar[l]_{\jmath^{V}_{c',d'}} \ar[ul]_{f^{U}_{c',d'}}\\
 & \dfrac{A}{\sqrt[\infty]{(f)}}\{ {e + \sqrt[\infty]{(f)}}\} \ar[u]^{\jmath^{e,f}_{c,d}} \ar[ur]^{f^{V}_{e,f}} \ar@/_4pc/[rr]_{t^{e,f}_{e',f'}} \ar[r]^{\jmath^{W}_{e,f}} &  F_A(W) \ar[u]^{f^{V}_{W}} & \dfrac{A}{\sqrt[\infty]{(f')}}\{ {e' + \sqrt[\infty]{(f')}}\} \ar[u]^{f^{e',f'}_{c',d'}} \ar[l]_{\jmath^{W}_{e',f'}} \ar[ul]_{f^{W}_{e',f'}}
}$$

Denote $f^{U}_{e,f} := \jmath^{U}_{a,b} \circ f^{e,f}_{a,b}$ and $f^{U}_{e',f'} := \jmath^{U}_{a',b'} \circ f^{e',f'}_{a',b'}$. We have:

$$(f^{U}_{V} \circ f^{V}_{W}) \circ \jmath^{W}_{e,f} = (\jmath^{U}_{a,b}) \circ (t^{c,d}_{a,b} \circ t^{e,f}_{c,d}) = (\jmath^{U}_{a,b}) \circ t^{e,f}_{a,b} = f^{U}_{e,f} = f^{U}_{W} \circ \jmath^{W}_{e,f},$$

and the last equality holds because

$$\xymatrixcolsep{3pc}\xymatrix{
F_A(W) \ar[r]^{f^{U}_{W}} & F_A(U) \\
\frac{A}{\sqrt[\infty]{(f)}}\{ {e + \sqrt[\infty]{(f)}}^{-1}\} \ar[u]^{\jmath^{W}_{e,f}} \ar[r]^{t^{e,f}_{a,b}} \ar[ur]^{f^{U}_{e,f}} & \frac{A}{\sqrt[\infty]{(b)}}\{ {a + \sqrt[\infty]{(b)}}^{-1}\} \ar[u]^{\jmath^{U}_{a,b}}}$$

commutes

By the universal property of $F_A(W)$, since

$$f^{U}_{e',f'} \circ t^{e,f}_{e',f'} = f^{U}_{e,f},$$

there is a unique $\mathcal{C}^{\infty}-$homomorphism $f^{U}_{W}: F_A(W) \to F_A(U)$ such that the following diagram commutes:

$$\xymatrixcolsep{3pc}\xymatrix{
  & F_A(U) &  \\
  & F_A(W) \ar@{.>}[u]^{\exists ! f^{U}_{W}} &   \\
\dfrac{A}{\sqrt[\infty]{(f)}}\{ e + \sqrt[\infty]{(f)}\} \ar@/^2pc/[uur]^{f^{U}_{e,f}}\ar[ur]^{\jmath^{W}_{e,f}}\ar[rr]^{t^{e,f}_{e',f'}} & & \dfrac{A}{\sqrt[\infty]{(f')}}\{ e' + \sqrt[\infty]{(f')}\} \ar@/_2pc/[uul]^{f^{U}_{e',f'}} \ar[ul]^{\jmath^{W}_{e',f'}}}$$

Hence, by the uniqueness of $f^{U}_{W}$, it follows that $f^{U}_{W} = f^{U}_{V} \circ f^{V}_{W}$.\\

Given the inclusion $\iota^{U}_{U} = {\rm id}_{U}: U \hookrightarrow U$, by the universal property of the colimit, there is a unique $\mathcal{C}^{\infty}-$homomorphism $f^{U}_{U}: F_A(U) \to F_A(U)$ such that the following diagram commutes for every $a,b,a',b' \in A$ such that $D^{\infty}(a)\cap Z^{\infty}(b) = U = D^{\infty}(a')\cap Z^{\infty}(b')$:

$$\xymatrixcolsep{3pc}\xymatrix{
  & F_A(U) &  \\
  & F_A(U) \ar@{.>}[u]^{\exists ! f^{U}_{U}} &   \\
\dfrac{A}{\sqrt[\infty]{(b)}}\{ a + \sqrt[\infty]{(b)}\} \ar@/^2pc/[uur]^{\jmath^{U}_{a,b}}\ar[ur]^{\jmath^{U}_{a,b}}\ar[rr]^{t^{a,b}_{a',b'}} & & \dfrac{A}{\sqrt[\infty]{(b')}}\{ a' + \sqrt[\infty]{(b')}\} \ar@/_2pc/[uul]_{\jmath^{U}_{a',b'}} \ar[ul]^{\jmath^{U}_{a',b'}}}$$

In particular, if we take $a=a'$ and $b=b'$, we will have $ t^{a,b}_{a,b} = {\rm id}_{a,b}: \left( \frac{A}{\sqrt[\infty]{(b)}}\right)\{ {a+\sqrt[\infty]{(b)}}^{-1}\} \rightarrow \left( \frac{A}{\sqrt[\infty]{(b)}}\right)\{ {a+\sqrt[\infty]{(b)}}^{-1}\}$, so there is a unique $\mathcal{C}^{\infty}-$homo\-morphism $f^{U}_{U}: F_A(U) \to F_A(U)$ such that the following diagram commutes:
$$\xymatrixcolsep{3pc}\xymatrix{
F_A(U) \ar[r]^{f^U_U} & F_A(U)\\
\left( \frac{A}{\sqrt[\infty]{(b)}}\right)\{ {a+\sqrt[\infty]{(b)}}^{-1}\} \ar[u]^{\jmath^{U}_{a,b}} \ar[ur]^{f^{U}_{a,b}} \ar[r]^{{\rm id}_{a,b}} & \left( \frac{A}{\sqrt[\infty]{(b)}}\right)\{ {a+\sqrt[\infty]{(b)}}^{-1}\} \ar[u]^{\jmath^{U}_{a,b}}
}$$

$$f^{U}_{U} \circ \jmath^{U}_{a,b} = f^{U}_{a,b} = \jmath^{U}_{a,b} \circ {\rm id}_{a,b} = {\rm id}_{F_A(U)} \circ \jmath^{U}_{a,b} \Rightarrow f^{U}_{U} = {\rm id}_{F_A(U)}$$

Thus, we have a functor from the category of the basic open subsets of ${\rm Spec}^{\infty-{\rm const}}\,(A)$, $(\mathcal{B}({\rm Spec}^{\infty-{\rm const}}), \subseteq)$ to the category of $\mathcal{C}^{\infty}-$rings:

$$F_A: (\mathcal{B}({\rm Spec}^{\infty-{\rm const}}\,(A)), \subseteq)^{\rm op} \rightarrow \mathcal{C}^{\infty}{\rm Rng}$$

Given any open set $U$ of ${\rm Spec}^{\infty-{\rm const}}\,(A)$ and any $\mathfrak{p} \in U$, there exists some basic open subset $V = D^{\infty}(a)\cap Z^{\infty}(b)$, for some $a,b \in A$, such that $\mathfrak{p} \in V \subseteq U$, so $\mathcal{B}({\rm Spec}^{\infty-{\rm const}}\,(A))^{\rm op}$ is a cofinal subset of ${\rm Open}\,({\rm Spec}^{\infty-{\rm const}}\,(A))^{\rm op}$.\\

In order to describe the stalk of this pre-sheaf at $\mathfrak{p}$, it suffices to calculate the colimit over the system $\mathcal{B}({\rm Spec}^{\infty-{\rm const}}\,(A))$:\\

$$ (F_A)_{\mathfrak{p}} = \displaystyle\varinjlim_{\substack{ U \in {\rm Spec}^{\infty-{\rm const}}\,(A)\\ \mathfrak{p} \in U}} F_A(U) =  \displaystyle\varinjlim_{\substack{V \in \mathcal{B}({\rm Spec}^{\infty-{\rm const}}\,(A))\\ \mathfrak{p} \in V}}  F_A(V) $$

We have the following cone:

\begin{equation} \label{Valinhos}\xymatrixcolsep{3pc}\xymatrix{
\left( \frac{A}{\sqrt[\infty]{(b)}}\right)\{ {a+\sqrt[\infty]{(b)}}^{-1}\} \ar[dr]^{t^{a,b}_{\mathfrak{p}}} &   \\
      & \left( \frac{A}{\mathfrak{p}} \right)\{ \frac{A}{\mathfrak{p}} \setminus \{ 0 + \mathfrak{p}\}^{-1}\}\\
\left( \frac{A}{\sqrt[\infty]{(d)}}\right)\{ {d+\sqrt[\infty]{(d)}}^{-1}\} \ar[uu]^{f^{c,d}_{a,b}} \ar[ur]_{t^{c,d}_{\mathfrak{p}}} &
}
\end{equation}

We claim that:

$$\substack{\displaystyle\varprojlim \\ V \in \mathcal{B}({\rm Spec}^{\infty-{\rm const}}\,(A)) } F_A(V) \cong \left( \frac{A}{\mathfrak{p}}\right)\{ {\frac{A}{\mathfrak{p}} \setminus \lbrace 0 + \mathfrak{p}\}}^{-1} \rbrace$$

In order to prove it, we must have:

\begin{itemize}
  \item[(i)]{The diagram \eqref{Valinhos} is a commutative cone;}
  \item[(ii)]{The family of arrows $\{ t^{a,b}_{\mathfrak{p}}\}_{a,b \in U}$ is collectively surjective;}
  \item[(iii)]{ For every $x_{a,b} \in \left( \frac{A}{\sqrt[\infty]{(b)}}\right)\{{a +\sqrt[\infty]{(b)}}^{-1} \}$ and $x_{a',b'} \in \left( \frac{A}{\sqrt[\infty]{(b')}}\right)\{{a' +\sqrt[\infty]{(b')}}^{-1} \}$ such that:
      $$t^{a,b}_{\mathfrak{p}}(x_{a,b}) = t^{a',b'}_{\mathfrak{p}}(x_{a',b'})$$
      there are $c,d \in A$ such that:
      $$D^{\infty}(c)\cap Z^{\infty}(d) \subseteq [D^{\infty}(a) \cap Z^{\infty}(b)] \cap [D^{\infty}(a') \cap Z^{\infty}(b')],$$
  i.e., there are arrows:
  $$f^{a,b}_{c,d}: \left( \frac{A}{\sqrt[\infty]{(b)}}\right)\{{a +\sqrt[\infty]{(b)}}^{-1} \} \to \left( \frac{A}{\sqrt[\infty]{(d)}}\right)\{{c +\sqrt[\infty]{(d)}}^{-1} \}$$
  and
  $$f^{a',b'}_{c,d}: \left( \frac{A}{\sqrt[\infty]{(b')}}\right)\{{a' +\sqrt[\infty]{(b')}}^{-1} \} \to \left( \frac{A}{\sqrt[\infty]{(d)}}\right)\{{c +\sqrt[\infty]{(d)}}^{-1} \}$$
 such that:

 $$f^{a,b}_{c,d}(x_{a,b}) = f^{a',b'}_{c,d}(x_{a',b'})$$
 that is, the following diagram commutes:

  $$\xymatrixcolsep{3pc}\xymatrix{
 \left( \frac{A}{\sqrt[\infty]{(b)}}\right) \{ {a+ \sqrt[\infty]{(b)}}^{-1}\} \ar[dr]^{f^{a,b}_{c,d}} \ar@/^2pc/[rrd]^{t^{a,b}_{\mathfrak{p}}} &   &   \\
      & \left( \frac{A}{\sqrt[\infty]{(d)}}\right)\{ {c+\sqrt[\infty]{(d)}}^{-1}\} \ar[r]^{t^{c,d}_{\mathfrak{p}}} & \left( \frac{A}{\mathfrak{p}}\right)\{ \frac{A}{\mathfrak{p}} \setminus \{ 0 + \mathfrak{p}\}^{-1} \}\\
 \left( \frac{A}{\sqrt[\infty]{(b')}}\right)\{ {a'+\sqrt[\infty]{(b')}}^{-1}\} \ar[ur]_{f^{a',b'}_{c,d}} \ar@/_2pc/[rru]_{t^{a',b'}_{\mathfrak{p}}} &   &
 }$$}

The latter condition is equivalent to saying that for every $a,b \in A$, the morphism:
$$\begin{array}{cccc}
    t^{a,b}_{\mathfrak{p}}: & \left( \frac{A}{\sqrt[\infty]{(b)}}\right)\{ a + \sqrt[\infty]{(b)}^{-1}\} & \rightarrow & \left( \frac{A}{\mathfrak{p}}\right)\{ \frac{A}{\mathfrak{p}}\setminus \{ 0 + \mathfrak{p} \}^{-1}\} \\
     & \gamma & \mapsto & t^{a,b}_{\mathfrak{p}}(\gamma)
  \end{array}$$

is injective, that is, $\ker t^{a,b}_{\mathfrak{p}} = \{ 0\}$.\\

Since $t^{a,b}_{\mathfrak{p}} = \jmath_a \circ \eta_{b,\mathfrak{p}}$ and $\jmath_a$ is injective, $\ker t^{a,b}_{\mathfrak{p}} = \ker \eta_{b,\mathfrak{p}}$ and showing that $t^{a,b}_{\mathfrak{p}}$ is injective is equivalent to prove that $ \ker \eta_{b,\mathfrak{p}} \{ 0\}$.\\

Consider the following diagram:

$$\xymatrixcolsep{3pc}\xymatrix{
 & \left( \frac{A}{\sqrt[\infty]{(b)}} \right) \ar[r]^{\eta_{a + \sqrt[\infty]{(b)}}} \ar[dd]^{q_{b,\mathfrak{p}}} & \left( \frac{A}{\sqrt[\infty]{(b)}}\right)\{ {a + \sqrt[\infty]{(b)}}^{-1}\} \ar[dr]^{t^{a,b}_{\mathfrak{p}}} \ar[dd]^{\eta_{b,\mathfrak{p}}} &  \\
 A \ar[ur]^{q_{\sqrt[\infty]{(b)}}} \ar[dr]_{q_{\mathfrak{p}}} & & & \left( \frac{A}{\mathfrak{p}}\right)\{ \frac{A}{\mathfrak{p}} \setminus \{ 0 +\mathfrak{p} \}^{-1}\}\\
   & \left( \frac{A}{\mathfrak{p}}\right)  \ar[r]^{\eta_{a+\mathfrak{p}}} & \left( \frac{A}{\mathfrak{p}}\right)\{ {a+\mathfrak{p}}^{-1}\} \ar[ur]^{\jmath_a} &
}$$

We have:
$$t^{a,b}_{\mathfrak{p}} = \jmath_a \circ \eta_{b,\mathfrak{p}}$$
so
$$t^{a,b}_{\mathfrak{p}} \circ \eta_{a + \sqrt[\infty]{(b)}} = \jmath_a \circ \eta_{b,\mathfrak{p}} \circ \eta_{a + \sqrt[\infty]{(b)}}$$
Since $\eta_{b,\mathfrak{p}} \circ \eta_{a + \sqrt[\infty]{(b)}} = \eta_{a+\mathfrak{p}} \circ q_{b,\mathfrak{p}}$, it follows that:
$$t^{a,b}_{\mathfrak{p}} \circ \eta_{a + \sqrt[\infty]{(b)}} = \jmath_a \circ \eta_{a+\mathfrak{p}}\circ q_{b,\mathfrak{p}}$$
$$ t^{a,b}_{\mathfrak{p}} \circ \eta_{a + \sqrt[\infty]{(b)}} \circ q_{\sqrt[\infty]{(b)}} = \jmath_a \circ \eta_{a+\mathfrak{p}}\circ \underbrace{q_{b,\mathfrak{p}} \circ q_{\sqrt[\infty]{(b)}}}_{=q_{\mathfrak{p}}}$$
hence
\begin{equation}\label{Desree}t^{a,b}_{\mathfrak{p}} \circ \eta_{a + \sqrt[\infty]{(b)}} \circ q_{\sqrt[\infty]{(b)}} = \jmath_a \circ \eta_{a+\mathfrak{p}}\circ q_{\mathfrak{p}}
\end{equation}

Let $\gamma \in \ker t^{a,b}_{\mathfrak{p}} \subseteq \left( \frac{A}{\sqrt[\infty]{(b)}}\right)\{ a + \sqrt[\infty]{(b)}^{-1}\}$, so there are $\alpha + \sqrt[\infty]{(b)},  \in \left( \frac{A}{\sqrt[\infty]{(b)}}\right)$ and $\beta  + \sqrt[\infty]{(b)} \in \{ a + \sqrt[\infty]{(b)}\}^{\infty-{\rm sat}}$ such that $\gamma = \dfrac{\eta_{a + \sqrt[\infty]{(b)}}(\alpha + \sqrt[\infty]{(b)})}{\eta_{a + \sqrt[\infty]{(b)}}(\beta  + \sqrt[\infty]{(b)})}$.\\

Now, $t^{a,b}_{\mathfrak{p}}(\gamma) = 0$ if, and only if, $(t^{a,b}_{\mathfrak{p}} \circ \eta_{a + \sqrt[\infty]{(b)}} \circ q_{\sqrt[\infty]{(b)}})(\alpha \cdot \beta^{-1}) = 0$. By \eqref{Desree}, this last condition is equivalent to:
$$\jmath_a(\eta_{a+\mathfrak{p}}(q_{\mathfrak{p}}(\alpha \cdot \beta^{-1}))) = \jmath (\eta_{a+\mathfrak{p}} \circ q_{\mathfrak{p}}(\alpha) \cdot \eta_{a+\mathfrak{p}}\circ q_{\mathfrak{p}}(\beta)^{-1})$$

so

$$(\eta_{a+\mathfrak{p}} \circ q_{\mathfrak{p}})(\alpha) \in \ker \jmath_a = \{ 0 + \mathfrak{p}\}$$
and
$$\eta_{a+\mathfrak{p}}(\alpha + \mathfrak{p}) = 0 + \mathfrak{p}.$$

However, $\eta_{a+\mathfrak{p}}(\alpha + \mathfrak{p}) = 0 \iff (\exists \delta + \mathfrak{p} \in \{ a + \mathfrak{p}\}^{\infty-{\rm sat}})((\delta + \mathfrak{p})\cdot (\alpha + \mathfrak{p}) = 0 + \mathfrak{p})$, equivalently, $\alpha \cdot \delta \in \mathfrak{p}$. Since $\mathfrak{p}$ is a prime ideal, $\alpha \cdot \delta \in \mathfrak{p} \to (\alpha \in \mathfrak{p}) \vee (\delta \in \mathfrak{p})$.\\

We claim that $\alpha \in \mathfrak{p}$. \textit{Ab absurdo}, suppose $\alpha \notin \mathfrak{p}$, so we must have $\delta \in \mathfrak{p}$. However, $\delta + \mathfrak{p} \in \{ a + \mathfrak{p}\}^{\infty-{\rm sat}}$, so $\eta_{a+\mathfrak{p}}(\delta + \mathfrak{p}) \in \left( \frac{A}{\mathfrak{p}}\right)\{ {a+\mathfrak{p}}^{-1}\}^{\times}$. But since $\delta \in \mathfrak{p}$, it would follow that $\eta_{a+\mathfrak{p}}(\delta + \mathfrak{p}) = \eta_{a+\mathfrak{p}}(0+\mathfrak{p}) \in \left( \frac{A}{\mathfrak{p}}\right)\{ {a+\mathfrak{p}}^{-1}\}^{\times}$, from which would follow that $\left( \frac{A}{\mathfrak{p}}\right)\{ {a+\mathfrak{p}}^{-1}\} \cong \{ 0\}$, a contradiction.\\

Thus we have $\alpha \in \mathfrak{p}$, so $\alpha  + \sqrt[\infty]{(b)}$, so $\alpha  + \sqrt[\infty]{(b)} \in q_{\sqrt[\infty]{(b)}}[\mathfrak{p}]$ and $q_{b, \mathfrak{p}}(\alpha  + \sqrt[\infty]{(b)}) = 0 + \mathfrak{p}$. Hence:

$$\eta_{b, \mathfrak{p}}(\eta_{a  + \sqrt[\infty]{(b)}}(\alpha  + \sqrt[\infty]{(b)})) = \eta_{a+\mathfrak{p}}(\alpha + \mathfrak{p}) = 0$$
and the result follows.\\

\end{itemize}

Now we are going to compute the fibers of the above presheaf.\\

\begin{theorem}Consider the presheaf defined on the basis of the constructible topology:
$$\begin{array}{cccc}
    F_A: & {\mathcal{B}}\,({\rm Spec}^{\infty-{\rm const}}\,(A))^{{\rm op}} & \rightarrow & \mathcal{C}^{\infty}{\rm Rng} \\
     & U = D^{\infty}(a)\cap Z^{\infty}(b) & \mapsto & F_A(U) = \varinjlim_{U=D^{\infty}(a)} \left( \frac{A}{\sqrt[\infty]{(b)}}\right)\{ (a + \sqrt[\infty]{(b)})^{-1}\} \\
     & \imath^{V}_{U}: U \hookrightarrow V & \mapsto & \rho^{V}_{U}: F_{A}(U) \to F_A(V)
  \end{array}$$

The stalk of $F_A$ at a point $\mathfrak{p} \in {\rm Spec}^{\infty}\,(A)$ is isomorphic to the $\mathcal{C}^{\infty}-$field $\left( \dfrac{A}{\mathfrak{p}}\right)\left\{ \left( \dfrac{A}{\mathfrak{p}} \setminus \{ 0 + \mathfrak{p}\}^{-1}\right)\right\}$.
\end{theorem}
\begin{proof}
Let $L = \varinjlim_{\mathfrak{p} \in U_{a,b}} \dfrac{A}{\sqrt[\infty]{(b)}}\{ (a + \sqrt[\infty]{(b)})^{-1}\}$.\\

Consider the following commutative diagram:\\

$$\xymatrixcolsep{3pc}\xymatrix{
A \ar[dr]_{t} \ar[r]^{\eta_{A \setminus \mathfrak{p}}} & A\{ {A \setminus \mathfrak{p}}^{-1}\}\\
  & L}$$

where $t = \jmath_{a,b} \circ \eta_{a + \sqrt[\infty]{(b)}} \circ q_{\sqrt[\infty]{(b)}}$.\\

It is clear that $(x \in A \setminus \mathfrak{p}) \to (t(x) \in L^{\times})$. Now, since $\mathfrak{p} \in {\rm Spec}^{\infty}\,(A)$, we have ${A \setminus \mathfrak{p}} = {A \setminus \mathfrak{p}}^{\infty-{\rm sat}}$, so there is a unique $\mathcal{C}^{\infty}-$homomorphism $t_{\mathfrak{p}}: A \{ {A \setminus \mathfrak{p}}^{-1}\} \to L$ such that the following triangle commutes:

$$\xymatrixcolsep{3pc}\xymatrix{
A \ar[dr]_{t} \ar[r]^{\eta_{A \setminus \mathfrak{p}}} & A\{ {A \setminus \mathfrak{p}}^{-1}\} \ar[d]^{t_{\mathfrak{p}}}\\
  & L}$$

In what follows we are going to show that $t_{\mathfrak{p}}$ is a $\mathcal{C}^{\infty}-$isomorphism.\\

Consider the diagram:\\

$$\xymatrixcolsep{3pc}\xymatrix{
A \ar@/_2pc/[dd]_{t_{a,b}} \ar[d]^{q_{\sqrt[\infty]{(b)}}} \ar[dr]^{t} \ar[r]^{\eta_{A \setminus \mathfrak{p}}} & A \{ {A \setminus \mathfrak{p}}^{-1}\} \ar[d]_{t_{\mathfrak{p}}}\\
\dfrac{A}{\sqrt[\infty]{(b)}} \ar[d]^{\eta_{a + \sqrt[\infty]{(b)}}} & L \\
\dfrac{A}{\sqrt[\infty]{(b)}}\{ (a + \sqrt[\infty]{(b)})^{-1}\} \ar[ur]^{\jmath_{a,b}} &
}$$

Now we claim that:

$t_{\mathfrak{p}}: A\{ {A \setminus \mathfrak{p}}^{-1}\} \to L$ is a surjective map, \textit{i.e.},

$$(\forall \gamma \in L)(\exists x \in A)(\exists y \in A \setminus \mathfrak{p})\left( \gamma = \frac{t(x)}{t(y)} = t_{\mathfrak{p}}\left( \frac{\eta_{\mathfrak{p}}(x)}{\eta_{\mathfrak{p}}(y)}\right)\right).$$

Now, given $\gamma \in L$, since $\jmath_{a,b}$ is surjective there must exist $\alpha_{a,b} \in \dfrac{A}{\sqrt[\infty]{(b)}}\{ (a + \sqrt[\infty]{(b)})^{-1}\}$ such that $\jmath_{a,b}(\alpha_{a,b}) = \gamma$.\\

Since $\alpha_{a,b} \in \dfrac{A}{\sqrt[\infty]{(b)}}\{ (a + \sqrt[\infty]{(b)})^{-1}\}$, there are $c + \sqrt[\infty]{(b)} \in \dfrac{A}{\sqrt[\infty]{(b)}}$ and $d + \sqrt[\infty]{(b)} \in \{ a + \sqrt[\infty]{(b)}\}^{\infty-{\rm sat}}$ such that $\alpha_{a,b} = \dfrac{\eta_{a+\sqrt[\infty]{(b)}}(c + \sqrt[\infty]{(b)})}{\eta_{a + \sqrt[\infty]{(b)}}(d + \sqrt[\infty]{(b)})}$.\\

Now, $d + \sqrt[\infty]{(b)} \in \{ a + \sqrt[\infty]{(b)}\}^{\infty-{\rm sat}} \iff (\exists z \in A)(\exists s \in \{ a\}^{\infty-{\rm sat}})(s \cdot (d \cdot z - 1) \in \sqrt[\infty]{(b)})$, and since $\sqrt[\infty]{(b)} \subseteq \mathfrak{p}$, we have:

$$(\exists z \in A)(\exists s \in \{ a\}^{\infty-{\rm sat}})(s \cdot (d \cdot z - 1) \in \mathfrak{p}).$$

Since $s \in \{ a\}^{\infty-{\rm sat}} \subseteq {A \setminus \mathfrak{p}}^{\infty-{\rm sat}} = A \setminus \mathfrak{p}$, it follows that $s \notin \mathfrak{p}$, and since $s \cdot (d \cdot z - 1) \in \sqrt[\infty]{(b)} \subseteq \mathfrak{p}$ and $\mathfrak{p}$ is a prime ideal, either $s \in \mathfrak{p}$ or $d \cdot z - 1 \in \mathfrak{p}$. It is not the case that $s \in \mathfrak{p}$  so $d \cdot z - 1 \in \mathfrak{p}$. However, if $d \in \mathfrak{p}$ we would have $d \cdot z - 1 \in \mathfrak{p}$, hence $1 \in \mathfrak{p}$ and $\mathfrak{p}$ would not be a proper ideal, hence $d \in A \setminus \mathfrak{p} = {A \setminus \mathfrak{p}}^{\infty-{\rm sat}}$.\\

It suffices to take:

$$\alpha_{a,b} = \frac{t_{a,b}(c)}{t_{a,b}(d)},$$

so
$$\gamma = \jmath_{a,b}(\alpha_{a,b}) = \jmath_{a,b}\left( \frac{t_{a,b}(c)}{t_{a,b}(d)}\right) = \frac{t(c)}{t(d)} = t_{\mathfrak{p}}\left( \frac{\eta_{A \setminus \mathfrak{p}}(c)}{\eta_{A \setminus \mathfrak{p}}(d)}\right)$$

and $t_{\mathfrak{p}}$ is surjective.\\

\textbf{Claim:} $\m_{\mathfrak{p}} \subseteq \ker t_{\mathfrak{p}}$.\\

Let $\frac{\eta_{A \setminus \mathfrak{p}}(x)}{\eta_{A \setminus \mathfrak{p}}(y)} \in \m_{\mathfrak{p}}$, with $x \in \mathfrak{p}$ and $y \in A \setminus \mathfrak{p}$ (cf. \textbf{Proposition 20}, p. 69 of \cite{BM2}). We have that  $t_{\mathfrak{p}}\left( \frac{\eta_{A \setminus \mathfrak{p}}(x)}{\eta_{A \setminus \mathfrak{p}}(y)}\right)$.\\

Now we claim that $t(x)=0$, from what follows that:

$$t_{\mathfrak{p}} \left( \frac{\eta_{A \setminus \mathfrak{p}}(x)}{\eta_{A \setminus \mathfrak{p}}(y)}\right) = \frac{t(x)}{t(y)} = 0.$$

Since the following diagram commutes for every $a \in A \setminus \mathfrak{p}$ and $b \in \mathfrak{p}$:

$$\xymatrixcolsep{3pc}\xymatrix{
A \ar[dr]^{t} \ar[d]_{t_{a,b}}& \\
\dfrac{A}{\sqrt[\infty]{(b)}}\{ (a + \sqrt[\infty]{(b)})^{-1}\} \ar[r]^{\jmath_{a,b}} & L} $$

in particular, making $a = 1 \in A \setminus \mathfrak{p}$ and $b = x \in \mathfrak{p}$ yields the following commutative diagram:

$$\xymatrixcolsep{3pc}\xymatrix{
A \ar[dr]^{t} \ar[d]_{t_{1,x}}& \\
\dfrac{A}{\sqrt[\infty]{(x)}}\{ (1 + \sqrt[\infty]{(x)})^{-1}\} \ar[r]^{\jmath_{1,x}} & L}$$

But $t(x) = \jmath_{1,x}(t_{1,x}(x)) = \jmath_{1,x}(\eta_{1 + \sqrt[\infty]{(x)}}(0 + \sqrt[\infty]{(x)})) = 0$, hence $\frac{\eta_{A \setminus \mathfrak{p}}(x)}{\eta_{A \setminus \mathfrak{p}}(y)} \in \ker t_{\mathfrak{p}}$.\\

By the \textbf{Theorem of the Homomorphism}, there exists a unique $\overline{t_{\mathfrak{p}}}: \dfrac{A_{\mathfrak{p}}}{\m_{\mathfrak{p}}} \to L$ such that:

$$\xymatrixcolsep{3pc}\xymatrix{
A_{\mathfrak{p}} \ar@{->>}[r]^{q_{\m_{\mathfrak{p}}}} \ar@{->>}[d]_{t_{\mathfrak{p}}} & \dfrac{A_{\mathfrak{p}}}{\m_{\mathfrak{p}}} \ar[dl]^{\overline{t_{\mathfrak{p}}}}\\
L &
}$$

and since $t_{\mathfrak{p}} = \overline{t_{\mathfrak{p}}} \circ q_{\m_{\mathfrak{p}}}$, it follows that $\overline{t_{\mathfrak{p}}}$ is surjective, hence an epimorphism.\\

In order to show that $\overline{t_{\mathfrak{p}}}$ is a $\mathcal{C}^{\infty}-$isomorphism, since it is already an epimorphism, it suffices to prove that it admits a section.\\

For every $a,b \in A$ such that $b \in \mathfrak{p}$ and $a \notin \mathfrak{p}$, consider the following diagram:

$$\xymatrixcolsep{7pc}\xymatrix{
A \ar[r]^{q_{\sqrt[\infty]{(b)}}} \ar@{->>}[rd]_{q_{\mathfrak{p}}} & \dfrac{A}{\sqrt[\infty]{(b)}} \ar@{->>}[d]^{q_{b,\mathfrak{p}}} \ar[r]^{\eta_{a + \sqrt[\infty]{(b)}}} & \dfrac{A}{\sqrt[\infty]{(b)}}\{ (a + \sqrt[\infty]{(b)})^{-1}\} \\
  & \dfrac{A}{\mathfrak{p}} \ar[r]^{\eta_{\frac{A}{\mathfrak{p}}\setminus \{ 0 + \mathfrak{p} \}}} & \left( \dfrac{A}{\mathfrak{p}} \right)\left\{ \dfrac{A}{\mathfrak{p}} \setminus \{ 0 + \mathfrak{p}\}^{-1}  \right\}
}$$

%Since $a \notin \mathfrak{p}$ and $\sqrt[\infty]{(b)} \subseteq \mathfrak{p}$, \textit{a fortiori} we have $a \notin \sqrt[\infty]{(b)}$, so $\eta_{\frac{A}{\mathfrak{p}}\setminus \{ 0 + \mathfrak{p} \} + \mathfrak{p}} \circ q_{b,\mathfrak{p}}(a + \sqrt[\infty]{(b)}) \in \left( \left( \dfrac{A}{\mathfrak{p}} \right)\left\{ \dfrac{A}{\mathfrak{p}} \setminus \{ 0 + \mathfrak{p}\} + \mathfrak{p}^{-1}  \right\}\right)^{\times}$ and by the universal property of $\eta_{a + \sqrt[\infty]{(b)}} : \frac{A}{\sqrt[\infty]{(b)}} \to \frac{A}{\sqrt[\infty]{(b)}}\{ (a + \sqrt[\infty]{(b)})^{-1}\}$ there is a unique $\zeta_{a,b}: \frac{A}{\sqrt[\infty]{(b)}}\{ (a + \sqrt[\infty]{(b)})^{-1}\} \to \left( \dfrac{A}{\mathfrak{p}} \right)\left\{ \dfrac{A}{\mathfrak{p}} \setminus \{ 0 + \mathfrak{p}\} + \mathfrak{p}^{-1}  \right\}$ such that the following triangle commutes:

%$$\xymatrixcolsep{3pc}\xymatrix{
%\dfrac{A}{\sqrt[\infty]{(b)}} \ar[dr]_{\eta_{\frac{A}{\mathfrak{p}}\setminus \{ 0 + \mathfrak{p} \} + \mathfrak{p}} \circ q_{b,\mathfrak{p}}}\ar[r]^{\eta_{a + \sqrt[\infty]{(b)}}} & \dfrac{A}{\sqrt[\infty]{(b)}} \{ (a+ \sqrt[\infty]{(b)})^{-1}\} \ar[d]^{\zeta_{a,b}}\\
%  & \left( \dfrac{A}{\mathfrak{p}} \right)\left\{ \dfrac{A}{\mathfrak{p}} \setminus \{ 0 + \mathfrak{p}\} + \mathfrak{p}^{-1}  \right\}
%}$$

We have the following commutative diagram:

$$\xymatrixcolsep{8pc}\xymatrix{
A \ar[r]^{q_{\sqrt[\infty]{(b)}}} \ar@{->>}[rd]_{q_{\mathfrak{p}}} & \dfrac{A}{\sqrt[\infty]{(b)}} \ar@{->>}[d]^{q_{b,\mathfrak{p}}} \ar[r]^{\eta_{a + \sqrt[\infty]{(b)}}} & \dfrac{A}{\sqrt[\infty]{(b)}}\{ (a + \sqrt[\infty]{(b)})^{-1}\} \ar@{.>}[d]^{\eta_{b,\mathfrak{p}}}\\
  & \dfrac{A}{\mathfrak{p}} \ar[r]^{\eta_{\frac{A}{\mathfrak{p}}\setminus \{ 0 + \mathfrak{p} \} + \mathfrak{p}}} & \left( \dfrac{A}{\mathfrak{p}} \right)\left\{ \dfrac{A}{\mathfrak{p}} \setminus \{ 0 + \mathfrak{p}\} + \mathfrak{p}^{-1}  \right\}
}$$

Let $a,b,c,d \in A$ be such that $U_{c,d} \subseteq U_{a,b}$, so we have the following diagram:
$$\xymatrixcolsep{3pc}\xymatrix{
 & \left( \dfrac{A}{\mathfrak{p}}\right) \left\{ \dfrac{A}{\mathfrak{p}}\setminus \{ 0 + \mathfrak{p}\}^{-1}\right\}& \\
 \left(\dfrac{A}{\sqrt[\infty]{(b)}}\right)\{ (a + \sqrt[\infty]{(b)})\} \ar[ur]_{t^{a,b}_{\mathfrak{p}}} \ar[rr]^{t^{c,d}_{a,b}} & & \left(\dfrac{A}{\sqrt[\infty]{(d)}}\right)\{ (c + \sqrt[\infty]{(d)})\} \ar[ul]^{t^{c,d}_{\mathfrak{p}}}
}$$

By the universal property of the inductive limit $L$, there must exist a unique
$$\Psi_{\mathfrak{p}}: L \to \left( \dfrac{A}{\mathfrak{p}} \right)\left\{ \dfrac{A}{\mathfrak{p}} \setminus \{ 0 + \mathfrak{p}\}^{-1}  \right\}$$
such that the following diagram commutes:

$$\xymatrixcolsep{3pc}
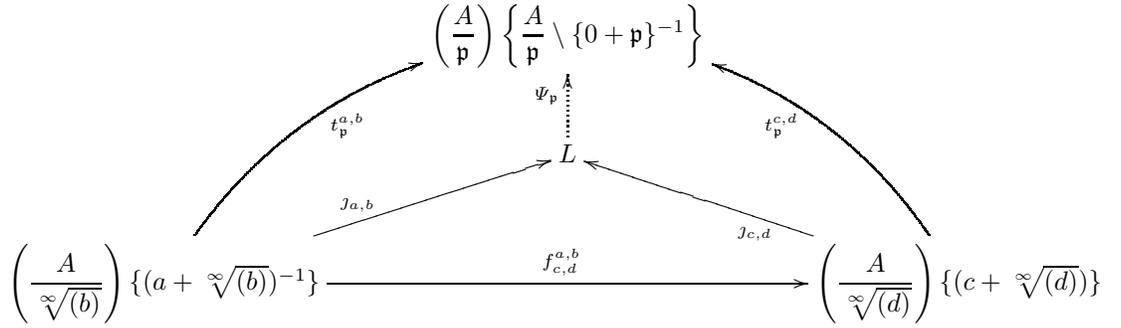
\xymatrix{
 & \left( \dfrac{A}{\mathfrak{p}}\right) \left\{ \dfrac{A}{\mathfrak{p}}\setminus \{ 0 + \mathfrak{p}\}^{-1}\right\} &  \\
 & L \ar@{.>}[u]^{\Psi_{\mathfrak{p}}} & \\
\left(\dfrac{A}{\sqrt[\infty]{(b)}}\right)\{ (a + \sqrt[\infty]{(b)})^{-1}\} \ar@/^2pc/[uur]_{t^{a,b}_{\mathfrak{p}}} \ar[ur]^{\jmath_{a,b}} \ar[rr]^{f^{a,b}_{c,d}} & & \left(\dfrac{A}{\sqrt[\infty]{(d)}}\right)\{ (c + \sqrt[\infty]{(d)})\} \ar@/_2pc/[uul]^{t^{c,d}_{\mathfrak{p}}} \ar[ul]^{\jmath_{c,d}}
}$$

Finally we have the following diagram:\\

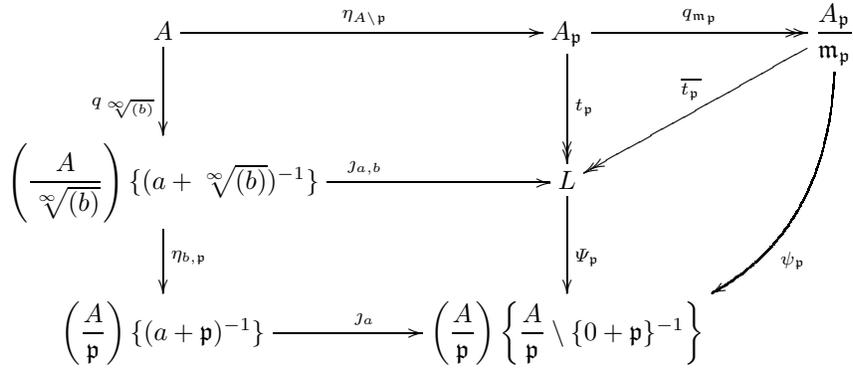
$$\xymatrixcolsep{3pc}\xymatrix{
A \ar[r]^{\eta_{A\setminus \mathfrak{p}}} \ar[d]_{q_{\sqrt[\infty]{(b)}}} & A_{\mathfrak{p}} \ar@{->>}[d]^{t_{\mathfrak{p}}} \ar@{->>}[r]^{q_{\m_{\mathfrak{p}}}} & \dfrac{A_{\mathfrak{p}}}{\m_{\mathfrak{p}}} \ar@{->>}[dl]_{\overline{t_{\mathfrak{p}}}}\ar@/^3pc/[ddl]^{\psi_{\mathfrak{p}}}\\
\left( \dfrac{A}{\sqrt[\infty]{(b)}}\right)\{ (a + \sqrt[\infty]{(b)})^{-1}\} \ar[d]^{\eta_{b,\mathfrak{p}}}\ar[r]^{\jmath_{a,b}} & L \ar[d]^{\Psi_{\mathfrak{p}}}& \\
\left( \dfrac{A}{\mathfrak{p}}\right)\{(a+\mathfrak{p})^{-1} \} \ar[r]^{\jmath_a} & \left( \dfrac{A}{\mathfrak{p}}\right)\left\{ \dfrac{A}{\mathfrak{p}} \setminus \{ 0 + \mathfrak{p}\}^{-1}\right\}&
}$$

where $\psi_{\mathfrak{p}}$ is given by \textbf{Theorem 24}, p. 96 of \cite{BM2}.\\

\textbf{Claim:} The following triangle commutes:

$$\xymatrixcolsep{3pc}\xymatrix{
 & \dfrac{A_{\mathfrak{p}}}{\m_{\mathfrak{p}}} \ar[dl]_{\overline{t_{\mathfrak{p}}}} \ar@/^1pc/[ddl]^{\psi_{\mathfrak{p}}}\\
L \ar[d]^{\Psi_{\mathfrak{p}}} & \\
\left( \dfrac{A}{\mathfrak{p}}\right) \left\{ \dfrac{A}{\mathfrak{p}}\setminus \{ 0 + \mathfrak{p}\}^{-1}\right\} &
 }$$

Since $q_{\m_{\mathfrak{p}}}$ is an epimorphism, we have:
$$\Psi_{\mathfrak{p}} \circ t_{\mathfrak{p}} = \Psi_{\mathfrak{p}} \circ \overline{t_{\mathfrak{p}}} \iff (\Psi_{\mathfrak{p}} \circ t_{\mathfrak{p}}) \circ q_{\m_{\mathfrak{p}}} = (\Psi_{\mathfrak{p}} \circ \overline{t_{\mathfrak{p}}}) \circ q_{\m_{\mathfrak{p}}}$$

Now, by the universal property of $\eta_{A \setminus \mathfrak{p}}: A \to A_{\mathfrak{p}}$, there is \textit{a unique} $\mathcal{C}^{\infty}-$homomorphism $\xi : A_{\mathfrak{p}} \to \left( \dfrac{A}{\mathfrak{p}}\right) \left\{ \dfrac{A}{\mathfrak{p}}\setminus \{ 0 + \mathfrak{p}\}^{-1}\right\}$ such that:
$$\xymatrixcolsep{3pc}\xymatrix{A \ar[dr]_{\jmath_a \circ \eta_{b,\mathfrak{p}} \circ t_{a,b}}\ar[r]^{\eta_{A\setminus \mathfrak{p}}} & A_{\mathfrak{p}} \ar[d]^{\xi}\\
  & \left( \dfrac{A}{\mathfrak{p}}\right) \left\{ \dfrac{A}{\mathfrak{p}}\setminus \{ 0 + \mathfrak{p}\}^{-1}\right\}}$$
commutes, so $(\Psi_{\mathfrak{p}} \circ t_{\mathfrak{p}}) \circ q_{\m_{\mathfrak{p}}} = \xi =  (\Psi_{\mathfrak{p}} \circ \overline{t_{\mathfrak{p}}}) \circ q_{\m_{\mathfrak{p}}}$ and the former triangle commutes.\\

Thus, since $\psi_{\mathfrak{p}} = \Psi_{\mathfrak{p}} \circ \overline{t_{\mathfrak{p}}}$ where $\overline{t_{\mathfrak{p}}}$ is an epimorphism which is also a section, it follows that $\overline{t_{\mathfrak{p}}}$ and $\Psi_{\mathfrak{p}}$ are both $\mathcal{C}^{\infty}-$isomorphisms, as we wanted to show.
\end{proof}

%Given any $\mathcal{C}^{\infty}-$ring $A$, we are going to write ${\rm vN}\,(A)$ to denote $A_{\omega}$ and $\nu: A \to {\rm vN}\,(A)$ to denote the $\mathcal{C}^{\infty}-$homomorphism $\nu_0: A \to A_{\omega}$.\\

\begin{remark} \label{vn-estavel} If $A$ is a {\em von Neumann regular} $C^\infty$-ring,  then for each $a, b \in V$, let $e, f$ (\textit{ditto}) be the {\em unique} idempotents of $V$ such that $(a) =(e)$ and $(b) = (f)$. Then:
$$D^\infty(a) \cap Z^\infty(b) = D^\infty(e) \cap Z^\infty(f) = D^\infty(e) \cap D^\infty(1-f) = D^\infty(e.(1-f))$$
and
$$\left(\frac{A}{\sqrt[\infty]{(b)}}\right)\{ (a + \sqrt[\infty]{(b)})^{-1}\} = \left( \frac{A}{\sqrt[\infty]{(f)}}\right)\{ (e + \sqrt[\infty]{(f)})^{-1}\} =$$
$$ = \left( \dfrac{A}{(f)} \right) \{ (e + {(f)})^{-1}\} \cong A\{(1-f)^{-1}\}\{\eta_{1-f}(e)^{-1}\} \cong A\{((1-f)\cdot e)^{-1}\}$$
Thus the (pre)sheaf basis $F_A$ is isomorphic to the ``affine'' sheaf basis $\Sigma_A$.

\end{remark}

Now we proceed as follows:\\

$\bullet$ For each $C^\infty$-ring $A$, $F_A: (\mathcal{B}({\rm Spec}^{\infty-{\rm const}}\,(A)), \subseteq)^{\rm op} \rightarrow \mathcal{C}^{\infty}{\rm \bf Rng}$ is a presheaf on the (canonical) basis of the constructible topology in ${\rm Spec}^\infty(A)$ and with  the stalks $\cong k_{\p}(A)$, $\p \in {\rm Spec}^\infty(A)$. Moreover there is a canonical $C^\infty$-homomorphism $\phi_A : A \to F_A({\rm Spec}^{\infty-{\rm const}}\,(A))) \cong A/\sqrt[\infty]{(0)}$, in such a way that we obtain a functor $A \overset{\mathbb{F}}{\mapsto}   F_A({\rm Spec}^{\infty-{\rm const}}\,(A)))$ and a natural transformation $\phi : {\rm id} \Rightarrow \mathbb{F}$.\\

$\bullet$ For each $C^\infty$-ring $A$, take $G_A$ the sheaf-basis functor associated to $F_A$ - both are defined on the basis of the constructible topology of the $C^\infty$-spectrum of $A$ and both share the same (up to isomorphisms) stalks. Moreover, since there is a natural transformation $F_A \to G_A$,  there is a canonical $C^\infty$-homomorphism $\gamma_A : A \to G_A({\rm Spec}^{\infty-{\rm const}}\,(A)))$, in such a way that we obtain a functor $A \overset{\mathbb{G}}{\mapsto}   G_A({\rm Spec}^{\infty-{\rm const}}\,(A)))$ and a natural transformation $\gamma : id \to \mathbb{G}$.\\

$\bullet$ For each $C^\infty$-ring $A$, take $H_A$ the (unique up to isomorphism) sheaf extension $G_A$ to the basis of constructible topology  into all the opens sets of this topology, thus $H_A \upharpoonright \cong G_A$, and   both functor keeps the same stalks (up to isomorphism).
Moreover, since there is a natural transformation (isomorphism) $G_A \to H_A\upharpoonright$,  there is a canonical $C^\infty$-homomorphism $\mu_A : A \to H_A({\rm Spec}^{\infty-{\rm const}}\,(A)))$, in such a way that we obtain a functor $A \overset{\mathbb{H}}{\mapsto}   H_A({\rm Spec}^{\infty-{\rm const}}\,(A)))$ and a natural transformation $\mu : {\rm id} \Rightarrow \mathbb{H}$.\\

$\bullet$ Since $H_A$ is a sheaf of $C^\infty$-rings, defined over all the opens of a boolean space, and whose stalks are $C^\infty$-fields, then $\mathbb{H}(A) = H_A({\rm Spec}^{\infty-{\rm const}}\,(A)))$ is a Von Neumann regular $C^\infty$-ring. Thus $A \overset{\mathbb{H}}{\mapsto}   H_A({\rm Spec}^{\infty-{\rm const}}\,(A)))$ determines a functor $C^\infty{\rm \bf Rng} \to C^\infty{\rm \bf vNRng}$ and we have a natural transformation $\mu : {\rm id} \Rightarrow i \circ \mathbb{H}$.\\

We are ready to (re)state the:

\begin{theorem}\label{bis}  {\bf (bis)} The inclusion functor $C^\infty{\rm \bf vNRng} \hookrightarrow C^\infty{\rm \bf Rng}$ has a left adjoint functor ${\rm  vN} : C^\infty{\rm \bf Rng} \to C^\infty{\rm \bf vNRng}$. In more details:  for $A$ be any $\mathcal{C}^{\infty}-$ring, the pair $(\mathbb{H}(A)), \mu_A: A \to \mathbb{H}(A)))$  is the $\mathcal{C}^{\infty}-$von Neumann regular hull of $A$, that is, for every von Neumann-regular $\mathcal{C}^{\infty}-$ring $V$ and for every $\mathcal{C}^{\infty}-$homomorphism $f: A \to V$ there is a unique $\mathcal{C}^{\infty}-$homomorphism $\widetilde{f}: A_{\omega} \rightarrow V$ such that the following diagram commutes:

$$\xymatrixcolsep{3pc}\xymatrix{
A \ar[r]^{\mu_A} \ar[dr]_{f} & \Gamma(\Sigma(A)) \ar@{-->}[d]^{\exists ! \widetilde{f}}\\
 & V
}$$

\end{theorem}
\begin{proof} We just have observed above that all the conditions in the \textbf{Theorem \ref{proposition12i}} hold.
\end{proof}

We finish this subsection with some general applications of the adjunction  above and one of its specific ``sheaf-theorectic'' construction of the left adjoint ${\rm vN} : C^\infty{\rm \bf Rng} \to C^\infty{\rm \bf vNRng}$.

\begin{proposition} \label{cor27AM} The functor ${\rm  vN}$ preserves all colimits. In particular it preserves:
\begin{itemize}
\item{directed inductive limits;}
\item{coproducts (= $\infty-$tensor products in ${\rm \bf C^\infty-Rng}$);}
\item{coequalizers/quotients.}
\end{itemize}
\end{proposition}
\begin{proof} Since it is a left adjoint, ${\rm  vN}$ preserves all colimits. We explain the meaning of the preservation of quotients. Consider the induced homomorphism $vN(q_I) : {\rm  vN}(A) \rightarrow {\rm  vN}(A/I)$: it is surjective since the coequalizers in ${\rm \bf C^\infty-Rng}$ and ${\rm \bf  C^\infty{\rm \bf vNRng}}$ coincide with the surjective homomorphisms. Then $\bar{I} := \ker({\rm  vN}(q_{I})) \subseteq {\rm  vN}(A)$ is such that $\overline{{\rm  vN}(q_{I})} :{\rm  vN}(A)/\bar{I} \stackrel{\cong}{\rightarrow} {\rm  vN}(A/I)$ and, since
$${\rm  vN}(A) \rightarrowtail \prod_{\mathfrak{p} \in {\rm  Spec^\infty}(A)} k_{\mathfrak{p}}(A),$$  $\bar{I}$ can be identified with   ${\rm  vN}(A) \cap \{ \vec{x} = (x_{\mathfrak{p}})_{\mathfrak{p} \in {\rm  Spec^\infty}(A)} \in \prod_{\mathfrak{p} \in {\rm  Spec^\infty}(A)} k_{\mathfrak{p}}(A) : \forall \mathfrak{p} \supseteq I, x_{\mathfrak{p}}=0 \in k_{\mathfrak{p}}(A)\}$. Note in particular, that if $I \subseteq \sqrt[\infty]{(0)}$, then  ${\rm  vN}(q_{I}) : {\rm  vN}(A) \stackrel{\cong}{\rightarrow} {\rm  vN}(A/I)$, thus $\bar{I} = \{0\}$.
\end{proof}

The following results are specific to the functor ${\rm  vN}$, i.e., they are not general consequences of it being a left adjoint functor.

\begin{proposition}\label{prp28MA} ${\rm  vN}$ preserves localizations. More precisely, given a $\mathcal{C}^{\infty}$-ring $A$ and a subset $S \subseteq A$, denote $S':= \eta_A[S] \subseteq {\rm  vN}(A)$ the corresponding  subset and let $\eta_A^S : A\{S^{-1}\} \rightarrow {\rm  vN}(A)\{{S'}^{-1}\}$ be the induced arrow, i.e., $\eta_A^S$ is the unique homomorphism such that $\eta_A^S \circ \sigma(A)_S = \sigma({\rm  vN}(A))_{S'} \circ \eta_A$. Then $\eta_{A\{S^{-1}\}} : A\{S^{-1}\} \rightarrow {\rm  vN}(A\{S\}^{-1})$ thus it is isomorphic to the arrow $\eta_A^S$, through the obvious  pair of inverse (iso)morphisms ${\rm  vN}(A\{S\}^{-1}) \rightleftarrows {\rm  vN}(A)\{S'\}^{-1}$.
\end{proposition}
\begin{proof}
First of all, note that ${\rm  vN}(A)\{S'\}^{-1}$ is a vN-regular ring, cf. \textbf{Proposition \ref{vNRingsClosedUnderLocalizations}}.
For each vN-regular ring $V$, the bijection \ $( - \circ \eta_A) : {\rm \bf  \mathcal{C}^{\infty}{\rm \bf vNRng}}({\rm  vN}(A),V) \stackrel{\cong}{\rightarrow} {\rm \bf \mathcal{C}^{\infty}-Rng}(A,V)$ \ restricts to the bijection:\\

$( - \circ \eta_A){\upharpoonright} : \{H \in {\rm \bf  \mathcal{C}^{\infty}{\rm \bf vNRng}}({\rm  vN}(A),V): H[S'] \subseteq V^{\times}\}$ \\ \noindent .  \hspace{33mm} $\ \ \ \ \ \ \ \ \ \ \stackrel{\cong}{\rightarrow} \{h \in {\rm \bf \mathcal{C}^{\infty}-Rng}(A,V): h[S]\subseteq V^{\times}\}.$

Composing the last bijection with the bijections below, obtained from the universal property of localizations,
$$ (- \circ \sigma(A)_S)^{-1}  : \{h \in {\rm \bf \mathcal{C}^{\infty}-Rng}(A,V) : h[S] \subseteq V^{\times}\} \stackrel{\cong}{\rightarrow}  {\rm \bf \mathcal{C}^{\infty}-Rng}(A[S]^{-1}, V),$$
\begin{multline*}( - \circ \sigma({\rm  vN}(A))_{S'})  : {\rm \bf \mathcal{C}^{\infty}-Rng}({\rm  vN}(A)\{S'^{-1}\},V) \stackrel{\cong}{\rightarrow}\\
\stackrel{\cong}{\rightarrow} \{H \in {\rm \bf \mathcal{C}^{\infty}-Rng}({\rm  vN}(A),V): H[S'] \subseteq V^{\times}\}\end{multline*}
we obtain, since ${\rm \bf  \mathcal{C}^{\infty}{\rm \bf vNRng}}$ is a full subcategory of ${\rm \bf \mathcal{C}^{\infty} Rng}$, the bijection
$$( - \circ \eta^S_A) : {\rm \bf  \mathcal{C}^{\infty}{\rm \bf vNRng}}({\rm  vN}(A)\{S'^{-1}\},V) \stackrel{\cong}{\rightarrow} {\rm \bf \mathcal{C}^{\infty} Rng}(A\{S^{-1}\},V).$$
Summing up, the arrow $\eta_A^S : A\{S^{-1}\} \rightarrow {\rm  vN}(A)\{S'^{-1}\}$ satisfies the universal property of vN-regular hull, thus it is isomorphic to the arrow $\eta_{A\{S^{-1}\}} : A\{S^{-1}\} \rightarrow {\rm  vN}(A\{S^{-1}\})$.
\end{proof}

The following result shows us how useful is  the sheaf-theoretic description of the vN-hull of a $\mathcal{C}^{\infty}$-ring.\\

\begin{proposition}\label{prp29MA} ${\rm  vN}$ preserves {\em finite} products.
More precisely, let $I$ be a finite set and $\{ A_i : i \in  I\}$ any family of $\mathcal{C}^{\infty}$-rings. Denote $\pi_j : \prod_{i \in I} A_i \twoheadrightarrow A_j$ the projection homomorphism, $j \in I$. Then $\prod_{i \in I} \eta_{A_i} :  \prod_{i \in I} A_i \rightarrow \prod_{i \in I} {\rm  vN}(A_i)$ satisfies the universal property of vN-regular hull, thus it is isomorphic to the arrow $\eta_{\prod_{i \in I} A_i} : \prod_{i \in I} A_i \rightarrow {\rm  vN}(\prod_{i \in I} A_i)$, through the obvious  pair of inverse (iso)morphisms ${\rm  vN}(\prod_{i \in I} A_i)  \rightleftarrows \prod_{i \in I} {\rm  vN}(A_i)$.
%but maybe $vN$ do not preserves products in general
\end{proposition}
\begin{proof}
(Sketch) %gluing of 2 sheafs of rings over dijoint boolean spaces with fields as stalks.
First of all, note that $\prod_{i \in I} {\rm  vN}(A_i)$ is a vN-regular $\mathcal{C}^{\infty}$-ring. If $I = \emptyset$, then $\prod_{i \in I} A_i$, $\prod_{i \in I} {\rm  vN}(A_i)$ and ${\rm  vN}(\prod_{i \in I} A_i)$ are isomorphic to the trivial $\mathcal{C}^{\infty}$-ring $\{0\}$, thus the result holds in this case.\\

By induction, we only need to see that $({\rm  vN}(\pi_1), {\rm  vN}(\pi_2)) : {\rm  vN}(A_1 \times A_2) \stackrel{\cong}{\rightarrow} {\rm  vN}(A_1) \times {\rm  vN}(A_2)$. But we have:\\

$\bullet$ \ $[\pi_1^\star, \pi_2^\star] : {\rm  Spec}^{\infty}(A_1) \sqcup {\rm  Spec^{\infty}}(A_2) \stackrel{\approx}{\rightarrow} {\rm  Spec}^{\infty}(A_1 \times A_2)$, \ $(p_i, i) \mapsto \pi_i^{-1}[p_i]$, $i =1,2$. \\

$\bullet$ \  $[\pi_1^\star, \pi_2^\star] : {\rm Spec}^{\infty-{\rm const}}(A_1) \sqcup {\rm Spec}^{\infty-{\rm const}}(A_2) \stackrel{\approx}{\rightarrow} {\rm Spec}^{\infty-{\rm const}}(A_1 \times A_2)$.\\

$\bullet$ \ Let $a:=(a_1,a_2) \in A_1 \times A_2$ and $\bar{b} := \{(b_{1,1},b_{2,1}), \cdots, (b_{1,n},b_{2,n})\} \subseteq_{{\rm fin}}  A_1 \times A_2$, then $(A_1 \times A_2)_{a, \bar{b}} \ \cong \ {A_1}_{a_1,\bar{b}_1} \times {A_2}_{a_2,\bar{b}_2}$.\\

$\bullet$ \ $\sqcup_{p \in { \rm  Spe\mathcal{C}^{\infty}}(A_1 \times A_2)} k_{\mathfrak{p}}(A_1 \times A_2) \approx  (\sqcup_{\mathfrak{p}_1 \in {\rm Spec}^{\infty}(A_1)} k_{\mathfrak{p}_1}(A_1)) \sqcup  (\sqcup_{\mathfrak{p}_2 \in {\rm Spec}^{\infty}(A_1)} k_{\mathfrak{p}_2}(A_2))$.\\

$\bullet$ \ Let $U \in {\rm Open}({\rm Spec}^{\infty}(A_1 \times A_2))$, then $U \approx U_1 \sqcup U_2$, where $U_i = (\pi_i^\star)^{-1}[U] \in {\rm Open}({\rm Spec}^{\infty}(A_i))$.  Then any sheaf $S$ satisfies $S(A)(U) \ \cong\ S(A_1)(U_1) \times S(A_2)(U_2)$.\vspace{2mm} \\

\noindent Now the construction of ${\rm  vN}$ as the global sections of of a sheaf allows to conclude  ${\rm  vN}(A_1 \times A_2) = S(A_1 \times A_2)({\rm Spec}^{\infty}(A_1 \times A_2)) \ \cong\ (S(A_1)({\rm Spec}^{\infty}(A_1)) \times S(A_2)({\rm Spec}^{\infty}(A_2))) = {\rm  vN}(A_1) \times {\rm  vN}(A_2)$.
\end{proof}

\section{Von Neumann-regular $\mathcal{C}^{\infty}-$Rings and Boolean Algebras}

In this section we also apply von Neumann regular $\mathcal{C}^{\infty}$-ring to naturally represent Boolean Algebras in a strong sense: i.e., not only all Boolean algebras are isomorphic to the Boolean algebra of idempotents of a von Neumann regular $\mathcal{C}^{\infty}$-ring, as every homomorphism between such boolean algebras of idempotents is induced by a $\mathcal{C}^{\infty}$-homomorphism.

We start with the following general results:

\begin{proposition}\label{tio}Let $(L, \wedge, \vee, \leq)$ be any lattice. Then:
$$(\forall a \in L)(\forall b \in L)(\forall c \in L)(a \wedge (b \vee c) = (a \wedge b) \vee (a \wedge c))$$
if, and only if,
$$(\forall a \in L)(\forall b \in L)(\forall c \in L)(a \vee (b \wedge c) = (a \vee b) \wedge (a \vee c))$$
\end{proposition}
\begin{proof}
See, for example, \textbf{Lemma 6.3} of \cite{Davey}.
\end{proof}

\begin{proposition}\label{BA-fa}Let $(A, +, \cdot, 0, 1)$ be any commutative unital ring and define $B(A):= \{ e \in A | e^2=e\}$. Then $(B(A), \wedge, \vee, *, \leq, \bot, \top)$, where:
$$\begin{array}{cccc}
    \wedge : & B(A)\times B(A) & \rightarrow & B(A) \\
     & (e,e') & \mapsto & e \cdot e'
  \end{array}$$

$$\begin{array}{cccc}
    \vee : & B(A)\times B(A) & \rightarrow & B(A) \\
     & (e,e') & \mapsto & e + e' - e\cdot e'
  \end{array}$$

$$\begin{array}{cccc}
    *: & B(A) & \rightarrow & B(A) \\
     & e & \mapsto & 1-e
  \end{array}$$

$$\leq = \{ (e,e') \in B(A)\times B(A) | e\cdot e' = e\}$$

$$\top = 1 \,\,\, \mbox{and}\,\,\, \bot = 0$$

is a bounded, distributive and complemented lattice, \textit{i.e.}, a Boolean Algebra.
\end{proposition}
\begin{proof}

\textbf{Claim:} $(e \wedge e')^2 = e \wedge e'$, so $e \wedge e' \in B(A)$.\\

$(e \wedge e')^2 = (e \cdot e')^2 = e\cdot e' \cdot e \cdot e' = e^2 \cdot {e'}^2 = e \cdot e' = e \wedge e'$.\\

\textbf{Claim:} $(e \vee e')^2 = e \vee e'$, so $e \vee e' \in B(A)$.\\

\begin{multline*}(e \vee e')^2 = (e + e' - e\cdot e')^2 = e^2 + 2 \cdot e \cdot e' + {e'}^2 - 2(e^2 \cdot e' + e \cdot {e'}^2) + {e \cdot e'}^2 =\\
= e^2 + {e'}^2 + 2 \cdot e \cdot e' - 2 \cdot e \cdot e' - 2 \cdot e^2 \cdot e' - 2 \cdot e \cdot {e'}^2 + e \cdot e' = e^2 + {e'}^2 - e \cdot e' = e + e' - e \cdot e' = e \vee e'
\end{multline*}

\textbf{Claim}: For any $e \in B(A)$, $e^{*} = 1 -e \in B(A)$.\\

$$(1-e)^2 = 1 - 2\cdot e + e^2 = 1 -e.$$

Note that for any $e \in B(A)$, $e \vee e^{*} = e \vee (1-e) = e + (1-e) - e \cdot (1-e) = 1$.\\

\textbf{Claim:} Given any $e,e' \in B(A)$, $e \vee e' = \sup \{ e,e'\}$.\\

We have:

$$e \cdot (e \vee e') = e \cdot (e + e' - e \cdot e') = e^2 + e \cdot e' - e^2\cdot e' = e + e \cdot e' - e \cdot e' = e, \,\, \mbox{so}\,\, e \leq (e \vee e').$$

A similar reasoning shows us that $e' \leq (e \vee e')$.\\

Let $f \in B(A)$ be such that $e \leq f$ and $e' \leq f$. Then:
$$e \cdot f = e \,\,\, \mbox{and}\,\,\, e' \cdot f = e',$$
so
$$f \cdot (e \vee e') = f \cdot (e + e' - e \cdot e') = f \cdot e + f \cdot e' - f \cdot e \cdot e' = e + e' - e \cdot e' = e \vee e',$$

hence $e \vee e' = \sup \{ e,e'\}$.

\textbf{Claim:} Given any $e,e' \in B(A)$, $e \wedge e' = \inf \{ e,e'\}$.\\

We have:

$$ e \cdot e'= e^2 \cdot e' = e \cdot (e \cdot e')= e \cdot (e \wedge e')  , \,\,\, \mbox{so}\,\, e \wedge e' \leq e.$$

A similar reasoning shows us that $e \wedge e' \leq e'$.\\

Let $f \in B(A)$ be such that $f \leq e$ and $f \leq e'$, so $f \cdot e = f$ and $f \cdot e' = f$.\\

We have:

$$f \cdot (e \wedge e') = f \cdot (e \cdot e') = (f \cdot e) \cdot e' = f \cdot e' = f, \,\,\, \mbox{so}\,\, f \leq e \wedge e',$$

hence $e \wedge e' = \inf \{ e,e'\}$.\\

Note that since $(\forall e \in B(A))(e \cdot 0 = 0)$, then $(\forall e \in B(A))(0 \leq e)$ so $0 = \bot$, and since $(\forall e \in B(A))(e \cdot 1 = e)$, then $(\forall e \in B(A))(e \leq 1)$, so $1 = \top$.\\

\textbf{Claim:} $(\forall e \in B(A))(\forall f \in B(A))(\forall g \in B(A))(e \wedge(f \vee g) = (e \wedge f)\vee (e \wedge g)).$

\begin{multline*}
e \wedge (f \vee g) = e \cdot (f \vee g) = e \cdot (f + g - f \cdot g) = e \cdot f + e \cdot g - e \cdot f \cdot g = \\
= (e \cdot f)\vee (e \cdot g) = e \cdot f + e \cdot g - (e \cdot f)\cdot (e \cdot g) = e \cdot f + e \cdot g - e \cdot f \cdot g = e \cdot (f + g - f \cdot g) = e \wedge (f \vee g)
\end{multline*}

By \textbf{Proposition \ref{tio}}, it follows that:

$$(\forall e \in B(A))(\forall f \in B(A))(\forall g \in B(A))(e \vee (f \wedge g) = (e \vee f)\wedge (e \vee g)).$$

It follows from these considerations that $(B(A), \wedge, \vee, *, \leq, \bot, \top)$ is a bounded, distributive and complemented lattice, hence a Boolean algebra.
\end{proof}

\begin{remark}\label{stone-dual}By Stone Duality, there is an anti-equivalence of categories between the category of Boolean algebras, ${\rm \bf Bool}$, and the category of Boolean spaces, ${\rm \bf BoolSp}$.

Under this anti-equivalence, a Boolean space $(X,\tau)$ is mapped to the Boolean algebra of clopen subsets of $(X,\tau)$, ${\rm Clopen}\,(X)$:

$$\begin{array}{cccc}
    {\rm Clopen}: & {\rm \bf BoolSp} & \rightarrow & {\rm \bf Bool} \\
     & \xymatrix{(X,\tau) \ar[r]^{f}& (Y,\sigma)} & \mapsto & \xymatrix{{\rm Clopen}\,(Y) \ar[r]^{f^{-1}\upharpoonright}& {\rm Clopen}\,(X)}
  \end{array}$$
	
The quasi-inverse functor is given by the Stone space functor: a Boolean algebra $B$ is mapped to the Stone space of $B$, ${\rm Stone}(B) = (\{ U \subseteq B: U \,\, \text{is an ultrafilter in} \,\, B\}, \tau_B)$ , where $\tau_B$ is the topology whose basis is given by the image of the map $t_B : B \to \wp({\rm Stone}(B))$, $b \mapsto t_B(b) = S_B(b) = \{ U \in {\rm Stone}(B) : b \in U\}$.

$$\begin{array}{cccc}
    {\rm Stone}: & {\rm \bf Bool} & \rightarrow & {\rm \bf BoolSp} \\
     & \xymatrix{B \ar[r]^{h}& B'} & \mapsto & \xymatrix{{{\rm Stone}(B')} \ar[r]^{h^{-1}\upharpoonright}& {{\rm Stone}(B)}}
  \end{array}$$

\end{remark}

\begin{remark}

Let $(A', +', \cdot', 0', 1')$ be any commutative unital ring and denote by $\wedge', \vee', *', \leq', 0'$ and $1'$ its respective associated Boolean algebra operations, relations and constant symbols as constructed above. Note that for any commutative unital ring homomorphism $f: A \to A'$, the map $B(f):= f\upharpoonright_{B(A)}: B(A) \to A'$ is such that:

\begin{itemize}
  \item[(i)]{$B(f)[B(A)] \subseteq B(A')$;}
  \item[(ii)]{$(\forall e_1 \in A)(\forall e_2 \in A)(B(f)(e_1 \wedge e_2) = f\upharpoonright_{B(A)}(e_1 \cdot e_2) = (f\upharpoonright_{B(A)}(e_1))\cdot' (f\upharpoonright_{B(A)}(e_2))= B(f)(e_1)\wedge'B(f)(e_2))$}
  \item[(iii)]{$(\forall e_1 \in A)(\forall e_2 \in A)(B(f)(e_1 \vee e_2) = f\upharpoonright_{B(A)}(e_1 + e_2) = (f\upharpoonright_{B(A)}(e_1))+' (f\upharpoonright_{B(A)}(e_2))= B(f)(e_1)\vee'B(f)(e_2))$}
  \item[(iv)]{$(\forall e \in B(A))(B(f)(e^{*}) = f\upharpoonright_{B(A)}(1-e)=f\upharpoonright_{B(A)}(1)-f\upharpoonright_{B(A)}(e) = 1' - f\upharpoonright_{B(A)}(e) = B(f)(e)^{*})$}
  \end{itemize}

hence a morphism of Boolean algebras.\\

We also have, for every ring $A$, $B({\rm id}_A) = {\rm id}_{B(A)}$ and given any $f: A \to A'$ and $f': A' \to A''$, $B(f' \circ f) = B(f')\circ B(f)$, since $B(f) = f\upharpoonright_{B(A)}$, so:

$$\begin{array}{cccc}
    B: & {\rm \bf CRing} & \rightarrow & {\rm \bf Bool} \\
     & A & \mapsto & B(A)\\
     & {\xymatrix{A \ar[r]^{f}& A'}} & \mapsto & \xymatrix{B(A) \ar[r]^{B(f)} & B(A')}
  \end{array}$$

is a (covariant) functor.

\end{remark}

Since we can regard any $\mathcal{C}^{\infty}-$ring $A$ as a commutative unital ring via the forgetful functor $\widetilde{U}: \mathcal{C}^{\infty}{\rm \bf Rng} \to {\rm \bf CRing}$, we have a (covariant) functor:

$$\begin{array}{cccc}
    \widetilde{B}: & \mathcal{C}^{\infty}{\rm \bf Rng} & \rightarrow & {\rm \bf Bool} \\
     & A & \mapsto & \widetilde{B}(A):= (B \circ U)(A)\\
     & {\xymatrix{A \ar[r]^{f}& A'}} & \mapsto & \xymatrix{\widetilde{B}(A) \ar[r]^{\widetilde{B}(f)} & \widetilde{B}(A')}
  \end{array}$$

Now, if $A$ is any $\mathcal{C}^{\infty}-$ring, we can define the following map:

$$\begin{array}{cccc}
    \jmath_A: & \widetilde{B}(A) & \rightarrow & {\rm Clopen}\,({\rm Spec}^{\infty}(A)) \\
     & e & \mapsto & D^{\infty}(e) = \{ \mathfrak{p} \in {\rm Spec}^{\infty}\,(A)| e \notin \mathfrak{p}\}
  \end{array}$$

\textbf{Claim:} The map defined above is a Boolean algebra homomorphism.\\

Note that for any $e \in B(A)$, $D^{\infty}(e^*) = D^{\infty}(1-e) = {\rm Spec}^{\infty}\,(A)\setminus D^{\infty}(e) = D^{\infty}(e)^{*}$, since:
$$D^{\infty}(e) \cap D^{\infty}(1-e) = D^{\infty}\,(e \cdot (1-e)) = D^{\infty}(0) = \varnothing$$
and
\begin{multline*}D^{\infty}(e)\cup D^{\infty}(e*) = D^{\infty}(e)\cup D^{\infty}(1-e) = D^{\infty}(e^2 + (1-e)^2) = D^{\infty}(e + (1-e))=\\
= D^{\infty}(1) = {\rm Spec}^{\infty}\,(A).\end{multline*}

Hence $\jmath_A(e^{*}) = \jmath_A(1-e) = D^{\infty}(1-e) = {\rm Spec}^{\infty}\,(A) \setminus D^{\infty}\,(e) = {D^{\infty}(e)}^{*} = \jmath_A(e)^{*}$.\\

By the item (iii) of \textbf{Lemma 1.4} of \cite{rings2}, $D^{\infty}(e \cdot e') = D^{\infty}(e) \cap D^{\infty}(e')$, so:

$$\jmath_A(e \wedge e') = D^{\infty}(e\cdot e') = D^{\infty}(e) \cap D^{\infty}(e') = \jmath_A(e)\cap \jmath_A(e').$$

Last,

\begin{multline*}
\jmath_A(e \vee e') = \jmath_A(e + e' - e \cdot e') = D^{\infty}(e + e' - e \cdot e') = D^{\infty}(e^2)\cup D^{\infty}(e' - e \cdot e')= \\
D^{\infty}(e)\cup D^{\infty}(e'\cdot(1 - e)) = D^{\infty}(e)\cup [D^{\infty}(e') \cap D^{\infty}(1-e)] = \\
= [D^{\infty}(e) \cup D^{\infty}(e')] \cap [D^{\infty}(e) \cup D^{\infty}(1-e)] = [D^{\infty}(e) \cup D^{\infty}(e')]\cap {\rm Spec}^{\infty}\,(A) =\\
= D^{\infty}(e) \cup D^{\infty}(e') = \jmath_A(e) \cup \jmath_A(e'),
\end{multline*}

and the claim is proved.\\

The map $\jmath_A : B(A) \to {\rm Spec}^{\infty}\,(A)$ suggests that the idempotent elements of the  Boolean algebra $B(A)$ associated with a $\mathcal{C}^{\infty}-$ring $A$ hold a strong relationship with the canonical basis of the Zariski topology of ${\rm Spec}^{\infty}(A)$.  In this section we are going to show that these idempotent elements, in the case of the von Neumann-regular $\mathcal{C}^{\infty}-$rings, represent \textit{all the Boolean algebras}.\\

\begin{theorem}\label{jordao}Let $A$ be a von Neumann regular  $\mathcal{C}^{\infty}-$ring. The map:
$$\begin{array}{cccc}
    \jmath_A: & \widetilde{B}(A) & \rightarrow & {\rm Clopen}\,({\rm Spec}^{\infty}(A)) \\
     & e & \mapsto & D^{\infty}(e) = \{ \mathfrak{p} \in {\rm Spec}^{\infty}\,(A)| e \notin \mathfrak{p}\}
  \end{array}$$
is an isomorphism of Boolean algebras.
\end{theorem}
\begin{proof}

It has already been proved that $\jmath_A: B(A) \to {\rm Spec}^{\infty}\,(A)$ is a homomorphism of Boolean algebras, so in order to prove that $\jmath_A : \widetilde{B}(A) \to {\rm Clopen}\,({\rm Spec}^{\infty}\,(A))$ is an injective map, it suffices to show that:
$$\jmath_A^{\dashv}[{\rm Spec}^{\infty}\,(A)] = \{ 1\}.$$

(see \textbf{Lemma 4.8} of \cite{BellSlomson})
\textbf{Claim:} $(\jmath_A(e) = {\rm Spec}^{\infty}\,(A)) \iff (e \in A^{\times})$.\\

Indeed,

$$a \in A^{\times} \Rightarrow D^{\infty}(a) = {\rm Spec}^{\infty}\,(A).$$

On the other hand, if $a \notin A^{\times}$ then there is some maximal $\mathcal{C}^{\infty}-$radical prime ideal $\m \in {\rm Spec}^{\infty}\,(A)$ such that $a \in \m$. If $\m \notin D^{\infty}(a)$ then $\jmath_A(a) = D^{\infty}(a) \neq {\rm Spec}^{\infty}\,(A)$. Hence $\jmath_A(a) = {\rm Spec}^{\infty}\,(A) \Rightarrow a \in A^{\times}$.\\

Let $e \in \widetilde{B}(A)$ be such that $\jmath_A(e) = D^{\infty}(e) = {\rm Spec}^{\infty}\,(A)$, so $(\forall \mathfrak{p} \in {\rm Spec}^{\infty}\,(A))(e \notin \mathfrak{p})$.

Since $0 \in \mathfrak{p}$ for every $\mathfrak{p} \in {\rm Spec}^{\infty}\,(A)$, then $(e) \neq (0)$, so $(e) = A = (1)$, and $e \in A^{\times}$ and since $e$ is idempotent, $e = 1$.\\

We already know that $\jmath_A$ is an injective Boolean algebras homomorphism, so it suffices to prove that it is also surjective.\\

\textbf{Claim:} Given any $a_1, a_2, \cdots, a_n \in A$, there is some $b \in A$ such that:
$$D^{\infty}(a_1)\cup D^{\infty}(a_2)\cup \cdots \cup D^{\infty}(a_n) = D^{\infty}(b).$$

We are going to prove this claim by induction.\\

\textbf{Case 1:} $n=2$.\\

By the item (iii) of \textbf{Lemma 1.4} of \cite{rings2}, we know that for any $a_1,a_2 \in A$,
$$D^{\infty}(a_1)\cup D^{\infty}(a_2) = D^{\infty}(a_1^2+a_2^2),$$
so we can take $b = a_1^2 + a_2^2$.\\

\textbf{Inductive step:} Suppose that given $a_1, \cdots, a_{n-1} \in A$ there is some $b_{n-1} \in A$ such that:

$$D^{\infty}(a_1)\cup \cdots \cup D^{\infty}(a_{n-1}) = D^{\infty}(b_{n-1}).$$

Given $a_1, \cdots, a_n$, we know by the inductive step, that there is some $b_{n-1} \in A$ such that:

$$D^{\infty}(a_1)\cup \cdots \cup D^{\infty}(a_{n-1}) = D^{\infty}(b_{n-1}).$$

Considering $b = {b_{n-1}}^2 + a_n^2$, we have:

$$D^{\infty}(a_1)\cup \cdots \cup D^{\infty}(a_{n-1})\cup D^{\infty}(a_n) = D^{\infty}(b_{n-1})\cup D^{\infty}(a_n) = D^{\infty}(b),$$

which proves the claim.\\

Let $E \in {\rm Clopen}\,({\rm Spec}^{\infty}(A))$. Since $\{ D^{\infty}(a) | a \in A\}$ is a basis for the topology of the compact topological space ${\rm Spec}^{\infty}(A)$, there are $a_1,a_2, \cdots, a_n \in A$ such that:
$$E = D^{\infty}(a_1)\cup D^{\infty}(a_2) \cup \cdots \cup D^{\infty}(a_n).$$

By the above claim, there is some $b \in A$ such that $E = D^{\infty}(a_1)\cup D^{\infty}(a_2) \cup \cdots \cup D^{\infty}(a_n) = D^{\infty}(b)$.\\

Now, since $A$ is a von Neumann regular $\mathcal{C}^{\infty}-$ring, there is some idempotent element $e \in A$ such that $(b)= (e)$.\\

We claim that $\jmath_A(e) = E$.\\

We have $\jmath_A(e) := D^{\infty}(e) = D^{\infty}(b) = D^{\infty}(a_1)\cup \cdots \cup D^{\infty}(a_n) = E$, and it follows that $\jmath_A: \widetilde{B}(A) \stackrel{\cong}{\rightarrow} {\rm Clopen}\,({\rm Spec}^{\infty}\,(A))$ is an isomorphism.
\end{proof}

\begin{theorem}\label{uhum}We have the following diagram of categories, functors and natural isomorphisms:
$$\xymatrixcolsep{3pc}\xymatrix{
\mathcal{C}^{\infty}{\rm \bf vNRng} \ar@/_/[ddrr]_{\widetilde{B}} \ar[rr]^{{\rm Spec}^{\infty}} & & {\rm \bf BoolSp} \ar[dd]^{{\rm Clopen}}\\
 & \ar@2[ur]_{\jmath} & \\
 & & {\rm \bf Bool}
}$$
\end{theorem}
\begin{proof}
First note that since $A$ is a von Neumann-regular $\mathcal{C}^{\infty}-$ring, the set of the compact open subsets of (the Boolean space) ${\rm Spec}^{\infty}\,(A)$ equals ${\rm Clopen}\,({\rm Spec}^{\infty}\,(A))$.\\

On the one hand, given a von Neumann regular $\mathcal{C}^{\infty}-$ring $A$, we have:

$$({\rm Clopen}\,({\rm Spec}^{\infty}))(A) = {\rm Clopen}\,({\rm Spec}^{\infty}\,(A)) = \{ D^{\infty}(e) | e \in \widetilde{B}(A)\}$$

For every von Neumann regular $\mathcal{C}^{\infty}-$ring $A$, by \textbf{Theorem \ref{jordao}}, we have the following isomorphism of Boolean algebras:

$$\begin{array}{cccc}
    \jmath_A: & \widetilde{B}(A) & \rightarrow & {\rm Clopen}\,({\rm Spec}^{\infty}(A)) \\
     & e & \mapsto & D^{\infty}(e)
  \end{array}$$

It is easy to see that for every $\mathcal{C}^{\infty}-$homomorphism $f: A \to A'$, we have the following commutative rectangle:

$$\xymatrixcolsep{3pc}\xymatrix{
\widetilde{B}(A) \ar[d]_{\widetilde{B}(f)} \ar[r]^{\jmath_A} & {\rm Clopen}\,({\rm Spec}^{\infty}\,(A)) \ar[d]^{{\rm Clopen}\,({\rm Spec}^{\infty}(f))}\\
\widetilde{B}(A') \ar[r]_{\jmath_{A'}} & {\rm Clopen}\,({\rm Spec}^{\infty}\,(A'))}$$

In fact, given $e\in \widetilde{B}(A)$, we have, on the one hand, $\jmath_A(e)=D_{A}^{\infty}(e)$ and ${\rm Clopen}({\rm Spec}^{\infty}(f))(D_A^{\infty}(e)) = D_{A'}^{\infty}(f(e))$. On the other hand, $\jmath_{A'}\circ \widetilde{B}(f)(e) = \jmath_{A'}(D_{A'}^{\infty}(f(e)))$, so the rectangle commutes.\\

Thus $\jmath$ is a natural transformation. Thus $\xymatrix{\jmath: \widetilde{B} \ar@2[r] & {\rm Clopen}\circ {\rm Spec}^{\infty}}$ is a natural isomorphism and the diagram ``commutes'' (up to natural isomorphism).
\end{proof}

\begin{lemma}\label{niedia}Let $(X,\tau)$ be a Boolean topological space, and let:
$$\mathcal{R} = \left\{ R \subseteq X \times X | (R \,\, {\rm is \,\, an \,\, equivalence \,\, relation \,\, on} \, X) \& (\dfrac{X}{R} \, {\rm is \,\, a \,\, discrete\,\, compact\,\, space})\right\},$$
which is partially ordered by inclusion. Whenever $R_i, R_j \in \mathcal{R}$ are such that $R_j \subseteq R_i$, we have the continuous surjective map:
$$\begin{array}{cccc}
    \mu_{R_jR_i}: & \dfrac{X}{R_j} & \twoheadrightarrow & \dfrac{X}{R_i} \\
     & [x]_{R_j} & \mapsto & [x]_{R_j}
  \end{array}$$
so we have the inverse system $\{ \frac{X}{R_i} ; \mu_{R_jR_i}: \frac{X}{R_j} \to \frac{X}{R_i} \}$. By definition (see, for instance, \cite{Dugundji}),

$$
\varprojlim_{R \in \mathcal{R}} \dfrac{X}{R} = \left\{ ([x]_{R_i})_{R_i \in \mathcal{R}} \in \prod_{R \in \mathcal{R}} \frac{X}{R} |(R_j \subseteq R_i \to (p_{R_i}(([x]_{R_i})_{R_i \in \mathcal{R}}) = \mu_{R_jR_i}\circ p_{R_j}(([x]_{R_i})_{R_i \in \mathcal{R}})))\right\}
$$

Let $X_{\infty}$ denote $\varprojlim_{R \in \mathcal{R}} \frac{X}{R}$, so we have the following cone:
$$\xymatrix{
 & X_{\infty} \ar[dl]_{\mu_{R_j}} \ar[dr]^{\mu_{R_i}} & \\
\frac{X}{R_j} \ar[rr]^{\mu_{R_jR_i}} & & \frac{X}{R_i}
}$$

We consider $X_{\infty}$ together with the induced subspace topology of $\prod_{R \in \mathcal{R}} \frac{X}{R}$.\\

By the universal property of $X_{\infty}$, there is a unique continuous map map $\delta: X \to X_{\infty}$ such that the following diagram commutes:

$$\xymatrixcolsep{3pc}\xymatrix{
 & X \ar@/_2pc/[ddl]_{q_{R_j}} \ar@/^2pc/[ddr]^{q_{R_i}} \ar@{-->}[d]^{\exists ! \delta}& \\
  & X_{\infty} \ar[dl]_{\mu_{R_j}} \ar[dr]^{\mu_{R_i}} & \\
  \dfrac{X}{R_j} \ar[rr]_{\mu_{R_jR_i}} & & \dfrac{X}{R_i}
}$$

We claim that such a $\delta: X \to X_{\infty}$ is a homeomorphism, so:

$$X \cong \varprojlim_{R \in \mathcal{R}} \frac{X}{R}$$
that is, $X$ a profinite space.
\end{lemma}
\begin{proof}
First we observe that $\mathcal{R}$ is a downwards directed set.\\

\textbf{Claim:} Given any $R_1,R_2 \in \mathcal{R}$, $R_1 \cap R_2 \in \mathcal{R}$.\\

Naturally, if both $R_1$ and $R_2$ are equivalence relations on $X$, then $R_1 \cap R_2$ is also an equivalence relation on $X$, and it is such that $R_1 \cap R_2 \subseteq R_1$ and $R_1 \cap R_2 \subseteq R_2$. Moreover, we have an injective continuous map:

$$\begin{array}{cccc}
    \varphi: & \dfrac{X}{R_1 \cap R_2} & \rightarrowtail & \dfrac{X}{R_1} \times \dfrac{X}{R_2} \\
     & [x]_{R_1 \cap R_2} & \mapsto & ([x]_{R_1}, [x]_{R_2})
  \end{array}$$

\textbf{Claim:} Given $R \in \mathcal{R}$, if $S$ is any equivalence relation on $X$ such that $R \subseteq S$, then $S \in \mathcal{R}$.\\

We know that each $\frac{X}{R}$ is a compact, Hausdorff and totally disconnected space, hence $\prod_{R \in \mathcal{R}} \frac{X}{R}$ is also a compact, Hausdorff and totally disconnected space, \textit{i.e.}, a Boolean space.\\

By \textbf{Theorem 2.4} of \cite{Dugundji}, since each space $\frac{X}{R_i}$ is a Hausdorff and compact space, $X_{\infty}$ is also a Hausdorff and compact space, so it is closed in the Boolean space $\prod_{R \in \mathcal{R}} \frac{X}{R}$, hence it is a Boolean space.\\

Consider:

$$\begin{array}{cccc}
    \delta_X : & X & \rightarrow & X_{\infty} \\
     & x & \mapsto & ([x]_{R_i})_{R_i \in \mathcal{R}}
  \end{array}$$

\textbf{Claim 1:} $\delta_X$ is injective;\\

Let $x,x' \in X$, $x \neq x'$. Since $(X, \tau)$ is a Boolean space, there exists some clopen $U \in \tau$ such that $x \in U$ and $x' \notin U$, that is, $x' \notin U$.\\

Consider the following equivalence relation on $X$:

$$R = (U \times U) \cup (X \setminus U \times X \setminus U)$$

and note that $\frac{X}{R}$ is a discrete topological space, hence $R \in \mathcal{R}$. We have $(x,x') \notin R$, so for every $S \in \mathcal{R}$ such that $R \supseteq S$ we have $[x]_S \neq [x']_S$, so $\delta_X(x) \neq \delta_X(x')$.\\

\textbf{Claim 2:} ${\rm Cl}. (\delta_X[X]) = X_{\infty}$.\\

Let $\Omega = \left(\{ [a_1]_{R_1} \} \times \{ [a_2]_{R_2} \} \times \cdots \times \{ [a_n]_{R_n}\} \times \prod_{R \neq R_1, \cdots, R_n} \frac{X}{R}\right)\cap X_{\infty} \subset X_{\infty}$

be a basic open set of $X_{\infty}$. Let:
$$R' = R_1 \cap R_2 \cap \cdots \cap R_n \in \mathcal{R}.$$

Since $R_i \subseteq R_j$ implies that:

$$\xymatrix{
 & X \ar[dl]_{q_{R_i}} \ar[dr]^{q_{R_j}}\\
\dfrac{X}{R_i} \ar[rr]^{\mu_{R_iR_j}} & & \dfrac{X}{R_j}
}$$

commutes, the map:

$$\xymatrix{ \dfrac{X}{R_1 \cap R_2 \cap \cdots \cap R_n} \ar@{>->}[r] & \dfrac{X}{R_1}\times \dfrac{X}{R_2} \times \cdots \times \dfrac{X}{R_n}}$$

is injective.\\

Since $\mathcal{R}$ is a directed set, we can assume without lost of generality, that $\{ R_1, R_2, \cdots,R_n \}$ has a minimum, that we shall denote by $R' = \min \{ R_1, \cdots, R_n \}$.\\

Now we can chose any $[a']_{R'} \in \frac{X}{R'}$. We claim that $\delta_X[\{a'\}] \cap \Omega \neq \varnothing$.\\

We have:
$$\delta_X(a') = ([a']_R)_{R \in \mathcal{R}}$$
so if $R = R_i$, $i \in \{1,2, \cdots, n \}$, then $[a']_{R_i} = \mu_{R'R_i}([a']_{R'}) = [a_i]_{R_i}$, hence $\delta_X(a') \in \Omega$. Thus:

$$\delta_X(a') \in ([a']_R)_{R \in \mathcal{R}}\delta_X[\{a'\}] \cap \Omega \neq \varnothing.$$

\textbf{Claim 3:} $\delta_X: X \to X_{\infty}$ is surjective.\\

Since $(X, \tau)$ is a compact space, $\delta_X: X \to X_{\infty}$ is continuous and $X_{\infty}$ is a Hausdorff space, it follows that $\delta_X$ is a closed map. Hence,

$$\delta_X[X] = {\rm Cl}. (\delta_X[X]) = X_{\infty}.$$

Thus $\delta_X: X \to X_{\infty}$ is a continuous closed bijection, hence a homeomorphism.

\end{proof}

Let ${\rm \bf BoolSp}$ be the category whose objects are all the Boolean spaces and whose morphisms are all the Boolean spaces homomorphisms. Given any Boolean space $(X,\tau)$, let
\begin{multline*}\mathcal{R}_X = \{ R \subseteq X \times X | R \mbox{is \,\,\, an \,\,\, equivalence \,\,\, relation\,\,\, on}\,\, X \,\, \mbox{and} \\
\frac{X}{R}\,\,\, \mbox{is \,\,\, discrete \,\,\, and \,\,\, compact}\}
\end{multline*}

We are going to describe an equivalence functor between the category of the Boolean spaces and the category of profinite spaces.\\

First we note that given any continuous function $f: X \to X'$ and any $R' \in \mathcal{R}_{X'}$,

$$R_{f}:= (f \times f)^{\dashv}[R'] \subseteq X \times X$$

is an equivalence relation on $X$ and the following diagram:

$$\xymatrixcolsep{3pc}\xymatrix{
X \ar[r]^{f} \ar[d]_{p_{R_f}} & X' \ar[d]^{p_R'}\\
\dfrac{X}{R_f} \ar[r]^{f_{R_fR'}} & \dfrac{X'}{R'}},$$

where $p_{R_f}: X \to \dfrac{X}{R_f}$ and $p_{R'}: X' \to \dfrac{X'}{R'}$ are the canonical projections, commutes.\\

We know, by \textbf{Theorem 4.3} of \cite{Dugundji} that $f_{R_fR'}: \frac{X}{R_f} \to \frac{X'}{R'}$ is a continuous map, and it is easy to see that $f_{R_fR'}$ is injective, as we are going to show.\\

Given any $[x]_{R_f},[y]_{R_f} \in \dfrac{X}{R_f}$ such that $[x]_{R_f} \neq [y]_{R_f}$, \textit{i.e.}, such that $(x,y) \notin R_f$, we have $(f(x),f(y)) \notin R'$, \textit{i.e.}, $[f(x)]_{R'} \neq [f(y)]_{R'}$. Thus, since: $$f_{R_fR'}([x]_{R_f}) = (f_{R_fR'}\circ p_{R_f})(x) = (p_{R'} \circ f)(x) = [f(x)]_{R'}$$
and
$$f_{R_fR'}([y]_{R_f}) = (f_{R_fR'}\circ p_{R_f})(y) = (p_{R'} \circ f)(y) = [f(y)]_{R'}$$
it follows that $f_{R_fR'}([x]_{R_f}) \neq f_{R_fR'}([y]_{R_f})$.\\

Since $f_{R_fR'}: \frac{X}{R_f} \to \frac{X'}{R'}$ is an injective continuous map and $\frac{X'}{R'}$ is discrete, it follows that given any $[x']_{R'} \in \frac{X'}{R'}$:
$$f_{R_fR'}^{\dashv}[\{ [x']_{R'}\}] = \begin{cases}
\varnothing, \,\, \mbox{if}\,\, [x']_{R'} \in {\rm im}\,(f_{R_fR'}),\\
\{ *\}, \,\, \mbox{otherwise}
\end{cases},$$
so every singleton of $\frac{X}{R_f}$ is an open subset of $\frac{X}{R_f}$, and $\frac{X}{R_f}$ is discrete. Also, since $X$ is compact, $\frac{X}{R_f}$ is compact, and it follows that and $R_f \in \mathcal{R}_X$.\\

Now, if $R_1',R_2' \in \mathcal{R}_{X'}$ are such that $R_1' \subseteq R_2'$, then ${R_1'}_f \subseteq {R_2'}_f$. In fact, given $(x,y) \in {R_1'}_f$, we have $(f \times f)(x,y) \in R_1'$, and since $R_1' \subseteq R_2'$, it follows that $(f \times f)(x,y) \in R_2'$, so $(x,y) \in {R_2'}_f$.\\

Let $X,Y,Z$ be Boolean spaces, $X \stackrel{f}{\rightarrow} Y$ and $Y \stackrel{g}{\rightarrow} Z$ be two Boolean spaces homomorphisms and $R \in \mathcal{R}_Z$. We have:

$$((g\circ f) \times (g \circ f))^{\dashv}[R] = (f \times f)^{\dashv}[(g \times g)^{\dashv}[R]].$$

In fact,
\begin{multline*}(x,x') \in ((g\circ f) \times (g \circ f))^{\dashv}[R] \iff (g(f(x)),g(f(x'))) \in R \iff \\
\iff (f(x),f(x')) \in (g \times g)^{\dashv}[R] \iff (x,x') \in (f\times f)^{\dashv}[(g \times g)^{\dashv}[R]].
\end{multline*}

Denoting $T:= (f \times f)^{\dashv}[(g \times g)^{\dashv}[R]]$ and $S:=(g \times g)^{\dashv}[R]$, we have the following commutative diagram:

$$\xymatrixcolsep{3pc}\xymatrix{
X \ar[r]^{f} \ar[d]_{\mu_T} & Y \ar[r]^{g} \ar[d]^{\mu_S} & Z \ar[d]^{p_R}\\
\dfrac{X}{T} \ar@/_3pc/[rr]_{(g\circ f)_{TR}}\ar[r]^{f_{TS}} & \dfrac{Y}{S} \ar[r]^{f_{SR}} & \dfrac{Z}{R}
}$$

Given a continuous map between Boolean spaces, $f: X \to X'$, we can define a map $\check{f}: \varprojlim_{R \in \mathcal{R}_X}\dfrac{X}{R} \to \varprojlim_{R' \in \mathcal{R}_{X'}} \dfrac{X'}{R'}$ in a functorial manner.\\

Let $R',S' \in \mathcal{R}_{X'}$ be any two equivalence relations such that  $R' \subseteq S'$, so given $f: X \to X'$ the following rectangle commutes:

$$\xymatrixcolsep{3pc}\xymatrix{
\dfrac{X}{S_f'} \ar[r]^{\mu_{R_f'S_f'}} \ar[d]_{f_{S_f'S'}} & \dfrac{X'}{R_f'} \ar[d]^{f_{R_f'R'}}\\
\dfrac{X'}{S'} \ar[r]^{\mu_{R'S'}} & \dfrac{X'}{R'}
}$$

and since:

$$\xymatrixcolsep{3pc}\xymatrix{
 & X_{\infty} \ar@/_/[dl]_{\mu_{S_f'}} \ar@/^/[dr]^{\mu_{R_f'}}& \\
\dfrac{X}{S_f'} \ar[rr]_{\mu_{R_f'S_f'}} & & \dfrac{X}{R_f'}
}$$

commutes, the following triangle is commutative:

$$\xymatrixcolsep{3pc}\xymatrix{
 & X_{\infty} \ar@/_/[dl]_{f_{S_f'S'}\circ \mu_{S_f'}} \ar@/^/[dr]^{f_{R_f'R'}\mu_{R_f'}}& \\
\dfrac{X'}{S'} \ar[rr]_{\mu_{R'S'}} & & \dfrac{X'}{R'}
}.$$

By the universal property of ${X'}_{\infty}$, there is a unique $\check{f}: X_{\infty} \to {X'}_{\infty}$ such that the following prism is commutative:

$$\xymatrix @!0 @R=4pc @C=6pc {
    X_{\infty} \ar[rr]^{\mu_{R_f'}} \ar[rd]^{\mu_{S_f'}} \ar@{.>}[dd]_{\exists ! \check{f}} && \frac{X}{R_f'} \ar[dd]^{f_{R_f'R'}} \\
    & \frac{X}{S_f'} \ar[ru]_{\mu_{R_f'S_f'}} \ar[dd]^(.3){f_{S_f'S'}} \\
    {{X'}_{\infty}} \ar[rr] |!{[ur];[dr]}\hole \ar[rd]^{\mu_{S'}'}^(.2){\mu_{R'}} && \frac{X'}{R'} \\
    & \frac{X}{S'} \ar[ru]_{\mu_{R'S'}} }$$

We claim that

$$\begin{array}{cccc}
    \Delta : & {\rm \bf BoolSp} & \rightarrow & {\rm\bf ProfinSp} \\
     & X & \mapsto & X_{\infty} = \varprojlim_{R \in \mathcal{R}_X} \dfrac{X}{R}\\
     & \xymatrix{X \ar[r]^{f}&X'} & \mapsto & \xymatrix{X_{\infty} \ar[r]^{\check{f}}& {X'}_{\infty}}
  \end{array}$$

is a functor.\\

Let $R,R' \in \mathcal{R}_X$ be such that $R \subseteq R'$. Since for any $R \in \mathcal{R}_X$ we have $R_{{\rm id}_X} = ({\rm id}_X \times {\rm id}_X)^{\dashv}[R] = R$, then the following diagram commutes for every $R,R' \in \mathcal{R}_X$ such that $R \subseteq R'$:

$$\xymatrixcolsep{3pc}\xymatrix{
X \ar[d]_{\mu_R} \ar[r]^{{\rm id}_X} & X \ar[d]^{\mu_{R'}}\\
\dfrac{X}{R} \ar[r]^{{{\rm id}_{X}}_{RR'}} & \dfrac{X}{R'}
},$$

and ${{\rm id}_X}_{RR'} = \imath_{\frac{X}{R}}^{\frac{X}{R'}}$, since ${{\rm id}_X}_{RR'}([x]_R) = \mu_{R'}({\rm id}_X(x)) = [x]_{R'}$.\\

We note that ${{\rm id}_{X_{\infty}}}: X_{\infty} \to X_{\infty}$ is such that the following prism commute:

$$\xymatrix @!0 @R=4pc @C=6pc {
    X_{\infty} \ar[rr]^{\mu_{R'}} \ar[rd]^{\mu_{R}} \ar[dd]_{{\rm id}_{X_{\infty}}} && \frac{X}{R'} \ar[dd]^{{{\rm id}_X}_{RR'}} \\
    & \frac{X}{R} \ar[ru]_{\mu_{RR'}} \ar[dd]^(.3){{{\rm id}_X}_{RR'}} \\
    {{X}_{\infty}} \ar[rr] |!{[ur];[dr]}\hole \ar[rd]^{\mu_{R}}_(.2){\mu_{R'}} && \frac{X}{R'} \\
    & \frac{X}{R} \ar[ru]_{\mu_{RR'}} }$$

and since $\check{{\rm id}_X}: X_{\infty} \to X_{\infty}$ is the unique map with this property, then:

$$\Delta({\rm id}_X) = \check{{\rm id}_X} = {{\rm id}_{X_{\infty}}}.$$

Let $X,Y,Z$ be Boolean spaces, $X \stackrel{f}{\rightarrow} Y$ and $Y \stackrel{g}{\rightarrow} Z$ be two Boolean spaces homomorphisms and $R \in \mathcal{R}_Z$. We have:

$$((g\circ f) \times (g \circ f))^{\dashv}[R] = (f \times f)^{\dashv}[(g \times g)^{\dashv}[R]].$$

In fact,
\begin{multline*}(x,x') \in ((g\circ f) \times (g \circ f))^{\dashv}[R] \iff (g(f(x)),g(f(x'))) \in R \iff \\
\iff (f(x),f(x')) \in (g \times g)^{\dashv}[R] \iff (x,x') \in (f\times f)^{\dashv}[(g \times g)^{\dashv}[R]].
\end{multline*}

Denoting $T:= (f \times f)^{\dashv}[(g \times g)^{\dashv}[R]]$ and $S:=(g \times g)^{\dashv}[R]$, we have the following commutative diagram:

$$\xymatrixcolsep{3pc}\xymatrix{
X \ar[r]^{f} \ar[d]_{p_T} & Y \ar[r]^{g} \ar[d]^{\mu_S} & Z \ar[d]^{\mu_R}\\
\dfrac{X}{T} \ar@/_3pc/[rr]_{(g\circ f)_{TR}}\ar[r]^{f_{TS}} & \dfrac{Y}{S} \ar[r]^{g_{SR}} & \dfrac{Z}{R}
},$$

that is, $(g \circ f)_{TR} = g_{SR} \circ f_{TS}$.\\

Let $R' \in \mathcal{R}_Z$ be such that $R \subseteq R'$. The same argument used above proves that, denoting $T':= (f \times f)^{\dashv}[(g \times g)^{\dashv}[R']]$ and $S':=(g \times g)^{\dashv}[R']$, we have the following commutative diagram:

$$\xymatrixcolsep{3pc}\xymatrix{
X \ar[r]^{f} \ar[d]_{\mu_{T'}} & Y \ar[r]^{g} \ar[d]^{\mu_{S'}} & Z \ar[d]^{\mu_{R'}}\\
\dfrac{X}{T'} \ar@/_3pc/[rr]_{(g\circ f)_{T'R'}}\ar[r]^{f_{T'S'}} & \dfrac{Y}{S'} \ar[r]^{g_{S'R'}} & \dfrac{Z}{R'}
}.$$

Under these circumstances, the following diagram commutes:

$$
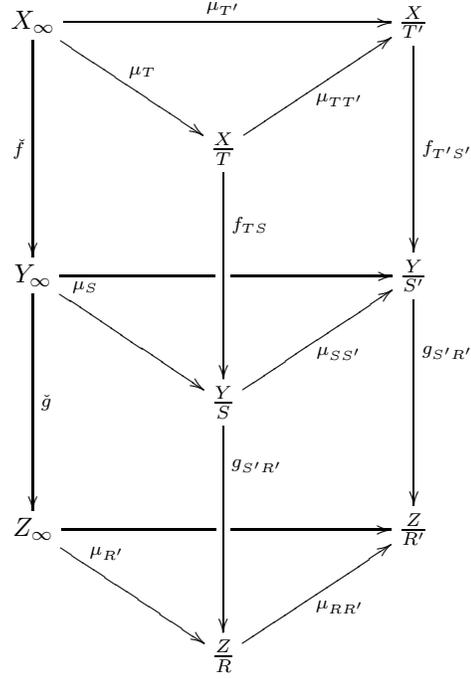
\xymatrix @!0 @R=4pc @C=6pc {
    X_{\infty} \ar[rr]^{\mu_{T'}}\ar[rd]^{\mu_{T}} \ar[dd]_{\check{f}} && \frac{X}{T'} \ar[dd]^{f_{T'S'}} \\
    & \frac{X}{T} \ar[ru]_{\mu_{TT'}} \ar[dd]^(.3){f_{TS}} \\
    {{Y}_{\infty}} \ar[rr] |!{[ur];[dr]}\hole \ar[rd]^(.2){\mu_{S}} \ar[dd]^{\check{g}}&& \frac{Y}{S'} \ar[dd]^(.3){g_{S'R'}} \\
    & \frac{Y}{S} \ar[ru]_{\mu_{SS'}} \ar[dd]^(.25){g_{S'R'}} & \\
       Z_{\infty} \ar[rr]|!{[ur];[dr]}\hole \ar[rd]^(.3){\mu_{R'}} & & \frac{Z}{R'}& \\
     & \frac{Z}{R} \ar[ru]_{\mu_{RR'}} & }$$

We have, then, the following commutative diagram:

$$\xymatrixcolsep{3pc}
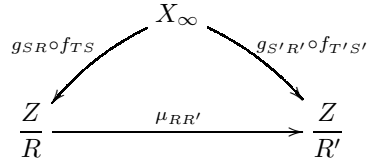
\xymatrix{
 & X_{\infty} \ar@/_/[dl]_{g_{SR}\circ f_{TS}} \ar@/^/[dr]^{g_{S'R'} \circ f_{T'S'}}\\
\dfrac{Z}{R} \ar[rr]^{\mu_{RR'}} & & \dfrac{Z}{R'}}$$

and by the universal property of $Z_{\infty}$ there is a unique $\check{g \circ f}: X_{\infty} \to Z_{\infty}$ such that the following diagram commutes:

$$\xymatrixcolsep{3pc}
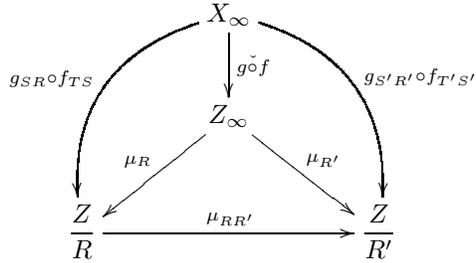
\xymatrix{
 & X_{\infty} \ar[d]^{\check{g \circ f}} \ar@/_2pc/[ddl]_{g_{SR}\circ f_{TS}} \ar@/^2pc/[ddr]^{g_{S'R'} \circ f_{T'S'}}\\
 & Z_{\infty} \ar[dl]_{\mu_{R}} \ar[dr]^{\mu_{R'}} & \\
\dfrac{Z}{R} \ar[rr]^{\mu_{RR'}} & & \dfrac{Z}{R'}}$$

Since the following diagram commutes:

$$
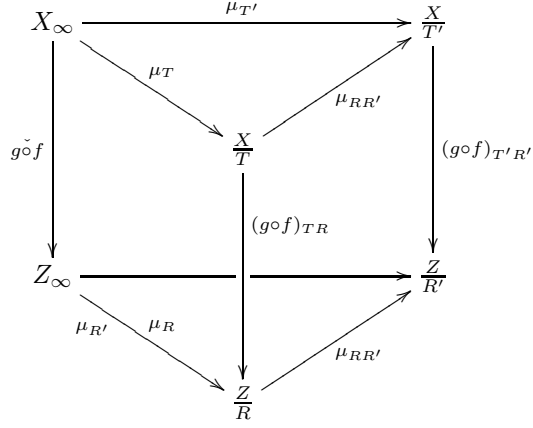
\xymatrix @!0 @R=4pc @C=6pc {
    X_{\infty} \ar[rr]^{\mu_{T'}} \ar[rd]^{\mu_{T}} \ar[dd]_{\check{g \circ f}} && \frac{X}{T'} \ar[dd]^{(g \circ f)_{T'R'}} \\
    & \frac{X}{T} \ar[ru]_{\mu_{RR'}} \ar[dd]^(.3){(g \circ f)_{TR}} \\
    {{Z}_{\infty}} \ar[rr] |!{[ur];[dr]}\hole \ar[rd]^{\mu_{R}}_(.3){\mu_{R'}} && \frac{Z}{R'} \\
    & \frac{Z}{R} \ar[ru]_{\mu_{RR'}} }$$

by the uniqueness of $\check{g \circ f}$, we have $\check{g \circ f} = \check{g} \circ \check{f}$, hence

$$\Delta(g \circ f) = \check{g \circ f} = \check{g} \circ \check{f} = \Delta(g) \circ \Delta(f).
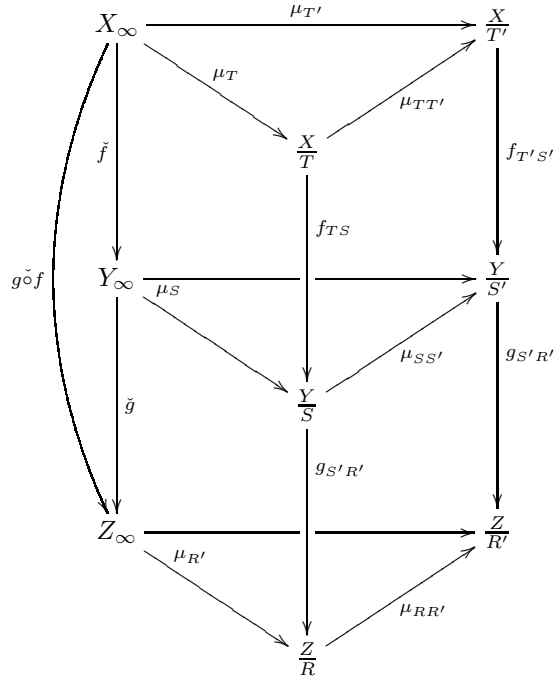
$$

$$\xymatrix @!0 @R=4pc @C=6pc {
    X_{\infty} \ar[rr]^{\mu_{T'}} \ar@/_2pc/[dddd]_{\check{g \circ f}} \ar[rd]^{\mu_{T}} \ar[dd]_{\check{f}} && \frac{X}{T'} \ar[dd]^{f_{T'S'}} \\
    & \frac{X}{T} \ar[ru]_{\mu_{TT'}} \ar[dd]^(.3){f_{TS}} \\
    {{Y}_{\infty}} \ar[rr] |!{[ur];[dr]}\hole \ar[rd]^(.2){\mu_{S}} \ar[dd]^{\check{g}}&& \frac{Y}{S'} \ar[dd]^(.3){g_{S'R'}} \\
    & \frac{Y}{S} \ar[ru]_{\mu_{SS'}} \ar[dd]^(.25){g_{S'R'}} & \\
       Z_{\infty} \ar[rr]|!{[ur];[dr]}\hole \ar[rd]^(.3){\mu_{R'}} & & \frac{Z}{R'}& \\
     & \frac{Z}{R} \ar[ru]_{\mu_{RR'}} & }$$

The above claims prove that

$$\begin{array}{cccc}
    \Delta: & {\rm \bf BoolSp} & \rightarrow & {\rm \bf BoolSp} \\
     & X & \mapsto & \delta_X[X] = \varprojlim_{R \in \mathcal{R}_X}\dfrac{X}{R}\\
     & \xymatrix{X \ar[r]^{f}& Y} & \mapsto & \xymatrix{X_{\infty} \ar[r]^{\check{f}} & Y_{\infty}}
  \end{array}$$

is a functor.\\

%%%%%%%%%%%%%%%%%%%%%%%%%%%%%%%%%%%%%%%%%%

\begin{theorem} \label{boospvn} Let $\mathbb{K}$ be a $\mathcal{C}^{\infty}-$field.
Following the notation of \textbf{Lemma \ref{niedia}}, define the
contravariant functor:

$$\begin{array}{cccc}
\widehat{k}: & {\rm \bf BoolSp} & \rightarrow & \mathcal{C}^{\infty}{\rm \bf vNRng} \\
& (X,\tau) & \mapsto & R_X:= \varinjlim_{R \in \mathcal{R}}
\mathbb{K}^{U\left( \frac{X}{R}\right)}
\end{array}$$

Then there is a natural isomorphism
$$\epsilon : Id_{\rm \bf BoolSp} \overset{\cong}{\to} {\rm Spec}^{\infty}
\circ \widehat{k}.$$

Thus the functor ${\rm Spec}^{\infty}: \mathcal{C}^{\infty}{\rm \bf vNRng} \rightarrow {\rm \bf BoolSp}$ is full and
\underline{isomorphism-dense}.\\

In particular: for each $(X,\tau)$ be a Boolean space, there is a von Neumann-regular $\mathcal{C}^{\infty}-$ring, $R_X$, such that:
$${\rm Spec}^{\infty}\,(R_X) \approx X.$$
.
\end{theorem}
\begin{proof}
%By \textbf{Lemma \ref{tadeu}},
%$${\rm Spec}^{\infty}\,(R_X) = \varprojlim_{R \in \mathcal{R}} {\rm
%Spec}^{\infty}\, (\mathbb{K}^{U\left( \frac{X}{R}\right)}).$$

By the \textbf{Theorem 34}, p. 118 of \cite{BM2},

$${\rm Spec}^{\infty}(R_X) \approx \varprojlim_{R \in \mathcal{R}}
{\rm Spec}^{\infty}\,(\mathbb{K}^{U\left(\frac{X}{R}\right)}).$$

By the \textbf{Theorem 33}, p. 118 of \cite{BM2},

$${\rm Spec}^{\infty}\,(\mathbb{K}^{U\left(\frac{X}{R}\right)})
\approx {\rm Discr}.\,\left( \frac{X}{R}\right),$$
so

$$\varprojlim_{R \in \mathcal{R}} {\rm
Spec}^{\infty}\,(\mathbb{K}^{U\left(\frac{X}{R}\right)}) \approx
\varprojlim_{R \in \mathcal{R}} {\rm Discr.}\, \left(
\frac{X}{R}\right) \approx \varprojlim_{R \in \mathcal{R}} \frac{X}{R}
\approx X$$

Since the homeomorphisms above are natural, just take $\epsilon_X : X
\to {\rm Spec}^\infty(R_X)$ as the composition of the these homeomorphisms.

Hence, in particular, ${\rm Spec}^{\infty}\,(R_X) \approx X$.
\end{proof}

\begin{theorem}\label{qmq} Let $\mathbb{K}$ be a $\mathcal{C}^{\infty}-$field.
Define the covariant functor (composition of contravariant functors):

$$\check{K} = \widehat{k}\circ {\rm Stone} : {\rm \bf Bool} \to
\mathcal{C}^{\infty}{\rm \bf vNRng}.$$

Then there is a natural isomorphism
$$\theta : {\rm Id}_{\rm \bf Bool} \overset{\cong}\to\to \widetilde{B}
\circ \check{K}.$$

Thus the functor $\widetilde{B} : \mathcal{C}^{\infty}{\rm \bf vNRng}
\rightarrow {\rm \bf Bool}$ is full and
\underline{isomorphism-dense}.

In particular: given any $\mathcal{C}^{\infty}-$field $\mathbb{K}$ and any
Boolean algebra $B$, there is a von
Neumann regular $\mathcal{C}^{\infty}-$ring which is a $\mathbb{K}-$algebra,
$\check{K}(B))$, such that
$$\widetilde{B}(\check{K}(B)) \cong B$$.
\end{theorem}
\begin{proof} This follows directly consequence by a combination of
the \textbf{Theorem \ref{boospvn}} above,
Stone duality (\textbf{Remark \ref{stone-dual}}), \textbf{Theorem
\ref{uhum}} and \textbf{Theorem \ref{jordao}}.
\end{proof}

The diagram below summarizes the main functorial connections
stablished in this section:

$$\xymatrixcolsep{3pc}\xymatrix{
& & \ar@2[d]^{\varepsilon}& & & & \\
{\rm \bf BoolSp}\ar@/^3pc/[rrrrr]^{{\rm Id}_{\rm \bf BoolSp}}
\ar[rr]^{\widehat{k}} & & \mathcal{C}^{\infty}{\rm \bf vNRng}
 \ar[rrr]^{{\rm Spec}^{\infty}}
\ar[ddrrr]_{\widetilde{B}} & & & {\rm \bf BoolSp} \ar@<1ex>[dd]^{{\rm
Clopen}}\\
& & & & \ar@2[ur]^{\jmath} & \\
{\rm \bf Bool} \ar[uu]^{{\rm Stone}} \ar[uurr]_{\check{K}}
\ar@/_/[rrrrr]_{{\rm Id}_{{\rm \bf Bool}}} &  & \ar@2[uu]^{\theta} & & & {\rm \bf Bool}}$$

\end{document}